\documentclass[11pt]{amsart}
\usepackage{indentfirst}
\usepackage{graphics,amsmath,amssymb,amsthm,mathrsfs,bm}
\usepackage{enumitem}

\usepackage{amsfonts, bm}
\usepackage{amsmath}
\usepackage{amssymb}
\usepackage{multicol}
\usepackage{stmaryrd}
\usepackage{mathrsfs}
\usepackage{bbm}
\usepackage{cite}
\usepackage{epsfig}
\usepackage{color}
\usepackage{graphics}
\usepackage{graphicx}
\usepackage{epstopdf}
\usepackage{multicol,graphics}
\usepackage{xpatch}
\usepackage{hyperref}
\makeatletter
\xpatchcmd{\@thm}{\thm@headpunct{.}}{\thm@headpunct{}}{}{}
\makeatother
\newcommand\bes{\begin{eqnarray}}
\newcommand\ees{\end{eqnarray}}
\newcommand\R{\mathbb R}
\newcommand{\ggs}{{\succ\!\!\succ}}
\newcommand{\llp}{{\prec\!\!\prec}}
\newtheorem{theorem}{Theorem}[section]
\newtheorem{lemma}[theorem]{Lemma}

\newtheorem{remark}[theorem]{Remark}

\numberwithin{equation}{section}
\allowdisplaybreaks

\begin{document}

\title[A nonlocal epidemic model with free boundaries ]{Long-time dynamics of a nonlocal epidemic model with  free boundaries:  spreading-vanishing dichotomy}

\author[R. Wang and Y. Du]{Rong Wang$^*$ and Yihong Du$^*$}
\thanks{$^*$School of Science and Technology, University of New England, Armidale, NSW 2351, Australia}
\thanks{Emails: ydu@une.edu.au (Y. Du), rwang3@myune.edu.au (R. Wang) }
\thanks{This research was supported by the Australian Research Coucile.}
\date{\today}

\begin{abstract}
In this paper, we examine the long-time dynamics of an epidemic model whose diffusion and reaction terms involve nonlocal effects described by suitable convolution operators.
The spreading front of the disease is represented by the free boundaries in the model. We show that the model is well-posed,  its long-time dynamical behaviour is characterised by a spreading-vanishing dichotomy, and we also obtain sharp criteria to determine the dichotomy. Some of the nonlocal effects in the model pose extra difficulties in the mathematical treatment, which are dealt with by introducing new approaches. The model can capture accelerated spreading, and its spreading rate will be discussed in a subsequent work.

\bigskip

\medskip


\end{abstract}
\maketitle
\section{Introduction}

In this paper, we examine the long-time dynamical behaviour of an epidemic model with nonlocal diffusion and free boundaries, which is a refined version of several previous epidemic models based on the following ODE system proposed by
Capasso and Paveri-fontana \cite{Capasso1979},
for the spread of
oro-faecal transmitted disease such as cholera,
\begin{equation}\label{model*ODE2}
u'=-a_{11}u+a_{12}v,\
v'=-a_{22}v+G(u),\ t>0.
\end{equation}
In \eqref{model*ODE2}, $u$ denotes the average density of the infective agents (bacteria,
virus, etc.), while $v$ represents the average density of the infective human population, both are functions of  time $t$. The constant
 $a_{11}$ denotes the natural death rate of the agents,
 $a_{12}$ represents the growth rate of the agents contributed by the infective humans,
 and $a_{22}$ denotes the fatality rate of the infective human population. These rates  are all  positive.
 The function $G(u)$ represents the infective rate of humans, which is assumed to satisfy
\begin{enumerate}
\item[$(G1)$]: \ \ \ \ $G\in C^{1}([0,\infty)),\ G(0)=0,\ G^\prime (z)>0\ \mbox{ for } z\geq 0;$

\item[$(G2)$]: \ \ \ \ $\frac{G(z)}{z}$ is strictly decreasing and $\lim_{z\rightarrow +\infty}\frac{G(z)}{z} < \frac{a_{11}a_{22}}{a_{12}}$.

\end{enumerate}

The dynamics of \eqref{model*ODE2} is governed by the following threshold parameter, known as the basic reproduction number,
\begin{equation}\label{R_0}
R_0:=\frac{a_{12}G'(0)}{a_{11}a_{22}}.
\end{equation}
Specifically, if $R_0<1$, the unique solution $(u,v)$ with initial values $u(0),v(0)>0$ goes to $(0,0)$ as $t\to\infty$;
while if $R_0>1$, then \eqref{model*ODE2} admits a unique positive equilibrium $(u^*,v^*)$ and $(u,v)\to (u^*, v^*)$ as $t\to\infty$.
Clearly the pair $(u^*, v^*)$ is uniquely determined by (when $R_0>1$)
 \begin{equation}\label{u^*v^*}
 \frac{G(u^*)}{u^*}=\frac{a_{11}a_{22}}{a_{12}},\; v^*=\frac{a_{11}}{a_{12}}u^*.
 \end{equation}

Over the years, this basic model has been refined many times along different directions.
To include the effect of spatial movements of both the infectious agent and the infective human population in the epidemic,
 Capasso and Maddalena \cite{CapassoM1981}
considered the corresponding reaction-diffusion system to \eqref{model*ODE2} over a bounded spatial domain with suitable boundary conditions,
and they obtained similar (but slightly weaker) results on the
long-term dynamics of the reaction-diffusion model, by replacing $R_0$ with a number or numbers determined by certain eigenvalue problems. Their result was sharp when diffusion only appears in the equation for the infectious agents $u$, assuming  the mobility of the infective humans $v$ is relatively small and hence ignored in the model.

Many subsequent works also ignored the mobility of the infective humans. For example, in \cite{WuSunLiu} and \cite{ZhaoWang}, in order to describe the spatial spreading of the disease, traveling wave solutions were obtained for the corresponding system
\begin{equation}\label{TW}
\begin{cases}
u_t=d_1\Delta u-a_{11}u+a_{12}v&~~\text{in}~~\mathbb{R}\times (0,+\infty),\\
v_t=-a_{22}v+G(u)&~~\text{in}~~\mathbb{R}\times (0,+\infty).
\end{cases}
\end{equation}

To capture the evolution of the spreading front of the disease, Ahn et al \cite{Ahn2016} investigated a free boundary version of \eqref{TW}, which is a special case of the following system (namely, with
 $d_2=\rho=0$),
\begin{equation}\label{OE}
\begin{cases}
u_t=d_1 u_{xx}- a_{11}u+a_{12}v,&
g(t)< x <h(t),~~t>0, \\
v_t=d_2v_{xx}-a_{22}v+G(u),&
g(t)<x<h(t),~~t>0, \\
u=v=0,&x=h(t)~\mbox{or}~x=g(t),~t> 0,~~~~~~~~~~~~~~~~~~~~~\\
h^{\prime}(t)=-\mu [u_{x}(h(t),t)+\rho v_{x}(h(t),t)],&t> 0,~~~~~~~~~~~~~~~~~~~~~ \\
g^{\prime}(t)=-\mu[u_{x}(g(t),t)+\rho v_{x}(g(t),t)],&t> 0,~~~~~~~~~~~~~~~~~~~~~ \\
h(0)=h_0,~g(0)=-h_0,&\\
u(x,0)=u_0(x),~~v(x,0)=v_0(x),&-h_0\leq x\leq h_0,
\end{cases}
\end{equation}
where the epidemic region is represented by the evolving interval $[g(t), h(t)]$, with $x=g(t)$ and $x=h(t)$ giving the spreading fronts of the disease.

The general case of \eqref{OE} with $d_2> 0$ and $\rho\geq 0$ was considered by Wang and Du \cite{WangDu2021},
where  a rather complete description of the long-time dynamical behaviour of
\eqref{OE} was obtained, which in particular reveals the influence of $d_2$ in the model. More precisely, with $R_0$ given by \eqref{R_0}, the epidemic modelled by \eqref{OE} vanishes
if $R_0\leq 1$,  i.e.,  as $t\to\infty$, the epidemic region $[g(t),h(t)]$ converges to a finite interval
and $(u(x,t),v(x,t))$ goes to $(0,0)$ uniformly in $x$.  If $R_0>1$, a spreading-vanishing dichotomy holds, namely  either the epidemic vanishes as above,
or it spreads successfully in the sense that  $[g(t),h(t)]$ converges to $(-\infty, \infty)$
and $(u(x,t),v(x,t))\to (u^*,v^*)$ locally uniformly in $x$
as $t\to \infty$. Furthermore, a critical length $L^*$ independent of the initial data $(u_0,v_0, h_0)$ was found such
that,
if the range radius of the initial infected region $h_0\geq L^*$, then the spreading is always successful, while
 if $h_0<L^*$, then
there exists $\mu^*>0$ depending on $(u_0,v_0)$ such that vanishing occurs
when $0<\mu\leq\mu^*$ and spreading persists when $\mu>\mu^*$.
Moreover, when spreading persists,
 it has a finite asymptotic speed, i.e., $-\lim_{t\to\infty}g(t)/t=\lim_{t\to\infty}h(t)/t=c_0>0$, with $c_0$ determined by an associated semi-wave problem.

When $d_2=\rho=0$ in problem \eqref{OE}, the spreading-vanishing dichotomy was first obtained
by Ahn et al. \cite{Ahn2016},  and the spreading speed for this case was obtained by Zhao et al. \cite{zhaomeng2020}.

The effect of $d_2$ on the dynamics of \eqref{OE} can be partly seen from the formula for the critical number $L^*$, which is given by
\begin{equation*}
L^*:=\frac{\pi}2\sqrt{\frac{d_1a_{22}+d_2a_{11}+\sqrt{(d_1a_{22}+
d_2 a_{11})^2+4d_1d_2a_{11}a_{22}(R_0-1)}}{2a_{11}a_{22}(R_0-1)}}.
\end{equation*}
So $L^*$ is increasing in $d_2$, implying that the chance of successful spreading of the epidemic is decreased when $d_2$ is increased, which might appear counter intuitive on first sight, but it is in agreement with result already arising in the reaction-diffusion model over a bounded domain considered by Capasso and Maddalena \cite{CapassoM1981} mentioned above.

In \eqref{OE}, the spatial movements of infectious  agents and humans are assumed to follow the rule of Brownian motion, represented by the
``local diffusion" terms $d_1u_{xx}$ and $d_2v_{xx}$.
 However, in reality the movements of infectious agents and infective humans often contain nonlocal factors such as long-distance dispersal caused by  modern ways of transportation, for example.  To include such nonlocal dispersal factors, one widely used approach is to  replace the local diffusion terms $d_1 u_{xx}$ and $d_2 v_{xx}$ by, respectively,
\begin{equation*}\begin{aligned}
d_1\int_{\mathbb{R}}J_1(x-y)u(y,t)dy-d_1u(x,t)\mbox{ and }d_2\int_{\mathbb{R}}J_2(x-y)v(y,t)dy-d_2v(x,t),
\end{aligned}\end{equation*}
where $J_i(x)$ $(i=1,2)$ is a continuous function satisfying
\begin{enumerate}
\item[$\textbf{(J)}$] $J_i\in C(\R)\cap L^{\infty}(\R)$, $J_i$ is symmetric and nonnegative,
 $J_i(0)>0$, $\int_{\R}J_i(x)dx=1$.
\end{enumerate}
Roughly speaking, the nonlocal diffusion term $\int_{\mathbb{R}}J_1(x-y)u(y,t)dy-u(x,t)$ is obtained from the assumption that
 an individual of $u$ at location $y$ and time $t$ has probability $J_1(x-y)$ to move to  location $x$ in a time unit.

When the local diffusion terms in \eqref{OE} are replaced by the above nonlocal diffusion terms, as in  Cao et al. \cite{CDLL2019}, the free boundary conditions need to be modified accordingly, and the local diffusion model \eqref{OE} becomes
\begin{equation}\label{model*}
\begin{cases}\displaystyle
u_t=d_1\int_{g(t)}^{h(t)}J_1(x-y)u(y,t)dy-d_1u-a_{11}u&\\
\displaystyle
\hspace{2.7cm}+a_{12}\int_{g(t)}^{h(t)}K(x-y)v(y,t)dy,&t>0,~~x\in(g(t),h(t)),\\
\displaystyle
v_t=d_2\int_{g(t)}^{h(t)}J_2(x-y)v(y,t)dy-d_2v-a_{22}v+G(u),&t>0,~~x\in(g(t),h(t)),\\
\displaystyle
g'(t)=\displaystyle-\mu \int_{g(t)}^{h(t)}\int_{-\infty}^{g(t)}J_1(x-y)u(x,t)dydx&\\
\displaystyle
\hspace{2.8cm}-\mu\rho \int_{g(t)}^{h(t)}\int_{-\infty}^{g(t)}J_2(x-y)v(x,t)dydx,&t>0,\\
\displaystyle
h'(t)=\mu \int_{g(t)}^{h(t)}\int^{\infty}_{h(t)}J_1(x-y)u(x,t)dydx&\\
\displaystyle
\hspace{2.8cm}+\mu\rho \int_{g(t)}^{h(t)}\int^{\infty}_{h(t)}J_2(x-y)v(x,t)dydx,&t>0,\\
u(h(t),t)=u(g(t),t)=v(h(t),t)=v(g(t),t)=0,&t>0,\\
h(0)=-g(0)=h_0,\
u(x,0)=u_0(x),~~~v(x,0)=v_0(x),&x\in[-h_0,h_0].
\end{cases}
\end{equation}
Here, following Capasso \cite{Capasso1984}, we have also replaced the term $c_{12}v(x,t)$ in \eqref{OE} by
 \begin{equation}\label{nonlocal}
 \displaystyle a_{12}\int_{g(t)}^{h(t)}K(x-y)v(y,t)dy
 \end{equation}
  in \eqref{model*}, to represent the more realistic and consistant assumption that
the contribution to the growth of the infectious agents at location $x$ by the
infective humans is also nonlocal. Indeed, in most societies, the human wastes, which are the main sources of the infectious agents in this model,
are moved away immediately from the exact location of the infective humans. Therefore the contribution to $u$ at location $x$ is usually from infective humans in a neighbourhood of $x$, which is better represented by a nonlocal operator.  If $K(x)$ is equal to the Dirac delta function $\delta_0(x)$, then
\[
\mbox{$\displaystyle a_{12}\int_{g(t)}^{h(t)}K(x-y)v(y,t)dy=a_{12}v(x,t)$ for $x\in (g(t), h(t))$;}
\]
 we will assume that $K$ satisfies {\bf (J)} and so the term \eqref{nonlocal} represents a true nonlocal effect.

 The initial functions
$u_0(x)$ and $v_0(x)$ in \eqref{model*} are assumed to
satisfy
\begin{equation}\label{Assumption}
\left\{\begin{aligned}
& u_0\in C[-h_0,h_0],~~u_0(\pm h_0)=0~~\text{and}~~u_0(x)>0~~\mbox{for}~~x\in (-h_0,h_0),\\
& v_0\in C[-h_0,h_0],~~v_0(\pm h_0)=0~~\text{and}~~v_0(x)>0~~\mbox{for}~~x\in (-h_0,h_0).
\end{aligned}\right.
\end{equation}

 The main purpose of this paper is to obtain a good understanding of the long-time dynamics of \eqref{model*}. Several special cases  or closely related cases of \eqref{model*} have already been considered in the recent literature. When $d_2=\rho=0$ and $K(x)=\delta_0(x)$ so the
term $\displaystyle a_{12}\int_{g(t)}^{h(t)}K(x-y)v(y,t)dy$ is changed back to $a_{12}v(x,t)$,   problem \eqref{model*} becomes a precise nonlocal diffusion version of \eqref{OE}
with $d_2=0$; for this special case,
Zhao et al. \cite{zhaomeng2020JDE} proved that the problem is well posed mathematically and the long-time dynamics of the model is governed by a spreading-vanishing dichotomy, similar to its local diffusion version, but differences arise in the criteria distinguishing the alternatives in the dichotomy.
More striking differences appear when results of  Du and Ni in \cite{dn2022} on the rate of spreading were applied to this case of the model.
It follows from \cite{dn2022} that when $J_1$  satisfies an additional threshold condition, then when spreading persisits,
$-\lim_{t\to\infty}g(t)/t=\lim_{t\to\infty}h(t)/t$ exists and is a finite number determined by an associated semi-wave problem,
while $-\lim_{t\to\infty}g(t)/t=\lim_{t\to\infty}h(t)/t=\infty$ when this threshold condition is not satisfied by $J_1$ (note that since $d_2=\rho=0$, $J_2$ does not appear in this case of the model). This latter case is known as ``accelerated spreading", and was first shown to occur for a nonlocal free boundary problem for the Fisher-KPP model in \cite{dlz2021}.

These results were extended by Zhao et al. \cite{zhaomeng2020CPAA} and Du et al. \cite{dlnz} to the case of \eqref{model*} with $d_2=\rho=0$, and by Chang and Du \cite{cd2022} for \eqref{model*} with
$d_2>0, \rho\geq 0$ and the
term $\displaystyle a_{12}\int_{g(t)}^{h(t)}K(x-y)v(y,t)dy$  replaced by $a_{12}v(x,t)$. Moreover, the rate of accelerated spreading was given in \cite{cd2022}
when $J_i(x)\sim |x|^{-\gamma}$ for $|x|\gg 1$ with $\gamma\in (1, 2]$ (see Theorem 1.6 there for details, which follow from more general results due to Du and Ni).

A West Nile virus model with nonlocal diffusion and free boundaries was studied by Du and Ni \cite{DuNi2020N}, which share some similarities to \eqref{model*}
in  that it is also  a cooperative evolution system, and its spreading speed can be treated by the same genera approach in \cite{dn2022}. The  nonlocal
term $\displaystyle a_{12}\int_{g(t)}^{h(t)}K(x-y)v(y,t)dy$  in \eqref{model*} renders it outside the scope of \cite{dn2022}, and moreover, it causes several technical difficulties which require very different  treatment from the earlier works.

There is a large literature on related
 nonlocal diffusion problems over the entire spatial space $\mathbb R$ or $\mathbb R^N$, without involving any free boundaries; a small sample of these works
 can be found in  \cite{bao2016,bao2017,Bates1997,Bates2007,Berestycki2016JMB,Berestycki2016JFA,Coville2007,Garnier2011,Hutson2003JMB,Kao2010
,FangLi2017DCDS,WLi2010,Rawal2012,Yagisita2009} and the references therein.
\medskip

Let us now describe the main results for \eqref{model*} precisely.

\begin{theorem}[Global existence and uniqueness]\label{uniexist}
Suppose that $J_1(x),J_2(x)$ and $K(x)$ satisfy assumption {\bf(J)}, $G$ satisfies $(G1)$-$(G2)$,
and $(u_0(x),v_0(x))$ satisfies \eqref{Assumption}. Then problem \eqref{model*} admits a unique solution defined
for all $t>0$.
\end{theorem}

\begin{theorem}[Spreading-vanishing dichotomy]\label{svd}
Under the conditions of Theorem \ref{uniexist}, suppose that $(u,v,g,h)$ is the solution of \eqref{model*}; then
  \begin{equation*}
   h_{\infty}:=\lim_{t\rightarrow \infty}h(t)\in (h_0, \infty]~\mbox{and}~g_{\infty}:=\lim_{t\rightarrow \infty}g(t)\in[-\infty, -h_0) \mbox{ exist},
\end{equation*}
and one of the following cases must occur:
\begin{itemize}
  \item [{\rm(i)}] \textbf{Vanishing}:
\begin{equation*}
  h_{\infty}-g_{\infty}<\infty ~~\text{and}~~ \lim_{t\rightarrow \infty}(u(x,t),v(x,t))=(0,0)~~\mbox{uniformly for}~~x\in[g(t),h(t)];
\end{equation*}
  \item [{\rm(ii)}]\textbf{Spreading}:
\begin{equation*}
   h_{\infty}=-g_{\infty}=\infty~~ \text{and}~~ \lim_{t\rightarrow \infty}(u(x,t),v(x,t))=(u^*,v^*) ~~\mbox{locally uniformly in}~~\R.
\end{equation*}
\end{itemize}
\end{theorem}

\begin{theorem}[Spreading-vanishing criteria]\label{spreading-vanishing}
Let  $R_0$ be given by \eqref{R_0}. Then the alternatives in Theorem \ref{svd} are determined as follows:
\begin{itemize}
  \item [{\rm(i)}] If $R_0\leq 1$, then vanishing always occurs.
  \item [{\rm(ii)}]If $R_0>1$, then there exists a critical length $l^*>0$ independent of the initial data
   $(u_0,v_0, h_0)$ such that
  if $h_0\geq l^*$, then spreading always persists, while if $h_0<l^*$, then
 there exists $\mu^*>0$ depending on $(u_0,v_0)$ such that vanishing happens
when $\mu\in(0,\mu^*]$ and spreading persists when $\mu\in(\mu^*,\infty)$.
\end{itemize}
\end{theorem}

When spreading persists, the spreading speed will be determined in a subsequent work. In particular we will find threshold conditions on $J_1$ and $J_2$
which determine whether the speed is finite or infinite.

The rest of the paper is organised as follows. In Section 2, we give some comparison principles to be used throughout the paper; the proofs of some of them depend on the existence uniqueness theorem (Theorem 1.1), whose proof is postponed to the last section of the paper, Section 5, since it is only a relatively simple modification of existing proofs in the literature. Section 3 treats the corresponding fixed boundary problem and the associated eigenvalue problem. Here the nonlocal term $\int_{-l}^l K(x-y)v(y,t)dy$ induces nontrivial technical difficulties, especially for the eigenvalue problem. A new approach is used in this section to obtain the desired results, which we believe should have applications elsewhere. In Section 4, we make use of the results in Section 3 to prove the spreading-vanishing dichotomy and obtain the sharp criteria
governing this dichotomy. As already mentioned, Section 5 is devoted to the proof of Theorem 1.1.

\section{Some basic results}

In this section, we collect some comparison results, which form the basis of this research. For convenience of later discussions, we introduce some notations first.
For any given $\mathbf{u}=(u_1,v_1),\mathbf{v}=(v_1,v_2)\in \mathbb{R}^2$, we denote
\[
\mathbf{u}\preceq(\succeq)\mathbf{v}\mbox{ if }u_i\leq (\geq) v_i,~i=1,2;\ \
\mathbf{u}\llp (\ggs) \mathbf{v}\mbox{ if }u_i<(>)v_i,~i=1,2.
\]
For given $T,\, h_0>0$ and $(u_0,v_0)$ satisfying \eqref{Assumption}, we define
\begin{equation}\label{definition_original}
\begin{split}
&H^T:=\Big\{h\in C([0,T]):~~h(0)=h_0,~~\inf_{0\leq t_1<t_2\leq T}\frac{h(t_2)-h(t_1)}{t_2-t_1}>0\Big\},\\
&G^T:=\Big\{g\in C([0,T]):~~g(0)=-h_0,~~\sup_{0\leq t_1<t_2\leq T}\frac{g(t_2)-g(t_1)}{t_2-t_1}<0\Big\}.\\
\end{split}
\end{equation}
For $g\in G^T$ and $h\in H^T$, we write
\begin{equation*}
\begin{split}
\Omega_T&=\Omega_T(g,h):=\Big\{(t,x)\in\R^2:~0<t\leq T,~g(t)<x<h(t)\Big\},\\
X^T&=X^T(g,h, u_0,v_0)\\
&:=\Big\{(\phi,\psi)\in (C(\overline{\Omega}_T))^2:~~(\phi(\cdot,0),\psi(\cdot,0))=(u_0,v_0) \mbox{ in }[-h_0,h_0],\ (\phi,\psi)\geq0\mbox{ in }\Omega_T,\\
 & \hspace{4.5cm} \phi(x,t)=\psi(x,t)=0 \mbox{ for  } x\in\{g(t), h(t)\} \mbox{ and } t\in [0,T]\Big\}.\\
\end{split}
\end{equation*}

\begin{lemma}\label{maximalprinciple}
Suppose that  {\bf (J)} is satisfied by $J_1, J_2$ and $K$. Let $T>0$,  $(g,h)\in G^T\times H^T$, $b_{ij}\in L^{\infty}(\overline{\Omega}_T)$ for
$i,j=1,2$, and $(u,v),~(u_t,v_t)\in C(\overline{\Omega}_T)$. If
\begin{equation*}\label{model1}
\begin{cases}\displaystyle
u_t\geq d_1\int_{g(t)}^{h(t)}J_1(x-y)u(y,t)dy-d_1u+b_{11}u&\\
\displaystyle
\hspace{2.7cm}+b_{12}\int_{g(t)}^{h(t)}K(x-y)v(y,t)dy,&0<t\leq T,~~x\in(g(t),h(t)),\\
\displaystyle
v_t\geq d_2\int_{g(t)}^{h(t)}J_2(x-y)v(y,t)dy-d_2v+b_{21}u+b_{22}v,&0<t\leq T,~~x\in(g(t),h(t)),\\
u(x,t)\geq 0,~~v(x,t)\geq0,&0<t\leq T,~~x=h(t)\mbox{ or }g(t),\\
u(x,0)\geq 0,~~~v(x,0)\geq0,&x\in[-h_0,h_0],
\end{cases}
\end{equation*}
and $b_{12}\geq 0$, $b_{21}\geq 0$, then $(u(x,t),v(x,t))\succeq(0,0)$ in $\overline{\Omega}_T$.
Moreover, if additionally $u(x,0)\not\equiv 0$
and $v(x,0)\not\equiv 0$ in $[-h_0,h_0]$, then $(u(x,t),v(x,t))\ggs(0,0)$ in $\Omega_T$.
\end{lemma}
\begin{proof}
Denote
\begin{equation*}
f_1(x,t)=e^{kt}u(x,t)\mbox{ and }f_2(x,t)=e^{kt}v(x,t)\mbox{ for }(x,t)\in\overline{\Omega}_T
\end{equation*}
with $k$ being some large constant satisfying $k>d_1+d_2+\|b_{11}\|_{\infty}+\|b_{22}\|_{\infty}$.
Simple calculations show that
\begin{equation}\label{model2}
\begin{cases}\displaystyle
(f_1)_t\geq d_1\int_{g(t)}^{h(t)}J_1(x-y)f_1(y,t)dy+(k-d_1+b_{11})f_1&\\
\displaystyle
\hspace{4cm}+b_{12}\int_{g(t)}^{h(t)}K(x-y)f_2(y,t)dy,&0<t\leq T,~~x\in(g(t),h(t)),\\
\displaystyle
(f_2)_t\geq d_2\!\!\int_{g(t)}^{h(t)}\!\!J_2(x\!-\!y)f_2(y,t)dy\!+\!(k\!-\!d_2\!+\!b_{22})f_2+b_{21}f_1,&0<t\leq T,~~x\in(g(t),h(t)),\\
f_1(x,t)\geq 0,~~f_2(x,t)\geq0,&0<t\leq T,~~x\in \{g(t),h(t)\},\\
f_1(x,0)\geq 0,~~~f_2(x,0)\geq0,&x\in[-h_0,h_0].
\end{cases}
\end{equation}

Denote
\begin{equation*}
T_0:=\min\Big\{T,\frac{1}{4(k+\sum_{i,j=1}^2\|b_{ij}\|_{\infty})}\Big\}.
\end{equation*}
We claim that
\begin{equation*}
f_m:=\min\bigg\{\inf_{(x,t)\in\overline{\Omega}_{T_0}}f_1(x,t),~~~\inf_{(x,t)\in\overline{\Omega}_{T_0} }f_2(x,t)\bigg\}\geq0.
\end{equation*}
Otherwise, $f_m<0$. By the continuity of $f_i(x,t)$ $(i=1,2)$ and the last two equations in \eqref{model2}, there exists some $(x_0,t_0)\in{\Omega}_{T_0}$ such that
\begin{equation*}
\frac{f_m}{2}=f_1(x_0,t_0)<0\mbox{ or }\frac{f_m}{2}=f_2(x_0,t_0)<0.
\end{equation*}
Define
\begin{equation*}
t_1=t_1(x_0):=\begin{cases}
t_{x_0}^h, &\mbox{ if } x_0\in(h_0,h(t_0))\mbox{ and } x_0=h(t_{x_0}^h),\\
0,&\mbox{ if } x_0\in[-h_0,h_0],\\
t_{x_0}^g,&\mbox{ if } x_0\in(g(t_0),-h_0)\mbox{ and } x_0=g(t_{x_0}^g).\\
\end{cases}
\end{equation*}
Obviously, $f_i(x_0,t_1)\geq 0$ for $i=1,2$.

If $\frac{f_m}{2}=f_1(x_0,t_0)<0$, then
\begin{equation*}
\begin{split}
&f_1(x_0,t_0)-f_1(x_0,t_1)=\int_{t_1}^{t_0}(f_1)_t(x_0,t)dt\\
\geq&\ d_1\int_{t_1}^{t_0}\int_{g(t)}^{h(t)}J_1(x_0-y)f_1(y,t)dydt+\int_{t_1}^{t_0}\int_{g(t)}^{h(t)}b_{12}(x,t)K(x_0-y)f_2(y,t)dydt\\
&+\int_{t_1}^{t_0}(k-d_1+b_{11}(x,t))f_1(x_0,t)dt\\
\geq&\ d_1\int_{t_1}^{t_0}\int_{g(t)}^{h(t)}J_1(x_0-y)f_mdydt+\int_{t_1}^{t_0}\int_{g(t)}^{h(t)}b_{12}(x,t)K(x_0-y)f_mdydt\\
&+\int_{t_1}^{t_0}(k-d_1+b_{11}(x,t))f_mdt\\
\geq&\ \big(k+\|b_{11}\|_{\infty}+\|b_{12}\|_{\infty}\big)f_m(t_0-t_1).
\end{split}
\end{equation*}
Therefore
\begin{equation*}
\frac{f_m}{2}\geq f_1(x_0,t_0)-f_1(x_0,t_1)\geq\big(k+\|b_{11}\|_{\infty}+\|b_{12}\|_{\infty}\big
)f_mT_0\geq \frac{f_m}{4},
\end{equation*}
which contradicts $f_m<0$.

If $\frac{f_m}{2}=f_2(x_0,t_0)<0$, then similarly
\begin{equation*}
\begin{split}
&f_2(x_0,t_0)-f_2(x_0,t_1)=\int_{t_1}^{t_0}(f_2)_t(x_0,t)dt\\
\geq&\ d_2\int_{t_1}^{t_0}\int_{g(t)}^{h(t)}J_2(x_0-y)f_2(y,t)dydt+\int_{t_1}^{t_0}(k-d_2+b_{22}(x,t))f_2(x_0,t)dt\\
&+\int_{t_1}^{t_0}b_{21}(x,t)f_1(x_0,t)dt\\
\geq&\ d_2\int_{t_1}^{t_0}\int_{g(t)}^{h(t)}J_2(x_0-y)f_mdydt+\int_{t_1}^{t_0}(k-d_2+b_{22}(x,t)+b_{21}(x,t))f_mdt\\
\geq&\ \big(k+\|b_{22}\|_{\infty}+\|b_{21}\|_{\infty}\big)f_m(t_0-t_1).
\end{split}
\end{equation*}
So we have
\begin{equation*}
\frac{f_m}{2}\geq f_2(x_0,t_0)-f_2(x_0,t_1)\geq\big(k+\|b_{22}\|_{\infty}+\|b_{21}\|_{\infty}\big)f_mT_0\geq \frac{f_m}{4},
\end{equation*}
which also contradicts $f_m<0$.
The claim is thus proved, and so $f_i(x,t)\geq 0$ in $\overline{\Omega}_{T_0}$ for $i=1,2$.

If $T_0=T$,
then it is clear that $(u(x,t),v(x,t))\succeq(0,0)$ for all $(x,t)\in\overline{\Omega}_T$. If $T_0<T$, by $(u(x,T_0),v(x,T_0))\succeq(0,0)$ in
$[g(T_0),h(T_0)]$, we can replace $(u(x,0),v(x,0))$ with $(u(x,T_0),v(x,T_0))$, and $(0,T_0]$ with $(T_0,T]$, and  then repeat the above process.
 After repeating the process finitely many times, we obtain $(u(x,t),v(x,t))\succeq(0,0)$ for all $(x,t)\in\overline{\Omega}_T$.

 Finally we show that $u(x,t)>0$ in $\Omega_{T}$ if $u(x,0)\not\equiv0$ in $[-h_0,h_0]$. Otherwise, there exists some $(\tilde{x},\tilde{t})\in\Omega_{T}$ such that $u(\tilde{x},\tilde{t})=0$. We claim that this leads to $u(x,\tilde{t})=0$ for $x\in(g(\tilde{t}),h(\tilde{t}))$. If not, then there exists
\begin{equation*}
x_1\in(g(\tilde{t}),h(\tilde{t}))\cap\partial\{x\in(g(\tilde{t}),h(\tilde{t})):~u(x,\tilde{t})>0\}
\end{equation*}
and thus $u(x_1,\tilde{t})=0$. It then follows from \eqref{model2} and the definition of $f_1$ that
\begin{equation*}
0\geq (f_1)_t(x_1,\tilde{t})\geq d_1\int_{g(\tilde{t})}^{h(\tilde{t})}J_1(x_1-y)f_1(y,\tilde{t})dy>0,
\end{equation*}
which is impossible. Hence, $f_1(x,\tilde{t})=0$ in $(g(\tilde{t}),h(\tilde{t}))$. In particular, for $x\in[-h_0,h_0]$,
\begin{equation*}
\begin{split}
&-f_1(x,0)=f_1(x,\tilde{t})-f_1(x,0)\\
\geq &\ d_1\int_{0}^{\tilde{t}}\int_{g(t)}^{h(t)}J_1(x-y)f_1(y,t)dydt+\int_{0}^{\tilde{t}}(k-d_1+b_{11}(x,t))f_1(x,t)dt\\
&+\int_{0}^{\tilde{t}}\int_{g(t)}^{h(t)}a_{12}(x,t)K(x-y)f_2(y,t)dydt\geq0,
\end{split}
\end{equation*}
which implies that $u(x,0)\equiv0$ in $[-h_0,h_0]$, a contradiction. The same argument can be applied to obtain $v(x,t)>0$ in $\Omega_{T}$ if $v(x,0)\not\equiv 0$.
The proof is complete.
\end{proof}
\begin{remark}\label{compariprinciple1} It is easily checked that the above proof remains valid if $g(t)\equiv -h_0$ and $h(t)\equiv h_0$. Therefore Lemma \ref{maximalprinciple} still holds for such $g$ and $h$.
\end{remark}

With the help of Lemma \ref{maximalprinciple}, we can follow the approach of \cite{CDLL2019} and \cite{DWZ} to prove Theorem 1.1. Since the proof is very long, and does not
require considerably new ideas, we postpone it to the end of the paper.

\medskip

Using the existence and uniqueness result, we can prove the following comparison principle.

\begin{lemma}\label{compariprinciple}
Assume that assumption {\bf (J)} is satisfied by $J_1$, $J_2$ and $K$, and $G$ satisfies $(G1)$-$(G2)$. Let $(u,v, g, h)$ be the unique solution of \eqref{model*}.
If $ T\in(0,+\infty)$,
$\bar{g},\bar{h}\in C([0,T])$, $\bar{u},\bar{v}, \bar u_t, \bar v_t\in C(\overline{\Omega_{T}(\bar{g},\bar{h})})$  satisfy $\bar{u},\bar{v}\geq 0$ and
\begin{equation}\label{model_upper}
\begin{cases}\displaystyle
\bar{u}_t\geq d_1\int_{\bar{g}(t)}^{\bar{h}(t)}J_1(x-y)\bar{u}(y,t)dy-d_1\bar{u}-a_{11}\bar{u}&\\
\displaystyle
\hspace{2.7cm}+a_{12}\int_{\bar{g}(t)}^{\bar{h}(t)}K(x-y)\bar{v}(y,t)dy,&0<t\leq T,~~x\in(\bar{g}(t),\bar{h}(t)),\\
\displaystyle
\bar{v}_t\geq d_2\int_{\bar{g}(t)}^{\bar{h}(t)}J_2(x-y)\bar{v}(y,t)dy-d_2\bar{v}-a_{22}\bar{v}+G(\bar{u}),&0<t\leq T,~~x\in(\bar{g}(t),\bar{h}(t)),\\
\displaystyle
\bar{g}'(t)\leq\displaystyle-\mu \int_{\bar{g}(t)}^{\bar{h}(t)}\int_{-\infty}^{\bar{g}(t)}J_1(x-y)\bar{u}(x,t)dydx&\\
\displaystyle
\hspace{3.2cm}-\mu\rho \int_{\bar{g}(t)}^{\bar{h}(t)}\int_{-\infty}^{\bar{g}(t)}J_2(x-y)\bar{v}(x,t)dydx,&0<t\leq T,\\
\displaystyle
\bar{h}'(t)\geq\mu  \int_{\bar{g}(t)}^{\bar{h}(t)}\int^{\infty}_{\bar{h}(t)}J_1(x-y)\bar{u}(x,t)dydx&\\
\displaystyle
\hspace{3.2cm}+\mu\rho \int_{\bar{g}(t)}^{\bar{h}(t)}\int^{\infty}_{\bar{h}(t)}J_2(x-y)\bar{v}(x,t)dydx,&0<t\leq T,\\
\bar{u}(x,t)\geq0,~~\bar{v}(x,t)\geq0,&0<t\leq T,~x=\bar{h}(t)\mbox{ or }\bar{g}(t),\\
\bar{g}(0)\leq -h_0,~~\bar{h}(0)\geq h_0,&\\
\bar{u}(x,0)\geq {u}_0(x),~~~\bar{v}(\bar{x},0)\geq v_0(x),&x\in[-h_0,h_0],
\end{cases}
\end{equation}
then
\begin{equation*}
\begin{cases}
h(t)\leq \bar{h}(t),~~g(t)\geq \bar{g}(t)~~~\text{for}~~t\in [0,T],\\
u(x,t)\leq \bar{u}(x,t),~~v(x,t)\leq \bar{v}(x,t) ~~~\text{for} ~~x\in [g(t),h(t)],~~t\in(0,T].
\end{cases}
\end{equation*}
\end{lemma}
\begin{proof}
By applying Lemma \ref{maximalprinciple},
we have $\bar{u}(x,t),\bar{v}(x,t)> 0$ for $t\in(0,T]$ and $x\in(\bar{g}(t),\bar{h}(t))$.
Therefore, $-\bar{g}$ and $\bar{h}$ are strictly increasing.

For small $\epsilon>0$, we denote by
$(u_{\epsilon},v_{\epsilon},g_{\epsilon},h_{\epsilon})$  the unique solution of \eqref{model*} with $h_0$ replaced by $h_0^{\epsilon}:=h_0(1-\epsilon)$ , $\mu$ replaced by $\mu_{\epsilon}:=\mu(1-\epsilon)$ and $u_0,v_0$ replaced by $u_0^{\epsilon},v_0^{\epsilon}\in C([-h_0^{\epsilon},h_0^{\epsilon}])$ satisfying
\begin{equation*}
(0,0)\preceq(u_0^{\epsilon}(x),v_0^{\epsilon}(x))\llp(u_0(x), v_0(x)) \mbox{ for }x\in[-h_0^{\epsilon},h_0^{\epsilon}]
\end{equation*}
and
\begin{equation*}
 \Big(u_0^{\epsilon}(\frac{h_0^{\epsilon}}{h_0}x),v_0^{\epsilon}(\frac{h_0^{\epsilon}}{h_0}x) \Big)\to (u_0(x), v_0(x)) \mbox{ as } \epsilon\to 0 \mbox{ in $(C([-h_0, h_0]))^2$}.
\end{equation*}

We claim $\bar{g}(t)<g_{\epsilon}(t)<h_{\epsilon}(t)<\bar{h}(t)$ in $[0,T]$. Otherwise, since $\bar g(0)<g_\epsilon(0)<h_\epsilon(0)<\bar h(0)$, we can find a first $t^*>0$ such that
$\bar{g}(t^*)=g_{\epsilon}(t^*)$ or $\bar{h}(t^*)=h_{\epsilon}(t^*)$, and $\bar g(t)<g_\epsilon(t)<h_\epsilon(t)<\bar h(t)$ for $t\in [0, t^*)$.
 Without loss of generality, we assume that $\bar{h}(t^*)=h_{\epsilon}(t^*)$ and $\bar{g}(t^*)\leq g_{\epsilon}(t^*)$; then we have $h_{\epsilon}'(t^*)\geq\bar{h}'(t^*)$.
Set $\omega=\bar{u}-u_{\epsilon}$ and $\nu=\bar{v}-v_{\epsilon}$; then $(\omega,\nu)$ satisfies
\begin{equation*}
\begin{cases}\displaystyle
\omega_t\geq d_1\int_{g_{\epsilon}(t)}^{h_{\epsilon}(t)}J_1(x-y)\omega(y,t)dy-d_1\omega-a_{11}\omega&\\
\displaystyle
\hspace{2.7cm}+a_{12}\int_{g_{\epsilon}(t)}^{h_{\epsilon}(t)}K(x-y)\nu(y,t)dy,&0<t\leq t^*,~~x\in(g_{\epsilon}(t),h_{\epsilon}(t)),\\
\displaystyle
\nu_t\geq d_2\int_{g_{\epsilon}(t)}^{h_{\epsilon}(t)}J_2(x-y)\nu(y,t)dy-d_2\nu-a_{22}\nu+G'(\bar{\xi})\omega,&0<t\leq t^*,~~x\in(g_{\epsilon}(t),h_{\epsilon}(t)),\\
\omega(x,t)\geq0,~~\nu(x,t)\geq0,&0<t\leq t^*,~x=g_{\epsilon}(t)\mbox{ or }h_{\epsilon}(t),\\
\omega(x,0)\geq,\not\equiv 0,~~~\nu(x,0)\geq, \not\equiv 0,&x\in[-h_0^{\epsilon},h_0^{\epsilon}],
\end{cases}
\end{equation*}
where $G'(\bar{\xi})\geq0$ with $\bar{\xi}:=\bar{\xi}(x,t)\in[\min\{u_{\epsilon},\bar{u}\},\max\{u_{\epsilon},\bar{u}\}]$. By Lemma \ref{maximalprinciple}, we have $\bar{u}> u_{\epsilon}$ and $\bar{v}>v_{\epsilon}$ for
$t\in(0,t^*]$ and $x\in(g_{\epsilon}(t),h_{\epsilon}(t))$.
It follows  that
\begin{equation*}
\begin{split}
0\geq&\ \bar{h}'(t^*)- h_{\epsilon}'(t^*)\\
\geq&\ \mu\int_{\bar{g}(t^*)}^{\bar{h}(t^*)}\int^{\infty}_{\bar{h}(t^*)}J_1(x-y)\bar{u}(x,t^*)dydx-
\mu_{\epsilon} \int_{ g_{\epsilon}(t^*)}^{ h_{\epsilon}(t^*)}\int^{\infty}_{ h_{\epsilon}(t^*)}J_1(x-y) u_{\epsilon}(x,t^*)dydx\\
&+\mu\rho\int_{\bar{g}(t^*)}^{\bar{h}(t^*)}\int^{\infty}_{\bar{h}(t^*)}J_2(x-y)\bar{v}(x,t^*)dydx
-\mu_{\epsilon}\rho \int_{ g_{\epsilon}(t^*)}^{ h_{\epsilon}(t^*)}\int^{\infty}_{ h_{\epsilon}(t^*)}J_2(x-y) v_{\epsilon}(x,t^*)dydx\\
>&\ \mu_{\epsilon} \int_{ g_{\epsilon}(t^*)}^{ h_{\epsilon}(t^*)}\int^{\infty}_{ h_{\epsilon}(t^*)}J_1(x-y)(\bar{u}(x,t^*)- u_{\epsilon}(x,t^*))dydx\\
&+\mu_{\epsilon}\rho\int_{ g_{\epsilon}(t^*)}^{ h_{\epsilon}(t^*)}\int^{\infty}_{ h_{\epsilon}(t^*)}J_2(x-y)(\bar{v}(x,t^*)- v_{\epsilon}(x,t^*))dydx>0.
\end{split}
\end{equation*}
This is a contradiction. Therefore, $\bar{g}(t)<g_{\epsilon}(t)$ and $h_{\epsilon}(t)<\bar{h}(t)$ hold for $t\in(0,T]$. It then follows that $\bar{u}> u_{\epsilon}$ and $\bar{v}>v_{\epsilon}$ for
$t\in(0,T]$ and $x\in(g_{\epsilon}(t),h_{\epsilon}(t))$.

By the continuous dependence  of the unique solution of \eqref{model*}
on the initial data, we see that $(u_{\epsilon},v_{\epsilon},g_{\epsilon},h_{\epsilon})\rightarrow (u,v,g,h)$ as $\epsilon\to 0$, and the desired inequalities thus follow immediately.
The lemma is proved.
\end{proof}
\begin{remark}\label{compariprinciple-lower}
$(\bar{u},\bar{v},\bar{g},\bar{h})$ in Lemma \ref{compariprinciple} is often called an upper solution of
\eqref{model*}.
By reversing all the inequalities in  \eqref{model_upper},  we have an analogous conclusion of Lemma \ref{compariprinciple}  for lower solutions.
\end{remark}

An obvious variation of the proof of Lemma \ref{compariprinciple} gives the following result.

\begin{lemma}\label{compariprinciple-speed}
In Lemma \ref{compariprinciple}, if we take $\bar g(t)\equiv{g}(t)$ for $t\in[0,T]$ and suppose
 $(\bar{u},\bar{v},g,\bar{h})$ satisfies
\begin{equation*}
\begin{cases}\displaystyle
\bar{u}_t\geq d_1\int_{g(t)}^{\bar{h}(t)}J_1(x-y)\bar{u}(y,t)dy-d_1\bar{u}-a_{11}\bar{u}&\\
\displaystyle
\hspace{2.7cm}+a_{12}\int_{g(t)}^{\bar{h}(t)}K(x-y)\bar{v}(y,t)dy,&0<t\leq T,~~x\in(g(t),\bar{h}(t)),\\
\displaystyle
\bar{v}_t\geq d_2\int_{g(t)}^{\bar{h}(t)}J_2(x-y)\bar{v}(y,t)dy-d_2\bar{v}-a_{22}\bar{v}+G(\bar{u}),&0<t\leq T,~~x\in(g(t),\bar{h}(t)),\\
\displaystyle
\displaystyle
\bar{h}'(t)\geq\mu  \int_{g(t)}^{\bar{h}(t)}\int^{\infty}_{\bar{h}(t)}J_1(x-y)\bar{u}(x,t)dydx&\\
\displaystyle
\hspace{1.2cm}+\mu\rho \int_{g(t)}^{\bar{h}(t)}\int^{\infty}_{\bar{h}(t)}J_2(x-y)\bar{v}(x,t)dydx,&0<t\leq T,\\
\bar{u}(x,t)\geq0,~~\bar{v}(x,t)\geq0,&0<t\leq T,~x=\bar{h}(t)\mbox{ or }g(t),\\
\bar{h}(0)\geq h_0,&\\
\bar{u}(x,0)\geq \bar{u}_0(x),~~~\bar{v}(\bar{x},0)\geq v_0(x),&x\in[-h_0,h_0],
\end{cases}
\end{equation*}
then the following holds:
\begin{equation*}
\begin{cases}
h(t)\leq \bar{h}(t)~~~\text{for}~~t\in [0,T],\\
u(x,t)\leq \bar{u}(x,t),~~v(x,t)\leq \bar{v}(x,t) ~~~\text{for} ~~x\in [g(t),h(t)],~~t\in(0,T].
\end{cases}
\end{equation*}
\end{lemma}

\section{The corresponding fixed boundary problem}
In this section, we consider the associated fixed boundary problem of \eqref{model*}, namely the system with
\[
-g(t)=h(t)\equiv l>0,
\]
which is given by
\begin{equation}\label{model_fixbound}
\begin{cases}\displaystyle
u_t=d_1\int_{-l}^{l}J_1(x-y)u(y,t)dy-d_1u-a_{11}u&\\
\displaystyle
\hspace{2.7cm}+a_{12}\int_{-l}^{l}K(x-y)v(y,t)dy,&t>0,~~x\in(-l,l),\\
\displaystyle
v_t=d_2\int_{-l}^{l}J_2(x-y)v(y,t)dy-d_2v-a_{22}v+G(u),&t>0,~~x\in(-l,l),\\
u(x,0)=u_0(x),~~~v(x,0)=v_0(x),&x\in[-l,-l],
\end{cases}
\end{equation}
where
 \begin{equation*}
(0,0)\preceq (u_0,v_0)\in C([-l,l])^2.
\end{equation*}
It turns out that here rather different techniques are needed from previous works where the nonlocal term $a_{12}\int_{-l}^{l} K(x-y) u(y,t)dy$ in \eqref{model_fixbound} was replaced by $a_{12} u(x,t)$. Note that a good understanding of the behaviour of this fixed boundary problem is crucial for determining the long-time behaviour of  \eqref{model*}.

\subsection{The associated eigenvalue problem}

The long-time dynamical behaviour of \eqref{model_fixbound} is governed by its linearised eigenvalue problem at the trivial solution $(0,0)$, which has the form
\begin{equation}\label{eigenprobelm1}
\begin{cases}
\displaystyle
d_1\int_{-l}^{l}J_1(x-y)\theta_1(y)dy-d_1\theta_1-a_{11}\theta_1
\displaystyle
+a_{12}\int_{-l}^{l}K(x-y)\theta_2(y)dy=\lambda \theta_1,&x\in[-l,l],\\
\displaystyle
d_2\int_{-l}^{l}J_2(x-y)\theta_2(y)dy-d_2\theta_2-a_{22}\theta_2+G'(0)\theta_1=\lambda\theta_2,&x\in[-l,l].\\
\end{cases}
\end{equation}

Let $D:=C([-l,l])^2$ be equipped with norm
$\|\bm{\theta}\|_{D}=\sup_{x\in[-l,l]}|\bm{\theta}(x)|$, where
\begin{equation*}
\bm{\theta}=(\theta_1,\theta_2)\in D\mbox{ and }|\bm{\theta}(x)|=\sqrt{\theta_1^2(x)+\theta_2^2(x)}\mbox{ for
}x\in[-l,l].
\end{equation*}
Denote
\begin{equation*}\begin{cases}
D_+:=\{\bm{\theta}\in D\mid\theta_1(x),\theta_2(x)\geq 0,x\in[-l,l] \},\\
D_{+}^0:=\{\bm{\theta}\in D\mid\theta_1(x),\theta_2(x)> 0,x\in[-l,l] \}.
\end{cases}
\end{equation*}
We then define two operators $\mathcal{J}, \mathcal{T}: D\to D$ by
\begin{equation*}
\mathcal{J}(\theta_1,\theta_2)(x)= \left(
\displaystyle d_1\int_{-l}^{l}J_1(x-y)\theta_1(y)dy,
\displaystyle d_2\int_{-l}^{l}J_2(x-y)\theta_2(y)dy\right)
\end{equation*}
and
\begin{equation*}
\mathcal{T}(\theta_1,\theta_2)(x)=\left(\begin{matrix}
(d_1+a_{11})\theta_1(x)\displaystyle-a_{12}\int_{-l}^{l}K(x-y)\theta_2(y)dy\\
-G'(0)\theta_1(x)+(d_2+a_{22})\theta_2(x)
\end{matrix}\right)^T
\end{equation*}
for $\bm{\theta}=(\theta_1,\theta_2)\in D$. Then \eqref{eigenprobelm1} can be rewritten as
\begin{equation*}
(\mathcal{J}-\mathcal{T})\bm{\theta}=\lambda\bm{\theta}.
\end{equation*}

Throughout this section, we always assume that $J_1, J_2$ and $K$ satisfy {\bf (J)}, and $G$ satisfies (G1) and (G2).

\begin{lemma}\label{lemma_eigen1}
 The eigenvalue problem
\begin{equation}\label{eigenpro2}
\begin{cases}
\displaystyle
(d_1+a_{11})\theta_1(x)-a_{12}\int_{-l}^{l}K(x-y)\theta_2(y)dy=\lambda \theta_1(x),&x\in[-l,l],\\
\displaystyle
-G'(0)\theta_1(x)+(d_2+a_{22})\theta_2(x)=\lambda\theta_2(x),&x\in[-l,l]\\
\end{cases}
\end{equation}
admits a unique real eigenvalue denoted by $\lambda_1$ whose  corresponding
  eigenfunction pair $({\theta}_1,{\theta}_2)\in D_{+}^0$.
\end{lemma}
\begin{proof}
We first consider the following eigenvalue problem:
\begin{equation}\label{eigenpro3}
\int_{-l}^{l}K(x-y)\theta_2(y)dy=\tilde{\lambda}\, \theta_2(x) \mbox{ for }x\in[-l,l].
\end{equation}
It follows from Theorem 3.1 in \cite{Hutson2003JMB} that \eqref{eigenpro3} admits a positive principal eigenvalue $\tilde{\lambda}$ given by
\begin{equation}\label{tilde-lambda}
\tilde{\lambda}=\sup_{\theta_*\in L^2([-l,l]),\|\theta_*\|\neq 0}\frac{\displaystyle\int_{-l}^l\int_{-l}^lK(x-y)\theta_*(x)\theta_*(y)dydx}{\displaystyle
\int_{-l}^l\theta_*^2(x)dx}\\
\end{equation}
and a unique (up to a scalar multiple) corresponding positive eigenfunction $\theta_*(x)$.
Note that problem \eqref{eigenpro2} can be rewritten as
\begin{equation*}
\int_{-l}^{l}K(x-y)\theta_2(y)dy=\frac{(d_1+a_{11}-\lambda)(d_2+a_{22}-\lambda)}{a_{12}G'(0)}\theta_2(y).
\end{equation*}
Let
\begin{equation}\label{tilda_lambda}
\frac{(d_1+a_{11}-\lambda)(d_2+a_{22}-\lambda)}{a_{12}G'(0)}=\tilde{\lambda}.
\end{equation}
Simple calculations show that the above identity holds when $\lambda$ takes the following values $\lambda_1$
and $\lambda_2$:
\begin{align*}
&\lambda_{1}=\frac{(d_1+a_{11}+d_2+a_{22})-\sqrt{(d_1+a_{11}-d_2-a_{22})^2+4a_{12}G'(0)\tilde{\lambda}}}{2},\\
&\lambda_{2}=\frac{(d_1+a_{11}+d_2+a_{22})+\sqrt{(d_1+a_{11}-d_2-a_{22})^2+4a_{12}G'(0)\tilde{\lambda}}}{2}.\\
\end{align*}
When we take $\tilde{\theta}:=(\tilde{\theta}_1,\tilde{\theta}_2)=(\delta_1\theta_*,\theta_*)$ with
\begin{equation*}
\delta_1:=\frac{d_2+a_{22}-\lambda_{1}}{G'(0)}>0,
\end{equation*}
it is easily checked that $(\delta_1\theta_*,\theta_*)$ is a positive solution of \eqref{eigenpro2} with $\lambda=\lambda_{1}$.
When $\lambda=\lambda_2$, $(\tilde{\theta}_1,\tilde{\theta}_2)=(\delta_2\theta_*,\theta_*)$ solve \eqref{eigenpro2}  with
\begin{equation*}
\delta_2:=\frac{d_2+a_{22}-\lambda_{2}}{G'(0)}<0.
\end{equation*}

Conversely, if $(\theta_1,\theta_2)\in D_+^0$ is a solution of \eqref{eigenpro2}, then
necessarily $\theta_2$ is a positive solution of \eqref{eigenpro3} with $\tilde{\lambda}$ defined in \eqref{tilda_lambda}
and thus \begin{equation*}\theta_1=\frac{d_2+a_{22}-\lambda}{G'(0)} \theta_2.\end{equation*}
It follows from the positivity of $(\theta_1,\theta_2)$ that $\lambda<d_2+a_{22}$ and hence $\lambda=\lambda_1$.
\end{proof}
\begin{lemma}\label{invert}
The following statements are valid:
\begin{itemize}
  \item [$(i)$] $(\mathcal{T}+\alpha \mathbf{I})^{-1}$ exists and is
  a bounded strongly positive operator for any $\alpha>-\lambda_{1}$;
   \item [$(ii)$]$\|(\mathcal{T}+\alpha \mathbf{I})^{-1}\|\rightarrow 0$ as $\alpha\rightarrow \infty$;
\item [$(iii)$] For any finite interval $[\alpha_1,\alpha_2]\subset (-\lambda_1,\infty)$,
$\sup_{\alpha\in[\alpha_1,\alpha_2]}\|(\mathcal{T}+\alpha \mathbf{I})^{-1}\|<\infty$.
   \end{itemize}
\end{lemma}


\begin{proof}
$(i)$ We first prove the existence of $(\mathcal{T}+\alpha\mathbf{I} )^{-1}$ by showing that
$\mathcal{T}+\alpha\mathbf{I} :  D \rightarrow  D$
is one-to-one and onto.
Suppose that there exists $(\theta_1,\theta_2)\in D$ such that
\begin{equation*}\label{T+alpha}
(\mathcal{T}+\alpha\mathbf{I} )(\theta_1,\theta_2)=0 ,
\end{equation*}
which is equivalent to
\begin{equation}\label{K}
\int_{-l}^{l}K(x-y)\theta_2(y)dy=\alpha_1\theta_2(x)
\end{equation}
with
\begin{equation*}
\alpha_1:=\frac{(d_1+a_{11}+\alpha)(d_2+a_{22}+\alpha)}{a_{12}G'(0)}.
\end{equation*}
From $\alpha>-\lambda_{1}$ we obtain  $\alpha_1> \tilde{\lambda}$, where $\tilde\lambda$ is given by \eqref{tilde-lambda}.
We next show that $(\theta_1,\theta_2)=(0,0)$ by adopting the proof
of the
Touching Lemma in \cite{FangLi2017DCDS}. Otherwise, necessarily
$\theta_2\not\equiv 0$; without loss of generality, we assume that $\theta_2(x)<0$ for some $x\in[-l,l]$.
Note that
\begin{equation}\label{touchlemma}
\int_{-l}^{l}K(x-y)\theta_*(y)dy=\tilde{\lambda}\,\theta_*(x)<\alpha_1\theta_*(x)
\end{equation}
with $\theta_*(x)>0$ in $[-l,l]$ by \eqref{eigenpro3}.
Then we can find some constant $c_0>0$ such that $\theta_*+c_0\theta_2$ is
less than $0$ somewhere in $[-l,l]$ if $c_0>0$ is large, and strictly
positive in $[-l,l]$ if $c_0$ is small.
Therefore, there exists $c_0^*\in(0,\infty)$ such that
 $\theta_*(x)+c_0^*\theta_2(x)\geq0$ for $x\in[-l,l]$ but $\theta_*(x_0)+c_0^*\theta_2(x_0)=0$ for some $x_0\in[-l,l]$.
If $\theta_*(x)+c_0^*\theta_2(x)\equiv0$ for $x\in[-l,l]$, then \eqref{K} implies $\alpha_1=\tilde\lambda$, which is a contradiction. Therefore, we can find
$\tilde x_0\in \partial\{x\in[-l,l]|\theta_*(x)+c_0^*\theta_2(x)>0\}$, which implies
\begin{equation*}
0<\int_{-l}^{l}K(\tilde x_0-y)\theta_*(y)dy+\int_{-l}^{l}K(\tilde x_0-y)c_0^*\theta_2(y)dy
<\alpha_1\theta_*(\tilde x_0)+\alpha_1c_0^*\theta_2(\tilde x_0)=0.
\end{equation*}
This contradiction implies that $\mathcal{T}+\alpha\mathbf{I}$ is one-to-one.

Next, we show that $\mathcal{T}+\alpha\mathbf{I} $ is onto, namely, for any given $(\xi_1,\xi_2)\in D$, there exists
$(\theta_1,\theta_2)\in D$ such that
\begin{equation}\label{ontoequa}
(\mathcal{T}+\alpha\mathbf{I})(\theta_1,\theta_2)=(\xi_1,\xi_2).
\end{equation}
Clearly, \eqref{ontoequa} is equivalent to
\begin{equation*}
\frac{(d_1+a_{11}+\alpha)(d_2+a_{22}+\alpha)}{a_{12}G'(0)}\theta_2-
\frac{(d_1+a_{11}+\alpha)}{a_{12}G'(0)}\xi_2-\int_{-l}^{l}K(x-y)\theta_2(y)dy=\frac{\xi_1}{a_{12}},
\end{equation*}
which can be rewritten as
\begin{equation}\label{ontoequa1}
\int_{-l}^{l}K(x-y)\theta_2(y)dy-\frac{(d_1+a_{1}+\alpha)(d_2+a_{22}+\alpha)}{a_{12}G'(0)}\theta_2=F(\xi_1,\xi_2),
\end{equation}
where \begin{equation*}F(\xi_1,\xi_2):=-\frac{\xi_1}{a_{12}}+\frac{(d_1+a_{11}+\alpha)}{a_{12}G'(0)}\xi_2.\end{equation*}
By the continuity  of $K$, for any $\varepsilon>0$, there exists $\delta>0$ such that
\begin{equation*}
|K(x_1)-K(x_2)|\leq \frac{\varepsilon}{2l}
\end{equation*}
for $x_1,x_2\in[-l,l]$ satisfying $|x_1-x_2|<\delta$ .
Then for every bounded sequence $\{\theta_{2n}\}\subset C([-l,l])$,
\begin{align*}
\bigg|\int^{l}_{-l}K(x_1-y)\theta_{2n}(y)dy-\int^{l}_{-l}K(x_2-y)\theta_{2n}(y)dy\bigg|
\leq \varepsilon \|\theta_{2n}\|,
\end{align*}
 which implies
$ \displaystyle\int_{-l}^{l}K(x-y)\theta_{2n}(y)dy$ is equicontinuous and thus
$\displaystyle\Big\{\int_{-l}^{l}K(x-y)\theta_{2n}(y)dy\Big\}$ is
precompact in $C([-l,l])$ by the Arzel\'{a}-Ascoli theorem.
Then by the Fredholm alternative (see e.g. Theorem 6.6 in \cite{Brezis2011}), \eqref{ontoequa1}
admits a solution $\theta_2\in C([-l,l])$ since \eqref{K} only has a trivial solution. Thus
$\mathcal{T}+\alpha\mathbf{I}$ is onto for any $\alpha>-\lambda_1$.
Hence, $(\mathcal{T}+\alpha \mathbf{I})^{-1}$ exists for any $\alpha>-\lambda_{1}$.
Moreover, by the Bounded Inverse Theorem, $(\mathcal{T}+\alpha\mathbf{I} )^{-1}$ is a bounded linear
operator for any $\alpha>-\lambda_{1}$.

Finally, we show that $(\mathcal{T}+\alpha \mathbf{I})^{-1}$ is a strongly positive operator for any $\alpha>-\lambda_{1}$.
It suffices to prove that,
for $(\bar{\xi}_1,\bar{\xi}_2)\in D$, if $(\mathcal{T}+\alpha \mathbf{I})(\bar{\xi}_1,\bar{\xi}_2)\in D_+\setminus \{\mathbf{0}\}$, then
$(\bar{\xi}_1,\bar{\xi}_2)\in D_{+}^0$.
Indeed, from
 \begin{equation}\label{strong+}
\begin{cases}
\displaystyle
(d_1+a_{11}+\alpha)\bar{\xi}_1(x)-a_{12}\int_{-l}^{l}K(x-y)\bar{\xi}_2(y)dy\geq0,&x\in[-l,l],\\
\displaystyle
-G'(0)\bar{\xi}_1(x)+(d_2+a_{22}+\alpha)\bar{\xi}_2(x)\geq 0,&x\in[-l,l],\\
\end{cases}
\end{equation}
it is easy to calculate that
\begin{equation*}
\int_{-l}^{l}K(x-y)\bar{\xi}_2(y)dy\leq \alpha_1\bar{\xi}_2(x).
\end{equation*}
We first prove that $\bar{\xi}_2\in C([-l,l])$ is nonnegative. If not, we assume $\bar{\xi}_2(x)<0$ somewhere
in $[-l,l]$. Note that \eqref{touchlemma} holds for $\theta_*(x)>0$ in $[-l,l]$. By applying the proof of the Touching Lemma again, we see that
\begin{equation*}
\int_{-l}^{l}K(x-y)\bar{\xi}_2(y)dy-\alpha_1\bar{\xi}_2(x)=\int_{-l}^{l}K(x-y)\theta_*(y)dy -\alpha_1\theta_*(x)\equiv0,
\end{equation*}
which is a contradiction since $\alpha_1>\tilde{\lambda}$. Next we prove that $\bar{\xi}_2(x)\not\equiv0$ in $[-l,l]$.
Otherwise, $\bar{\xi}_2(x)\equiv0$ in $[-l,l]$. It follows from the first inequality of \eqref{strong+} that
$\bar{\xi}_1(x)\geq0$ in $[-l,l]$, while $\bar{\xi}_1(x)\leq0$ holds in $[-l,l]$ by the second inequality. Thus
$\bar{\xi}_1(x)\equiv0$ in $[-l,l]$. It is a contradiction since then
$(\mathcal{T}+\alpha\mathbf{I})(\bar{\xi}_1,\bar{\xi}_2)\equiv(0,0)$.
Finally we show that $\bar{\xi}_2(x)$ is strictly positive in $[-l,l]$; otherwise, there exists some
point $x_0\in\partial\{x\in[-l,l]\mid\bar{\xi}_2(x)>0\}$ such that
$\bar{\xi}_2(x_0)=0$. Therefore, we see that
\begin{equation*}
0<\int_{-l}^{l}K(x_0-y)\bar{\xi}_2(y)dy\leq\alpha_1\bar{\xi}_2(x_0)=0,
\end{equation*}
which is impossible. Therefore, $\bar{\xi}_2(x)>0$ in $[-l,l]$
and it follows from the first inequality in \eqref{strong+} that $\bar{\xi}_1(x)>0$ holds in $[-l,l]$.

$(ii)$ Suppose for contradiction there exists a sequence $\alpha_n>-\lambda_1$ such that
$\|(\mathcal{T}+\alpha_n\mathbf{I})^{-1}\|\geq 2\tilde{c}>0$
and $\alpha_n\rightarrow \infty$. Therefore, $\|(\mathcal{T}+\alpha_n\mathbf{I})^{-1}{\Psi_n}\| \geq\tilde{c}$ for some $\Psi_n\in D$ with $\|\Psi_n\|=1$.
Let $\Phi_n:=(\mathcal{T}+\alpha_n \mathbf{I})^{-1}\Psi_n$;
then $\|\Phi_n\|\geq\tilde{c}$ and
\begin{equation*}
\frac{\Psi_n}{\|\Phi_n\|}=(\mathcal{T}+\alpha_n \mathbf{I})\frac{\Phi_n}{\|\Phi_n\|},
\ \
\|(\mathcal{T}+\alpha_n \mathbf{I})\Theta_{n}\|\leq 1/\tilde{c},
\end{equation*}
where $\Theta_n=\frac{\Phi_n}{\|\Phi_n\|}\in D$ and clearly
$\|\Theta_{n}\|=1$ for all $n\geq 1$.
Then
\begin{equation*}
\alpha_n\|\Theta_{n}\|=\|(\mathcal{T}+\alpha_n \mathbf{I})\Theta_{n} -\mathcal{T}\Theta_{n}\|\leq 1/\tilde c + \|\mathcal{T}\|\|\Theta_{n}\|
\end{equation*}
and thus
\begin{equation*}
1=\|\Theta_{n}\|\leq \frac{1}{\tilde{c}(\alpha_n-\|T\|)} \to 0
\end{equation*}
as $n\rightarrow \infty$. This  contradiction proves the desired conclusion.

$(iii)$ By way of contradiction, we assume that there exists a sequence
$\tilde{\alpha}_n\rightarrow \tilde{\alpha}\in[\alpha_1,\alpha_2]\subset (-\lambda_1, \infty)$
such that $\|(\mathcal{T}+\tilde{\alpha}_n \mathbf{I})^{-1}\|\rightarrow\infty$ as $n\rightarrow \infty$.
Then we can find $\Psi_n\in D$ such that $\|\Psi_n\|=1$ and
$\|(\mathcal{T}+\tilde{\alpha}_n \mathbf{I})^{-1}\Psi_n\|\rightarrow\infty$ as $n\rightarrow \infty$.
Letting $\tilde{\Phi}_n:=(\mathcal{T}+\tilde{\alpha}_n \mathbf{I})^{-1}\Psi_n$, then
$\lim_{n\rightarrow \infty}\|\tilde{\Phi}_n\|=\infty$.
Similar to the proof of $(ii)$, we have
\begin{equation}\label{tilda_alpha}
\lim_{n\rightarrow \infty}(\mathcal{T}+\tilde{\alpha}_n \mathbf{I})\tilde{\Theta}_{n}=\bm{0},
\end{equation}
with $\tilde{\Theta}_{n}=(\tilde{\theta}_{1n},\tilde{\theta}_{2n}):=\frac{\tilde{\Phi}_n}{\|\tilde{\Phi}_n\|}\in D$ and
$\|\tilde{\Theta}_{n}\|=1$ for all $n\in \mathbb{N}^+$.

Since $\displaystyle\Big\{\int_{-l}^{l}K(x-y)\tilde{\theta}_{2n}(y)dy\Big\}$ is
precompact in $C([-l,l])$, there
exists a subsequence of $\{\tilde{\theta}_{2n}\}$,  still denoted by $\{\tilde{\theta}_{2n}\}$
for simplicity, such that
\begin{equation*}\lim_{n\rightarrow \infty}\int_{-l}^{l}K(x-y)\tilde{\theta}_{2n}(y)dy \mbox{ exists in $C([-l,l])$}.\end{equation*}
Therefore, by the first equation in \eqref{tilda_alpha}, along the same subsequence,
$\lim_{n\rightarrow \infty}\{\tilde{\theta}_{1n}\}$ exists in $C([-l,l])$. It then
follows from the second equation  in \eqref{tilda_alpha} that $\lim_{n\rightarrow\infty}\tilde{\theta}_{2n}$ exists in $C([-l,l])$.
Therefore,
\begin{equation}\label{bound_T}
\int_{-l}^{l}K(x-y)\lim_{n\rightarrow\infty}\tilde{\theta}_{2n}(y)dy=
\frac{(d_1+a_{11}+\tilde{\alpha})(d_2+a_{22}+\tilde{\alpha})}{a_{12}G'(0)}\lim_{n\rightarrow\infty}\tilde{\theta}_{2n}(x).
\end{equation}
Since $\tilde{\alpha}>-\lambda_1$, by the proof of $(i)$,
\eqref{bound_T} holds only when $\lim_{n\rightarrow\infty}\tilde{\theta}_{2n}=0$ in $C([-l,l])$, which yields
that $\lim_{n\rightarrow\infty}\tilde\theta_{1n}=0$ in $C([-l,l])$.
Hence,
\[\lim_{n\rightarrow\infty}\|\tilde{\Theta}_{n}\|=
\lim_{n\rightarrow\infty}\max_{x\in[-l,l]}\sqrt{\tilde{\theta}_{1n}^2(x)+\tilde{\theta}_{2n}^2(x)}=0,\]
which is a contradiction.
\end{proof}

\begin{lemma}\label{compactness}  For each $\alpha>-\lambda_{1}$,
$\mathcal{J}(\mathcal{T}+\alpha\mathbf{I} )^{-1}$ and  $ (\mathcal{T}+\alpha\mathbf{I} )^{-1}\mathcal J$ are both compact and strongly positive operators.
\end{lemma}
\begin{proof}
It follows from Lemma \ref{invert} that $(\mathcal{T}+\alpha \mathbf{I} )^{-1}$ is a bounded strongly positive linear operator for any $\alpha>-\lambda_{1}$.
We now prove that $\mathcal{J}$ is compact. Let
$\{\mathbf{{u_n}}\}\subset D$ be a bounded sequence.
Define
$
\mathbf{v_n}=\mathcal{J}\mathbf{u_n}$.
It  follows from assumption $\textbf{(J)}$ that $\{\mathbf{v_n}\}$ is bounded and equicontinuous. By the Arzel\`{a}-Ascoli theorem,
$\{\mathbf{v_n}\}$ is precompact.
Hence, $\mathcal{J}$ is a compact operator.

Moreover, it is also easily seen that if $(\theta_1,\theta_2)\in D_+\setminus\{\bf 0\}$, then $\mathcal J(\theta_1,\theta_2)\in D_+\setminus\{\bf 0\}$,
and if $(\theta_1,\theta_2)\in D_+^0$, then $\mathcal J(\theta_1,\theta_2)\in D_+^0$. Thus $\mathcal{J}(\mathcal{T}+\alpha\mathbf{I} )^{-1}$ and  $ (\mathcal{T}+\alpha\mathbf{I} )^{-1}\mathcal J$ are both compact and strongly positive for any $\alpha>-\lambda_{1}$.
 This finishes the proof.
\end{proof}

\begin{lemma}\label{spectral_J}
For $\alpha>-\lambda_1$, denote by $r(\alpha):=r(\mathcal{J}(\mathcal{T}+\alpha\mathbf{I} )^{-1})$ the spectral
radius of $\mathcal{J}(\mathcal{T}+\alpha\mathbf{I}
)^{-1}$. Then the following statements are valid:
\begin{itemize}
\item[$(i)$] $r((\mathcal{T}+\alpha\mathbf{I} )^{-1}\mathcal J)=r(\alpha)$,
  \item [$(ii)$] $\lim_{\alpha\searrow -\lambda_1}r(\alpha)=\infty$,
   \item [$(iii)$]$\lim_{\alpha\rightarrow \infty}r(\alpha)=0$,
   \item [$(iv)$] $r(\alpha)$ is continuous and strictly decreasing for $\alpha>-\lambda_{1}$.

   \end{itemize}
\end{lemma}

\begin{proof}
$(i)$ By Lemma \ref{compactness}  and the Krein-Rutman Theorem,
$r(\alpha)=r(\mathcal J(\mathcal{T}+\alpha \textbf{I})^{-1})$ is the unique eigenvalue which corresponds to an eigenvector $\Phi\in D_+^0$, namely
\[
\mathcal J(\mathcal{T}+\alpha \textbf{I})^{-1}\Phi=r(\alpha) \Phi.
\]
Letting $\Psi:=(\mathcal{T}+\alpha \textbf{I})^{-1}\Phi$, then
\[
(\mathcal{T}+\alpha \textbf{I})^{-1}\mathcal J\Psi=r(\alpha)\Psi,\; \Psi\in D_+^0.
\]
Applying the Krein-Rutman theorem to $(\mathcal{T}+\alpha \textbf{I})^{-1}\mathcal J$, we similarly conclude that
$r((\mathcal{T}+\alpha\mathbf{I} )^{-1}\mathcal J)$ is the only eigenvalue which corresponds to an eigenvector in $D_+^0$.
Thus the above identity implies that $r(\alpha)=r((\mathcal{T}+\alpha\mathbf{I} )^{-1}\mathcal J)$. This proves $(i)$.

$(ii)$  By Lemma \ref{lemma_eigen1} and Lemma \ref{invert},
for $x\in[-l,l]$, we have, with $\bm{\tilde{\theta}}=(\tilde{\theta}_1,\tilde{\theta}_2)=(\delta_1\theta_*,\theta_*)$,
\begin{equation*}
\mathcal{J}(\mathcal{T}+\alpha\bm{I} )^{-1}(\tilde{\theta}_1,\tilde{\theta}_2)(x)=
\left(\begin{matrix}\displaystyle
\frac{d_1}{\lambda_{1}+\alpha}\int^{l}_{-l}J_1(x-y)\tilde{\theta}_1(y)dy\\
\displaystyle
\frac{d_2}{\lambda_{1}+\alpha}\int^{l}_{-l}J_2(x-y)\tilde{\theta}_2(y)dy\\
\end{matrix}\right)^T.
\end{equation*}
It follows from $\textbf{(J)}$ that there exist a small constant $\varepsilon>0$ and some positive constant $\delta^0>0$ such that, for $i=1,2$,
\begin{equation*}
J_i(x-y)\geq\delta^0 \mbox{ if }|x-y|\leq \varepsilon.
\end{equation*}
Let $\tilde{\theta}_m:=\min_{x\in[-l,l]}\min\{\tilde{\theta}_1(x),\tilde{\theta}_2(x)\}$.
Then
\begin{equation*}\displaystyle
\int^{l}_{-l}J_i(x-y)\tilde{\theta}_i(y)dy
\geq\varepsilon\delta^0\tilde{\theta}_m \mbox{ for any } x\in [-l,l].
\end{equation*}
Since
\begin{equation*}\displaystyle
\lim_{\alpha\searrow -\lambda_{1}}\frac{1}{\lambda_{1}+\alpha}=\infty,
\end{equation*}
given any $\gamma>1$, for $0<\alpha+\lambda_{1}\ll1$ we thus have
\begin{equation}\label{gamma}\displaystyle
\frac{d_i}{\lambda_{1}+\alpha}\int^{l}_{-l}J_i(x-y)\tilde{\theta}_i(y)dy
\geq \frac{d_i}{\lambda_1+\alpha}\frac{\varepsilon \delta^0\tilde \theta_m}{\|\tilde \theta_i\|_\infty}\tilde\theta_i(x)\geq  \gamma \tilde{\theta}_i(x) \mbox{ for } x\in [-l,l],\ i=1,2.
\end{equation}
By \eqref{gamma}, we see that, for any integer $n\geq1$, $(\mathcal{J}(\mathcal{T}+\alpha\mathbf{I} )^{-1})^n\bm{\tilde{\theta}}\geq \gamma^n  \bm{\tilde{\theta}}$, and hence
\begin{equation*}
\|(\mathcal{J}(\mathcal{T}+\alpha\mathbf{I} )^{-1})^n\|\|\bm{\tilde{\theta}}\|\geq \|(\mathcal{J}(\mathcal{T}+
\alpha\mathbf{I} )^{-1})^n\bm{\tilde{\theta}}\|\geq \gamma^n\|\bm{\tilde{\theta}}\|.
\end{equation*}
Therefore,
\begin{equation*}
r(\alpha)=\lim_{n\rightarrow\infty}\|(\mathcal{J}(\mathcal{T}+\alpha\mathbf{I} )^{-1})^n\|^{\frac{1}{n}}\geq\gamma,
\end{equation*}
which clearly implies the desired conclusion.

$(iii)$ 
For any $\alpha>-\lambda_1$,
\[
r(\alpha)\leq \|\mathcal{J}(\mathcal{T}+\alpha\mathbf{I} )^{-1}\|\leq \|\mathcal{J}\|
\|(\mathcal{T}+\alpha\mathbf{I} )^{-1}\|,
\]
and it follows from $(ii)$ of Lemma \ref{invert} that $r(\alpha)\rightarrow 0$ as $\alpha$ goes to $\infty$.

$(iv)$ Some of the idea of proving the continuity follows the proof Lemma 2 in \cite{Burger1988}.
We first prove that $r(\alpha)$ is continuous for $\alpha>-\lambda_1$.
Assume that $\hat{\alpha},\alpha\in(-\lambda_1,\infty)$. Then we see
\begin{align*}
\mathcal{J}(\mathcal{T}+\alpha\mathbf{I} )^{-1}=&\mathcal{J}(\mathcal{T}+\hat{\alpha}\mathbf{I} )^{-1}
(T+\hat{\alpha}\mathbf{I} )(\mathcal{T}+\alpha\mathbf{I})^{-1}\\
=&\mathcal{J}(\mathcal{T}+\hat{\alpha}\mathbf{I} )^{-1}
+(\hat{\alpha}-\alpha)\mathcal{J}(\mathcal{T}+\hat{\alpha}\mathbf{I} )^{-1}(\mathcal{T}+\alpha\mathbf{I})^{-1},
\end{align*}
which gives
\begin{align*}
\|\mathcal{J}(\mathcal{T}+\alpha\mathbf{I} )^{-1}-\mathcal{J}(\mathcal{T}+\hat{\alpha}\mathbf{I} )^{-1}\|
\leq|\hat{\alpha}-\alpha|\|\mathcal{J}(\mathcal{T}+\hat{\alpha}\mathbf{I} )^{-1}\|
\|(\mathcal{T}+\alpha\mathbf{I})^{-1}\|.
\end{align*}
By $(ii)$ and $(iii)$ of Lemma \ref{invert},
$\sup_{\alpha\geq\hat{\alpha}}\|(\mathcal{T}+\alpha\mathbf{I})^{-1}\|<\infty$ for
any fixed $\hat{\alpha}>-\lambda_1$.
Hence, $\mathcal{J}(\mathcal{T}+\alpha \mathbf{I})^{-1}$ converges to
$\mathcal{J}(\mathcal{T}+\hat{\alpha}\mathbf{I})^{-1}$ as $\alpha\rightarrow\hat{\alpha}$. We may now apply the Krein-Rutman theorem to the strongly positive operator
$\mathcal{J}(\mathcal{T}+{\alpha}\mathbf{I})^{-1}$ to conclude  the continuity of $r(\alpha)$ for $\alpha>-\lambda_1$.

Next we prove that $r(\alpha_1)<r(\alpha_2)$ when $\alpha_1> \alpha_2> -\lambda_1$. For $\alpha=\alpha_1$,
 there exists a pair of eigenfunctions $\Phi=(\phi_1,\phi_2)\in D_{+}^0$
such that the following hold:
\begin{equation*}\label{specral_1}
r(\alpha_1)\Phi=\mathcal{J}(\mathcal{T}+\alpha_1\mathbf{I} )^{-1}\Phi.
\end{equation*}
Now we claim that $\alpha_1>\alpha_2$ implies
\begin{equation}\label{specral_2}
(\mathcal{T}+\alpha_1\mathbf{I} )^{-1}\Phi \llp (\mathcal{T}+\alpha_2\mathbf{I} )^{-1}\Phi.
\end{equation}

By the positivity of $\Phi$ and $(\mathcal{T}+\alpha_i\mathbf{I} )^{-1}$ ($i=1,2$), we see that
$\Psi_1:=(\mathcal{T}+\alpha_1\mathbf{I} )^{-1}\Phi$ and $ \Psi_2:=(\mathcal{T}+\alpha_2\mathbf{I} )^{-1}\Phi$
belong to $ D_{+}^0$. Since
\begin{equation*}
\Phi=(\mathcal{T}+\alpha_1\mathbf{I} )\Psi_1=(\mathcal{T}+\alpha_2\mathbf{I} )\Psi_2,
\end{equation*}
we have
\begin{equation*}
\mathbf{0}=(\mathcal{T}+\alpha_2\mathbf{I} )\Psi_2-(\mathcal{T}+\alpha_2\mathbf{I}+(\alpha_1-\alpha_2)\mathbf{I} )\Psi_1
=(\mathcal{T}+\alpha_2 \mathbf{I})(\Psi_2-\Psi_1)-(\alpha_1-\alpha_2)\Psi_1.
\end{equation*}
As $\alpha_1-\alpha_2>0$, we see $(\mathcal{T}+\alpha_2 \mathbf{I})(\Psi_2-\Psi_1)\ggs \mathbf{0}$, which implies
 $\Psi_2\ggs\Psi_1$ since $(\mathcal{T}+\alpha_2I)^{-1}$ is strongly positive. The claim is proved.

By the definition of  $\mathcal{J}$ and \eqref{specral_2}, for $\Phi\in D_{+}^0$,
\begin{equation*}
\mathcal{J}(\mathcal{T}+\alpha_1\mathbf{I})^{-1}\Phi\llp\mathcal{J}(\mathcal{T}+\alpha_2\mathbf{I})^{-1}\Phi.
\end{equation*}
By applying part $(d)$ of Theorem 19.3 in \cite{Deimling1985}, we see that $r(\alpha_1)<r(\alpha_2)$.
\end{proof}
\begin{theorem}\label{PEV_exist}
Problem
\eqref{eigenprobelm1} admits a unique
eigenvalue $\lambda_0$ corresponding to a pair of positive eigenfunctions.
\end{theorem}
\begin{proof}
By Lemma \ref{spectral_J}, there exists  a unique $\alpha_0>-\lambda_{1}$ such that
$r(\alpha_0)=r(\mathcal{J}(\mathcal{T}+\alpha_0\mathbf{I} )^{-1})=1$. Therefore, there exists $\Phi\in D_+^0$ such that
$\mathcal{J}(\mathcal{T}+\alpha_0\mathbf{I} )^{-1}\Phi=\Phi$, and so $\Psi:=(\mathcal{T}+\alpha_0\mathbf{I} )^{-1}\Phi\in D_+^0$
and
\[
(\mathcal{T}+\alpha_0\mathbf{I} )^{-1}\mathcal{J}\Psi=\Psi, \mbox{ which implies } (\mathcal{J}-\mathcal T)\Psi=\alpha_0 \Psi.
\]
Thus $\lambda_0:=\alpha_0$ is an eigenvalue of \eqref{eigenprobelm1} corresponding to
a pair of positive eigenfunctions in $D_{+}^0$.

Conversely, if $\lambda_0$  is an eigenvalue of \eqref{eigenprobelm1} associated with an eigenfunction pair  $\Phi=(\phi_1,\phi_2)\in D_+\setminus\{\bf 0\}$, then
$(\mathcal{J}-\mathcal T)\Phi=\lambda_0 \Phi$. Hence, with $\alpha_0>-\lambda_1$ as given above, we obtain
\[
(\mathcal{T}+\alpha_0\mathbf{I} )^{-1}\mathcal{J}\Phi=\Phi+(\lambda_0-\alpha_0)\Psi,\ \ \Psi:=(\mathcal T+\alpha_0I)^{-1}\Phi\in D_+^0.
\]
Let $D^*$ be the dual space of $D$ and let $V: D^*\rightarrow D^*$ be the conjugate operator of $(\mathcal{T}+\alpha_0\mathbf{I} )^{-1}\mathcal J$.
By the Krein-Rutman Theorem (see, e.g., Theorem 1.1 of \cite{Du2006}), $r(\alpha_0)=1$ is the principal eigenvalue of the operator
$V$  corresponding to a strongly positive eigenvector $\mathcal{E}\in D^*$,
i.e., $
V\mathcal{E}=\mathcal{E}.
$
It follows that
\[
\langle \mathcal{E}, \Phi\rangle=\langle V\mathcal{E}, \Phi\rangle=\langle \mathcal E, (\mathcal{T}+\alpha_0\mathbf{I} )^{-1}\mathcal{J}\Phi\rangle =\langle \mathcal{E}, \Phi\rangle
+(\lambda_0-\alpha_0)\langle \mathcal E, \Psi\rangle.
\]
Hence $(\lambda_0-\alpha_0)\langle \mathcal E, \Psi\rangle=0$, which implies $\lambda_0=\alpha_0$ since $\langle \mathcal E, \Psi\rangle>0$.
\end{proof}

As usual, we will call this $\lambda_0$ the ``principal
eigenvalue'' of \eqref{eigenprobelm1}. Next we give several results on how to estimate $\lambda_0$.


\begin{lemma}\label{upper-eigen-lemma}
 Let $\lambda_0$ be the
principal eigenvalue of \eqref{eigenprobelm1}.
If there exist a pair of  functions $\bm{\varphi}=(\varphi_1,\varphi_2)\in D_{+}$ with $\varphi_1\not\equiv 0, \varphi_2\not\equiv 0$, and a number
$\bar{\lambda}$ such that
\begin{equation}\label{upper_eigenvalue}
\begin{cases}
\displaystyle
\bar{\lambda}\varphi_1\geq d_1\int_{-l}^{l}J_1(x-y)\varphi_1(y)dy-d_1\varphi_1-a_{11}\varphi_1
\displaystyle
+a_{12}\int_{-l}^{l}K(x-y)\varphi_2(y)dy,&x\in[-l,l],\\
\displaystyle\bar{\lambda}\varphi_2\geq d_2\int_{-l}^{l}J_2(x-y)\varphi_2(y)dy-d_2\varphi_2-a_{22}\varphi_2+G'(0)\varphi_1,&x\in[-l,l],\\
\end{cases}
\end{equation}
then $\bar{\lambda}\geq\lambda_0$ holds.  Moreover, $\lambda_0=\bar{\lambda}$ if and only if equalities hold in \eqref{upper_eigenvalue}.
\end{lemma}
\begin{proof}
If the first assertion does not hold, then $\lambda_0>\bar{\lambda}$, and we can fix a constant $\lambda_2$ such that
$\lambda_0>\lambda_2>\max\{-\lambda_1,\bar{\lambda}\}$. By \eqref{upper_eigenvalue}, $\mathcal{J}{\bm{\varphi}}-\mathcal{T}{\bm{\varphi}}\preceq \bar\lambda {\bm{\varphi}}\preceq \lambda_2{\bm{\varphi}}$
and hence
\begin{equation*}
{\bm{\varphi}}\succeq (\mathcal{T}+\lambda_2\mathbf{I} )^{-1}\mathcal J{\bm{\varphi}}.
\end{equation*}

Let $D^*$ be the dual space of $D$ and $V_{\lambda_2}: D^*\rightarrow D^*$ the conjugate operator of $(\mathcal{T}+\lambda_2\mathbf{I} )^{-1}\mathcal J$.
By the Krein-Rutman Theorem (see Theorem 1.1 of \cite{Du2006}), $r(\lambda_2)$ is the principal eigenvalue of the operator
$V_{\lambda_2}$  corresponding to a strongly positive eigenvector $\mathcal{E}_{\lambda_2}\in D^*$: $
V_{\lambda_2}\mathcal{E}_{\lambda_2}=r(\lambda_2)\mathcal{E}_{\lambda_2}.
$
Therefore, due to ${\bm{\varphi}}\in D_{+}\setminus\{\bf 0\}$, we have $\langle\mathcal{E}_{\lambda_2} ,{\bm{\varphi}}\rangle>0$ and
\begin{align*}
r(\lambda_2)\langle \mathcal{E}_{\lambda_2} ,{\bm{\varphi}} \rangle=
\langle V_{\lambda_2}\mathcal{E}_{\lambda_2} ,{\bm{\varphi}} \rangle
=\langle\mathcal{E}_{\lambda_2} ,(\mathcal{T}+\lambda_2\mathbf{I} )^{-1}\mathcal J{\bm{\varphi}} \rangle\leq
\langle\mathcal{E}_{\lambda_2} ,{\bm{\varphi}}\rangle.
\end{align*}
This implies that $r(\lambda_2)\leq 1=r(\lambda_0)$ and thus $\lambda_2\geq\lambda_0$, which is a contradiction
with $\lambda_0>\lambda_2$. This proves $\bar{\lambda}\geq\lambda_0$.

Next, we prove the second assertion.
If the equalities hold in \eqref{upper_eigenvalue}, then clearly $\lambda_0=\bar{\lambda}$ by the uniqueness of the principal eigenvalue. Moreover, $(\varphi_1,\varphi_2)\in D_{+}^0$ is a corresponding positive eigenfunction pair.
It remains to prove that the equalities must hold in \eqref{upper_eigenvalue} if $\bar{\lambda}=\lambda_0$.
Suppose that $\Phi=(\varphi_1,\varphi_2)\in D_{+}\setminus\{\bf 0\}$ satisfies \eqref{upper_eigenvalue}  with $\bar\lambda=\lambda_0$.
Then
\begin{equation}\label{upper_eigenvalue1}
\Psi:=\Phi-(\mathcal{T}+{\lambda_0}\mathbf{I} )^{-1}\mathcal J\Phi\preceq\mathbf{0}.
\end{equation}
With $V$ and $\mathcal E$ defined as in the proof of Theorem \ref{PEV_exist}, we  obtain from the calculation there that
\[
\langle \mathcal E, \Psi\rangle =\langle \mathcal E, \Phi\rangle -\langle \mathcal E, (\mathcal{T}+{\lambda_0}\mathbf{I} )^{-1}\mathcal J\Phi\rangle =0.
\]
Since $\mathcal E$ is strongly positive, it follows that $\Psi={\bf 0}$, and we may now use \eqref{upper_eigenvalue1} to see that
 equalities hold  in \eqref{upper_eigenvalue}.
\end{proof}

\begin{lemma}\label{lower-eigen-lemma}
 Let $\lambda_0$ be the
principal eigenvalue of \eqref{eigenprobelm1}.
If there exist  $\bm{\eta}=(\eta_1,\eta_2)\in D_{+}\setminus \{\bf 0\}$ and
$\underline{\lambda}$ satisfying
\begin{equation}\label{lower-eigen-eq}
\begin{cases}
\displaystyle
\underline{\lambda} \eta_1\leq d_1\int_{-l}^{l}J_1(x-y)\eta_1(y)dy-d_1\eta_1-a_{11}\eta_1
\displaystyle
+a_{12}\int_{-l}^{l}K(x-y)\eta_2(y)dy,&x\in[-l,l],\\
\displaystyle
\underline{\lambda} \eta_2\leq d_2\int_{-l}^{l}J_2(x-y)\eta_2(y)dy-d_2\eta_2-a_{22}\eta_2+G'(0)\eta_1,&x\in[-l,l].\\
\end{cases}
\end{equation}
then $\underline{\lambda}\leq\lambda_0$ holds. Moreover, $\underline{\lambda}=\lambda_0$ if and only
if equalities hold in \eqref{lower-eigen-eq}.
\end{lemma}
\begin{proof}
The proof for this lemma is similar to Lemma \ref{upper-eigen-lemma}. The
first assertion can also be proved by a simpler argument as follows.  Arguing by contradiction, we assume $\underline{\lambda}>\lambda_0$.
By \eqref{lower-eigen-eq} we have
\begin{equation*}
\bm{{\eta}}\preceq(\mathcal{T}+\underline{\lambda}\mathbf{I})^{-1}\mathcal J\bm{{\eta}}, \mbox{ which implies }
\bm{{\eta}}\preceq((\mathcal{T}+\underline{\lambda}\mathbf{I})^{-1}\mathcal J)^n\bm{{\eta}},\ n=1,2,...
\end{equation*}
Combining this with Lemma \ref{spectral_J}, we have, for any integer $n\geq1$,
\begin{equation*}
1\leq\lim_{n\rightarrow\infty}\|\big((\mathcal{T}+\underline{\lambda}
\mathbf{I})^{-1}\mathcal J\big)^n\|^{\frac{1}{n}}
=r(\underline{\lambda})<r(\lambda_0)=1,
\end{equation*}
which is a contradiction.
\end{proof}

\begin{lemma}\label{PEV-l}
 Let $\lambda_0=\lambda_0(l)$ denote the
principal eigenvalue of \eqref{eigenprobelm1} to stress its dependence on $l$. Then $\lambda_0(l)$ is strictly increasing and continuous with respect to $l\in(0,\infty)$.
\end{lemma}

\begin{proof}
We first prove the monotonicity of $\lambda_0(l)$ with respect to $l$.
For $0<l_1<l_2$, let $(\zeta_1,\zeta_2)$ be a  positive eigenfunction pair associated with $\lambda_0(l_2)$; then, for $x\in[-l_1,l_1]$, we have
\begin{align*}
 &\lambda_0(l_2) \zeta_1+a_{11}\zeta_1-a_{12}\int_{-l_1}^{l_1}K(x-y)\zeta_2(y)dy\\
\geq, &\not\equiv\lambda_0(l_2) \zeta_1+a_{11}\zeta_1-a_{12}\int_{-l_2}^{l_2}K(x-y)\zeta_2(y)dy=\displaystyle d_1\int_{-l_2}^{l_2}J_1(x-y)\zeta_1(y)dy-d_1\zeta_1\\
\geq, &\not\equiv \displaystyle d_1\int_{-l_1}^{l_1}J_1(x-y)\zeta_1(y)dy-d_1\zeta_1.
\end{align*}
Similarly, for $x\in[-l_1,l_1]$, we have
\begin{align*}
 &\lambda_0(l_2) \zeta_2+a_{22}\zeta_2-G'(0)\zeta_1=\displaystyle d_2\int_{-l_2}^{l_2}J_2(x-y)\zeta_2(y)dy-d_2\zeta_2\\
\geq, &\not\equiv \displaystyle d_2\int_{-l_1}^{l_1}J_2(x-y)\zeta_2(y)dy-d_2\zeta_2.
\end{align*}
The above two inequalities imply that $\lambda_0(l_2)>\lambda_0(l_1)$ by applying Lemma \ref{upper-eigen-lemma}.

Next we prove the continuity of $\lambda_0(l)$ with respect to $l>0$. To achieve this, we will show that, for any $l>0$ and small $\epsilon>0$, there exists some constant $\delta>0$ such that
\begin{equation*}\label{continuity_inequality}
|\lambda_0(l)-\lambda_0(l_0)|<\epsilon \mbox{ if }|l_0-l|<\delta.
\end{equation*}

We first discuss the case when $l_0\in(l-\delta,l]$. By the above-proved monotonicity,
\begin{equation*}
\lambda_0(l_0)\leq\lambda_0(l)<\lambda_0(l)+\epsilon
\end{equation*}
for any $\epsilon>0$. So it suffices to find  $\delta$ such that
$\lambda_0(l)-\epsilon<\lambda_0(l_0)$ for $l_0\in(l-\delta,l]$.

We denote by $(\theta_1,\theta_2)$ a  positive eigenfunction pair associated with $\lambda_0(l)$. Define
\begin{equation*}\label{C_MC_m}
C_m:=\min\{\min_{x\in[-l,l]}\theta_1(x),\min_{x\in[-l,l]}\theta_2(x)\},
C_M:=\max\{\max_{x\in[-l,l]}\theta_1(x),\max_{x\in[-l,l]}\theta_2(x)\}.\\
\end{equation*}
For $x\in[-l_0,l_0]$, we have
\begin{align*}
&\displaystyle d_1\int_{-l}^{l}J_1(x-y)\theta_1(y)dy\\
=&\ d_1\int_{-l_0}^{l_0}J_1(x-y)\theta_1(y)dy+d_1\Big
[\int^{l}_{l_0}+\int^{-l_0}_{-l}\Big]J_1(x-y)\theta_1(y)dy\\
\leq &\ d_1\int_{-l_0}^{l_0}J_1(x-y)\theta_1(y)dy+2d_1\|J_1\|_{\infty}C_M(l-l_0)\\
\leq &\ d_1\int_{-l_0}^{l_0}J_1(x-y)\theta_1(y)dy+\frac{2d_1\|J_1\|_{\infty}C_M(l-l_0)}{C_m}\theta_1(x),
\end{align*}
and similarly
\begin{align*}
\displaystyle a_{12}\int_{-l}^{l}K(x-y)\theta_2(y)dy
\leq a_{12}\int_{-l_0}^{l_0}K(x-y)\theta_2(y)dy+\frac{2a_{12}\|K\|_{\infty}C_M(l-l_0)}{C_m}\theta_1(x).
\end{align*}
Therefore, for $x\in[-l_0,l_0]$, we have
\begin{align*} &-a_{12}\int_{-l_0}^{l_0}K(x-y)\theta_2(y)dy-\frac{2a_{12}\|K\|_{\infty}C_M(l-l_0)}{C_m}\theta_1(x)
+a_{11}\theta_1(x)+\lambda_0(l)\theta_1(x)\\
\leq &\ -a_{12}\int_{-l}^{l}K(x-y)\theta_2(y)dy
+a_{11}\theta_1(x)+\lambda_0(l)\theta_1(x)\\
=&\ d_1\int_{-l}^{l}J_1(x-y)\theta_1(y)dy-
d_1\theta_1(x)\\
\leq &\ d_1\int_{-l_0}^{l_0}J_1(x-y)\theta_1(y)dy-
d_1\theta_1(x)+\frac{2d_1\|J_1\|_{\infty}C_M(l-l_0)}{C_m}\theta_1(x),
\end{align*}
which is equivalent to
\begin{align*}
\big[\lambda_0(l)-\tilde{C}_1(l-l_0)\big]\theta_1
\leq d_1\int_{-l_0}^{l_0}J_1(x-y)\theta_1(y)dy-d_1\theta_1-a_{11}\theta_1+a_{12}\int_{-l_0}^{l_0}K(x-y)\theta_2(y)dy
,
\end{align*}
where
\begin{align*}
\tilde{C}_1:=\frac{2(a_{12}\|K\|_{\infty}+d_1\|J_1\|_{\infty})C_M}{C_m}.
\end{align*}

Similarly, we have
\begin{align*}
\big[\lambda_0(l)-\tilde{C}_2(l-l_0)\big]\theta_2
\leq d_2\int_{-l_0}^{l_0}J_2(x-y)\theta_2(y)dy-a_{22}\theta_2+G'(0)\theta_1
\end{align*}
with \begin{equation*}
\tilde{C}_2:=\frac{2d_2\|J_2\|_{\infty}C_M}{C_m}.
\end{equation*}
Applying Lemma \ref{lower-eigen-lemma}, we have, with $C_*:=\max\{\tilde{C}_1,\tilde{C}_2\}$,
\begin{align*}
\lambda_0(l_0)\geq\lambda_0(l)-C^*(l-l_0)>\lambda_0(l)-\epsilon
\end{align*}
provided that
\begin{align*}
l-l_0<\delta:=\frac{\epsilon}{C_*}.
\end{align*}

Finally, we consider the case when $l_0\in[l,l+\delta)$. It suffices to show that $\lambda_0(l_0)<\lambda_0(l)+\epsilon$, since
$\lambda_0(l_0)\geq\lambda_0(l)> \lambda_0(l)-\epsilon$  by the monotonicity of $\lambda_0(l)$.
Let $(\theta_1,\theta_2)$ be a  positive eigenfunction associated with
$\lambda_0(l)$.
Define
\begin{equation*}
(\hat{\theta}_1(x),\hat{\theta}_2(x)):=\begin{cases}
(\theta_1(x),\theta_2(x)),&|x|\leq l,\\
(\theta_1(l),\theta_2(l)),&x> l,\\
(\theta_1(-l),\theta_2(-l)),&x<- l,\\
\end{cases}
\end{equation*}
and
\begin{equation*}\label{hat_C_MC_m}
\hat{C}_m:=\min\{\min_{x\in[-l_0,l_0]}\hat{\theta}_1(x),\min_{x\in[-l_0,l_0]}\hat{\theta}_2(x)\},
\hat{C}_M:=\max\{\max_{x\in[-l_0,l_0]}\hat{\theta}_1(x),\max_{x\in[-l_0,l_0]}\hat{\theta}_2(x)\}.\\
\end{equation*}

Similar to the first case, for $x\in[-l,l]$, we have
\begin{equation}\label{Continuity_case2_1}
\begin{aligned}
&\displaystyle d_1\int_{-l}^{l}J_1(x-y)\theta_1(y)dy\\
=&\ d_1\int_{-l_0}^{l_0}J_1(x-y)\hat{\theta}_1(y)dy
-d_1\Big[\int^{l_0}_{l}+\int^{-l}_{-l_0}\Big]J_1(x-y)\hat{\theta}_1(y)dy\\
\geq &\ d_1\int_{-l_0}^{l_0}J_1(x-y)\hat{\theta}_1(y)dy-2d_1\|J_1\|_{\infty}\hat{C}_M(l_0-l)\\
\geq &\ d_1\int_{-l_0}^{l_0}J_1(x-y)\hat{\theta}_1(y)dy-\frac{2d_1\|J_1\|_{\infty}\hat{C}_M(l_0-l)}{\hat{C}_m}\hat{\theta}_1,
\end{aligned}
\end{equation}
and
\begin{equation}\label{Continuity_case2_2}
\begin{aligned}
&a_{12}\int_{-l}^{l}K(x-y)\theta_2(y)dy\\
=&\ a_{12}\int_{-l_0}^{l_0}K(x-y)\hat{\theta}_2(y)dy -a_{12}\Big[\int^{l_0}_{l}+\int^{-l}_{-l_0}\Big]K(x-y)\hat{\theta}_2(y)dy\\
\geq &\ a_{12}\int_{-l_0}^{l_0}K(x-y)\hat{\theta}_2(y)dy-\frac{2a_{12}\|K\|_{\infty}\hat{C}_M(l_0-l)}{\hat{C}_m}\hat{\theta}_1
\end{aligned}
\end{equation}
Therefore, for $x\in[-l,l]$, we have
\begin{equation*}
\big[\lambda_0(l)+\hat{C}_1(l_0-l)\big]\hat{\theta}_1
\geq d_1\int_{-l_0}^{l_0}J_1(x-y)\hat{\theta}_1(y)dy-d_1\hat{\theta}_1-a_{11}\hat{\theta}_1+a_{12}\int_{-l_0}^{l_0}K(x-y)
\hat{\theta}_2(y)dy
\end{equation*}
with \begin{align*}
\hat{C}_1:=\frac{2(a_{12}\|K\|_{\infty}+d_1\|J_1\|_{\infty})\hat{C}_M}{\hat{C}_m},
\end{align*}
and
 \begin{align*}
\big[\lambda_0(l)+\hat{C}_2(l_0-l)\big]\hat{\theta}_2
\geq d_2\int_{-l_0}^{l_0}J_2(x-y)\hat{\theta}_2(y)dy-a_{22}\hat{\theta}_2+G'(0)\hat{\theta}_1
\end{align*}
with \begin{align*}
\hat{C}_2:=\frac{2d_2\|J_2\|_{\infty}\hat{C}_M}{\hat{C}_m}.
\end{align*}
Moreover, we define, for $x\in \mathbb{R}$
\begin{equation*}
\begin{aligned}
&\phi_1(x):=d_1\int_{-l}^{l}J_1(x-y)\hat{\theta}_1(y)dy-d_1\hat{\theta}_1(x)-
a_{11}\hat{\theta}_1(x)+a_{12}\int_{-l}^{l}K(x-y)\hat{\theta}_2(y)dy,\\
&\phi_2(x):=d_2\int_{-l}^{l}J_1(x-y)\hat{\theta}_2(y)dy-d_2\hat{\theta}_2(x)+G'(0)\hat{\theta}_1(x)-
a_{22}\hat{\theta}_2(x).\\
\end{aligned}
\end{equation*}
By the continuity of $(\hat{\theta}_1,\hat{\theta}_2)$ in $x$, we see that $\phi_1(x)$ and $\phi_2(x)$
are continuous in $x$. Moreover, for
$x\in[-l,l]$,
\begin{equation*}
\phi_1(x)=\lambda_0(l)\theta_1(x),~~\phi_2(x)=\lambda_0(l)\theta_2(x).
\end{equation*}
Therefore there exists some $\delta_0>0$ such that the following holds for $x\in[-l-\delta_0,-l]\cup [l,l+\delta_0]$:
\begin{equation*}
|\phi_1(x)-\lambda_0(l)\hat{\theta}_1(x)|<\frac{\hat{C}_m\epsilon}{4},~~|\phi_2(x)-
\lambda_0(l)\hat{\theta}_2(x)|<\frac{\hat{C}_m\epsilon}{4}.
\end{equation*}
Thus, for $x\in[-l-\delta_0,-l]\cup [l,l+\delta_0]$, by \eqref{Continuity_case2_1} and \eqref{Continuity_case2_2}, we have
\begin{equation*}
\begin{aligned}
&\big(\lambda_0(l)+\frac{\epsilon}{4}\big)\hat{\theta}_1(x)\geq\lambda_0(l)\hat{\theta}_1(x)+\frac{\hat{C}_m\epsilon}{4}\geq  \phi_1(x)\\
=&\ d_1\int_{-l}^{l}J_1(x-y)\hat{\theta}_1(y)dy-d_1\hat{\theta}_1-
a_{11}\hat{\theta}_1+a_{12}\int_{-l}^{l}K(x-y)\hat{\theta}_2(y)dy\\
\geq
&\ d_1\int_{-l_0}^{l_0}J_1(x-y)\hat{\theta}_1(y)dy-d_1\hat{\theta}_1-
a_{11}\hat{\theta}_1+a_{12}\int_{-l_0}^{l_0}K(x-y)\hat{\theta}_2(y)dy-\hat{C}_1(l_0-l)\hat{\theta}_1,\\
\end{aligned}
\end{equation*}
which implies that for $x\in[-l-\delta_0,l+\delta_0]$,
\begin{equation*}
\begin{aligned}
&\big(\lambda_0(l)+\frac{\epsilon}{4}+\hat{C}_1(l_0-l)\big)\hat{\theta}_1\\
\geq&
d_1\int_{-l_0}^{l_0}J_1(x-y)\hat{\theta}_1(y)dy-d_1\hat{\theta}_1-
a_{11}\hat{\theta}_1+a_{12}\int_{-l_0}^{l_0}K(x-y)\hat{\theta}_2(y)dy.\\
\end{aligned}
\end{equation*}
Analogously, we have, for  $x\in[-l-\delta_0,l+\delta_0]$,
\begin{equation*}
\begin{aligned}
\big(\lambda_0(l)+\frac{\epsilon}{4}+\hat{C}_2(l_0-l)\big)\hat{\theta}_2
\geq
d_2\int_{-l_0}^{l_0}J_2(x-y)\hat{\theta}_2(y)dy-d_2\hat{\theta}_2-a_{22}\hat{\theta}_2+G'(0)\hat{\theta}_1
.\\
\end{aligned}
\end{equation*}
By applying Lemma \ref{upper-eigen-lemma}, we obtain
\begin{equation*}
\lambda_0(l_0)\leq \lambda_0(l)+\frac{\epsilon}{4}+\hat{C}_*(l_0-l) <\lambda_0(l)+\epsilon
\end{equation*}
provided that
\begin{equation*}
l_0-l<\delta:=\min\{\delta_0,\frac{3\epsilon}{4\hat{C}_*}\}\mbox{ and }\hat{C}_*:=\max\{\hat{C}_1,\hat{C}_2\}.
\end{equation*}
We have now proved the desired continuity.
\end{proof}

\begin{lemma}\label{eigen_proposition}
The principal eigenvalue $\lambda_0(l)$ of \eqref{eigenprobelm1} has the following properties:
\begin{itemize}
  \item [$(i)$] If $R_0\leq1$, then $\lambda_0(l)<0$ for every $l>0$;
   \item [$(ii)$] If $R_0>1$, then
    there exists a unique $l^*\in(0,\infty)$ such that
       $\lambda_0(l^*)=0.$
\end{itemize}

\end{lemma}
\begin{proof}
$(i)$ Since $R_0\leq1$, there exists some positive constant $H_1$ such that
\begin{equation*}
\frac{G'(0)}{a_{22}}\leq H_1\leq \frac{a_{11}}{a_{12}}.
\end{equation*}
It is easily checked that $(\varphi_1,\varphi_2)=(1,H_1)$ and $\bar{\lambda}=0$ satisfy \eqref{upper_eigenvalue} for any $l>0$. Moreover,
for $x\in[-l,l]$,
\begin{equation*}
d_1\int^l_{-l}J_1(x-y)dy-d_1-a_{11}+a_{12}\int^l_{-l}K(x-y)H_1dy\leq,\not\equiv -a_{11}+a_{12}H_1\leq0,
\end{equation*}
and
\begin{equation*}
d_2\int^l_{-l}J_2(x-y)H_1dy-d_2H_1-a_{22}H_1+G'(0)\leq,\not\equiv -a_{22}H_1+G'(0)\leq0.
\end{equation*}
By Lemma \ref{upper-eigen-lemma}, we have $\lambda_0(l)<0$.

 $(ii)$ By Lemma \ref{PEV-l}, it suffices to show that  there exist some $l_1<l_2$ such that $\lambda_0(l_1)<0<\lambda_0(l_2)$.
We first find some constant $H_2$ such that
\begin{equation*}
\frac{G'(0)}{a_{22}+d_2}< H_2< \infty.
\end{equation*}
 Then define
 \begin{equation*}
 \sigma_0:=H_2-\frac{G'(0)}{a_{22}+d_2}.
\end{equation*}
Next we fix $l_1>0$  small such that, for $x\in [-l_1, l_1]$,
\[
\begin{cases}
l_1<\min\Big\{\frac{d_1+a_{11}}{8d_1\|J_1\|_{\infty}},\frac{d_1+a_{11}}{8a_{12}\|K\|_{\infty}H_2},
\frac{(d_2+a_{22})\sigma_0}{4d_2\|J_2\|_{\infty}H_2}\Big\} \\
 d_1\int^{l_1}_{-l_1}J_1(x-y)dy<\frac{d_1+a_{11}}{4},\\
 a_{12}\int^{l_1}_{-l_1}K(x-y)H_2dy<\frac{d_1+a_{11}}{4},\\
 d_2\int^{l_1}_{-l_1}J_2(x-y)H_2dy<(d_2+a_{22})\frac{\sigma_0}{2}.
\end{cases}
\]
If we choose $(\varphi_1,\varphi_2)=(1,H_2)$, we see that \eqref{upper_eigenvalue} holds with $\bar{\lambda}=0$ and $l=l_1$, i.e.,
   \begin{align*}
 d_1\int^{l_1}_{-l_1}J_1(x-y)dy-(d_1+a_{11})+a_{12}\int^{l_1}_{-l_1}K(x-y)H_2dy
 <\frac{d_1+a_{11}}{2}-(d_1+a_{11})<0,
\end{align*}
 and
    \begin{align*}
 &d_2\int^{l_1}_{-l_1}J_2(x-y)H_2dy-(d_2+a_{22})H_2+G'(0)
 <(d_2+a_{22})\frac{\sigma_0}{2}+(d_2+a_{22})\Big(-H_2+\frac{G'(0)}{d_2+a_{22}}\Big)\\
=&\ (d_2+a_{22})\Big(\frac{\sigma_0}{2}-\sigma_0\Big)<0.
\end{align*}
 By applying Lemma \ref{upper-eigen-lemma}, we have $\lambda_0(l_1)<0$.

Next we will seek for a large $l_2>0$ such that $\lambda_0(l_2)>0$. To achieve this, we first define
 \begin{equation*}
\theta_{l}(x):=\max\{l-|x|,0\},~x\in \R.
\end{equation*}
We claim that for any small $\epsilon>0$,
 there exists a large enough $l_{\epsilon}>0$ such that
 \begin{equation}\label{key_l_2}
 \begin{cases}
\int^{l}_{-l}J_i(x-y) \theta_{l}(y)dy-\theta_{l}(x)>-\epsilon\theta_{l}(x)~
\mbox{ for }x\in [-l,l], l\geq l_{\epsilon},i=1,2;\\
\int^{l}_{-l}K(x-y) \theta_{l}(y)dy-\theta_{l}(x)>-\epsilon\theta_{l}(x)~
\mbox{ for }x\in [-l,l], l\geq l_{\epsilon}.
\end{cases}
\end{equation}
Before proving this claim, let us see how it can be used  to define some large $l_2$ such that $\lambda_0(l_2)>0$.
Firstly, due to $R_0>1$,
we can choose some positive constant $H_3$ such that
\begin{equation*}
\frac{a_{11}}{a_{12}}< H_3<\frac{G'(0)}{a_{22}} .
\end{equation*}
Define
\begin{equation*}
\tilde{\sigma}_0:=\min\Big\{H_3-\frac{a_{11}}{a_{12}},\frac{G'(0)}{a_{22}}-H_3\Big\}.
\end{equation*}
By \eqref{key_l_2}, there exists some large constant $l_2>0$ such that for $l\geq l_2$,
 \begin{align*}
&\int^{l}_{-l}J_i(x-y) \theta_{l}(y)dy-\theta_{l}(x)>
-\min\Big\{\frac{a_{12}\tilde{\sigma}_0}{4d_1},\frac{a_{22}\tilde{\sigma}_0}{4d_2H_3}\Big\}\theta_{l}(x),
\mbox{ for }x\in [-l,l],i=1,2;\\
&\int^{l}_{-l}K(x-y) \theta_{l}(y)dy>\theta_{l}(x)-\frac{\tilde{\sigma}_0}{4H_3}\theta_{l}(x),~
\mbox{ for }x\in [-l,l].
\end{align*}
 Then   for $x\in(-l,l)$ with $l\geq l_2$, we obtain
 \begin{align*}
&\ d_1\int^{l}_{-l}J_1(x-y) \theta_{l}(y)dy-d_1\theta_{l}(x)-a_{11}\theta_{l}(x)+
a_{12}H_3\int^{l}_{-l}K(x-y) \theta_{l}(y)dy\\
>&\ -\frac{a_{12}\tilde{\sigma}_0}{4} \theta_{l}(x)-\frac{a_{12}\tilde{\sigma}_0}{4}\theta_{l}(x)+a_{12}\theta_{l}(x)\Big(H_3-\frac{a_{11}}{a_{12}}\Big)
\\
\geq&\  a_{12}\theta_{l}(x)\big(-\frac{\tilde{\sigma}_0}{2}+\tilde{\sigma}_0\big)>0,
\end{align*}
and
 \begin{align*}
&d_2H_3\int^{l}_{-l}J_2(x-y) \theta_{l}(y)dy-d_2H_3\theta_{l}(x)-a_{22}H_3\theta_{l}(x)+
G'(0)\theta_{l}(x)\\
>&-\frac{a_{22}\tilde{\sigma}_0}{4} \theta_{l}(x)+a_{22}\theta_{l}(x)\Big(-H_3+\frac{G'(0)}{a_{22}}\Big)
\\
\geq&\ a_{22}\theta_{l}(x)\big(-\frac{\tilde{\sigma}_0}{4}+\tilde{\sigma}_0\big)>0,
\end{align*}
Hence we can apply Lemma \ref{lower-eigen-lemma} with $(\eta_1,\eta_2):=
(\theta_{l},H_3\theta_{l})$ to obtain $\lambda_0(l_2)>0$.

To complete the proof, it remains to prove that \eqref{key_l_2} holds. We only prove the first assertion
in \eqref{key_l_2}, since the argument also works  for the second one.

Since $\displaystyle\int^{\infty}_{-\infty}J_i(y)dy=1$ for $i=1,2$, for any small $\epsilon>0$, there exists some $l_0=l_0(\epsilon)>0$ such that
\begin{equation}\label{J_i}
\int^{l_0}_{-l_0}J_i(y)dy\geq 1-\epsilon/2, ~i=1,2.
\end{equation}

Now fix $l$  such that $l\gg l_0$. We discuss the following three cases separately.

\textbf{ Case 1:} $x\in[-l_0,l_0]$.

By the choice of $l$, we have
$l-x> l_0$ and $-l-x < -l_0$ for all $x\in[-l_0,l_0]$. Therefore,
\begin{align*}
&\int^{l}_{-l}J_i(x-y)\theta_l(y)dy-\theta_l(x)=\int^{l-x}_{-l-x}J_i(y)\theta_l(y+x)dy-\theta_l(x)\\
\geq&\int^{l_0}_{-l_0}J_i(y)\theta_l(y+x)dy-\theta_l(x).
\end{align*}
When $y\in[-l_0,l_0]$, $y+x\in[-2l_0,2l_0]$ and thus $\theta_l(y+x)=l-|y+x|\geq l-2l_0$. By \eqref{J_i}, we see
\begin{align*}
\int^{l_0}_{-l_0}J_i(y)\theta_l(y+x)dy-\theta_l(x)\geq (1-\epsilon/2)(l-2l_0)-l
\geq \frac{-l\epsilon/2-2l_0}{l-l_0}\theta_l(x)>-\epsilon\theta_l(x),
\end{align*}
when $l$ is sufficiently large.

\textbf{ Case 2:} $x\in[l_0,l-l_0]\cup[-l+l_0,-l_0]$.

We first prove that \eqref{key_l_2} holds for $x\in[l_0,l-l_0]$. Since $l-x\geq l_0$ and $-l-x\leq-l_0$,
we see that
\begin{align*}
&\int^{l}_{-l}J_i(x-y)\theta_l(y)dy-\theta_l(x)=\int^{l-x}_{-l-x}J_i(y)\theta_l(y+x)dy-\theta_l(x)\\
\geq&\int^{l_0}_{-l_0}J_i(y)\theta_l(y+x)dy-\theta_l(x).
\end{align*}
Note that, for $x\in[l_0,l-l_0]$ and $y\in[-l_0,l_0]$, $y+x\in[0,l]$. Thus we see $\theta_l(y+x)=l-(y+x)$,
and
\begin{align*}
&\int^{l_0}_{-l_0}J_i(y)\theta_l(y+x)dy-\theta_l(x)=\int^{l_0}_{-l_0}J_i(y)(l-(y+x))dy-(l-x)\\
\geq &-\int^{l_0}_{-l_0}J_i(y)ydy+(1-\epsilon/2)(l-x)-(l-x)= -\epsilon/2 (l-x)>-\epsilon\theta_l(x),
\end{align*}
where $\displaystyle -\int^{l_0}_{-l_0}J_i(y)ydy=0$ since $J_i(y)y$ is an odd function.
Since $J_i(x)$ and $\theta_l(x)$ are both symmetric in $\R$, \eqref{key_l_2} also holds for $x\in[-l+l_0,-l_0]$.

\textbf{Case 3:} $x\in[l-l_0,l]\cup[-l,-l+l_0]$.

Similar to \textbf{Case 2}, we only discuss $x\in[l-l_0,l]$. By $l-x\in[0,l_0]$ and $-l-x\in[-2l,-2l+l_0]$,
for $l\gg l_0$,
we have
\begin{align*}
&\int^{l}_{-l}J_i(x-y)\theta_l(y)dy-\theta_l(x)=\int^{l-x}_{-l-x}J_i(y)\theta_l(y+x)dy-\theta_l(x)\\
\geq&\int^{l_0}_{-l_0}J_i(y)\theta_l(y+x)dy-\int^{l_0}_{l-x}J_i(y)\theta_l(y+x)dy-\theta_l(x).
\end{align*}

For $x\in[l-l_0,l]$, when $y\in[-l_0,l_0]$, $y+x\geq l-2l_0$ and $\theta_l(y+x)\geq l-(y+x)$, while for
$y\in[l-x,l_0]$, $y+x\geq l$ and $\theta_l(y+x)=0$, therefore, by the above estimates,
\begin{align*}
&\int^{l}_{-l}J_i(x-y)\theta_l(y)dy-\theta_l(x)
\geq\int^{l_0}_{-l_0}J_i(y)\theta_l(y+x)dy-\theta_l(x)\\
\geq&\int^{l_0}_{-l_0}J_i(y)(l-y-x)dy-(l-x)\geq -\epsilon/2 (l-x)>-\epsilon\theta_l(x).
\end{align*}
The proof is completed.
\end{proof}
\subsection{The long-time dynamics of \eqref{model_fixbound}}
 Replacing
the time-dependent interval $[g(t),h(t)]$ with a fixed
interval $[-l,l]$,
it follows from  the proof of Lemma \ref{solution_uv} that \eqref{model_fixbound} admits a unique positive
solution for all $t>0$ when $(u_0, v_0)\in D_+\setminus\{\bf 0\}$. In what follows, we always use $(u(x,t), v(x,t))$ to denote this unique solution.

The steady state of problem \eqref{model_fixbound} satisfies
\begin{equation}\label{model_steady}
\begin{cases}\displaystyle
d_1\int_{-l}^{l}J_1(x-y)w(y)dy-(d_1+a_{11})w(x)
+a_{12}\int_{-l}^{l}K(x-y)z(y)dy=0,&x\in[-l,l],\\
\displaystyle
d_2\int_{-l}^{l}J_2(x-y)z(y)dy-(d_2+a_{22})z(x)+G(w(x))=0,&x\in[-l,l].
\end{cases}
\end{equation}

\begin{theorem}\label{0steadystate}
Suppose that $J_1, J_2$ and $K$ satisfy {\bf (J)}, and $G$ satisfies $(G1)$-$(G2)$.
Denote by $\lambda_0(l)$  the principal eigenvalue of \eqref{eigenprobelm1}.
Then the following statements are valid:
\begin{itemize}
  \item [$(i)$] If $\lambda_0(l)>0$, then problem \eqref{model_steady} admits a unique
  positive solution $(w,z)\in D$. Moreover, $(u,v)$ converges to $(w,z)$
  as $t$ goes to infinity uniformly for $x\in[-l,l]$.
   \item [$(ii)$] If $\lambda_0(l)\leq 0$, then $(0,0)$ is the only nonnegative solution of \eqref{model_steady}
and $(u,v)$ converges to $(0,0)$
  as $t$ goes to infinity uniformly for $x\in[-l,l]$.
\end{itemize}
\end{theorem}

\begin{proof}
$(i)$ By Lemma 3.9, $\lambda_0(l)>0$ implies $R_0>1$.
By (G1) and (G2) for all $M_1\gg1$ we have
\[
G(M_1)/M_1<a_{11}a_{22}/a_{12}.
\]
Let $M_2:=M_1 a_{11}/a_{12}$.
Then clearly
\begin{equation}\label{M_1M_2}
a_{12}M_2= a_{11}M_1\mbox{ and } G(M_1)\leq a_{22}M_2.
\end{equation}
This implies that $(w,z)=(M_1, M_2)$ is an upper solution of \eqref{model_steady}.

Denote by $(\hat{u},\hat{v})$ the unique positive
solution of \eqref{model_fixbound} with initial functions $(M_1,M_2)$.
By \eqref{M_1M_2} and Remark \ref{compariprinciple1}, we have
\begin{equation*}
(\hat{u}(x,t),\hat{v}(x,t))\preceq (M_1,M_2) \mbox{ for all }t>0,~x\in[-l,l].
\end{equation*}
Therefore, for any fixed $t_0>0$,
\begin{equation*}
(\hat{u}(x,t),\hat{v}(x,t))\preceq (\hat{u}(x,t-t_0),\hat{v}(x,t-t_0)) \mbox{ for }t>t_0.
\end{equation*}
This implies that $(\hat{u}(x,t),\hat{v}(x,t))$ is nonincreasing in $t$.
We next show that $(\epsilon\theta_1,\epsilon\theta_2)$ is a lower solution of problem \eqref{model_steady}
for sufficiently small $\epsilon>0$,
where $(\theta_1,\theta_2)$ is a positive eigenfunction pair of \eqref{eigenprobelm1} with $\lambda=\lambda_0(l)$.
Due to $\lambda_0(l)>0$, we have
\begin{align*}
d_1\int_{-l}^{l}J_1(x-y)\epsilon\theta_1(y)dy-d_1\epsilon\theta_1-a_{11}\epsilon\theta_1
+a_{12}\int_{-l}^{l}K(x-y)\epsilon\theta_2(y)dy=\epsilon\lambda_0(l)\theta_1>0,
\end{align*}
and
\begin{align*}
&d_2\int_{-l}^{l}J_2(x-y)\epsilon\theta_2(y)dy-d_2\epsilon\theta_2-a_{22}\epsilon\theta_2
+G(\epsilon\theta_1)\\
=&\ \epsilon\lambda_0(l)\theta_2-G'(0)\epsilon\theta_1+G(\epsilon\theta_1)
= \epsilon\lambda_0(l)\theta_2+o(\epsilon)\theta_1>0
\end{align*}
for all small $\epsilon>0$, where $o(\epsilon)/\epsilon\rightarrow 0$ as $\epsilon\rightarrow 0$.
Choosing  $\epsilon>0$ small enough such that the above inequalities hold (which means that $(\epsilon\theta_1,\epsilon\theta_2)$ is a lower solution of \eqref{model_steady}) and $(\epsilon\theta_1,\epsilon\theta_2)\llp (M_1,M_2)$.
By the comparison principle, we have
$(\hat{u}(x,t),\hat{v}(x,t))\succeq(\epsilon\theta_1,\epsilon\theta_2)$ for $t>0$ and $x\in[-l,l]$. Thus we see that
\begin{equation}\label{hatwz}
(0,0)\llp(w_1(x),z_1(x)):=\lim_{t\rightarrow\infty}(\hat{u}(x,t),\hat{v}(x,t))\mbox{ for }x\in[-l,l].
\end{equation}

Now we claim that $(w_1,z_1)$ is a positive solution pair of problem \eqref{model_steady}.
For any $t,s>0$ and $x\in[-l,l]$, we have
\begin{align*}
&\hat{u}(x,t+s)-\hat{u}(x,s)\\
=&\int^{t}_0\Big(d_1\int_{-l}^{l}J_1(x-y)\hat{u}(y,\tau+s)dy-
(d_1+a_{11})\hat{u}(x,\tau+s)
+a_{12}\int_{-l}^{l}K(x-y)\hat{v}(y,\tau+s)dy\Big)d\tau,
\end{align*}
and
\begin{align*}
&\hat{v}(x,t+s)-\hat{v}(x,s)\\
=&\int^{t}_0\Big(d_2\int_{-l}^{l}J_2(x-y)\hat{v}(y,\tau+s)dy-
(d_2+a_{22})\hat{v}(x,\tau+s)
+G(\hat{u}(x,\tau+s))\Big)d\tau.
\end{align*}
Letting $s\rightarrow \infty$ in both sides of the above equalities, we obtain,
 for $x\in[-l,l]$,
\begin{equation*}
\begin{cases}
\displaystyle0=t\Big[d_1\int_{-l}^{l}J_1(x-y)w_1(y)dy-
(d_1+a_{11})w_1(x)
+a_{12}\int_{-l}^{l}K(x-y)z_1(y)dy\Big],\\
\displaystyle0=t\Big[d_2\int_{-l}^{l}J_2(x-y)z_1(y)dy-
(d_2+a_{22})z_1(x)+
G(w_1(x))\Big].
\end{cases}
\end{equation*}
Due to the arbitrariness of $t>0$,
we see that $(w_1,z_1)$ solves \eqref{model_steady}, from which we also easily see that they are  continuous in $x\in[-l,l]$. By Dini's theorem, this in turn implies that
the monotone convergence in \eqref{hatwz} is uniform for $x\in [-l, l]$.

 Next we show that problem \eqref{model_steady} admits a unique positive solution. Let $(w,z)$ be an arbitrary positive solution of \eqref{model_steady}.
 By enlarging $M_1$ in the earlier choice if necessary, we may assume $(w,z)\preceq(M_1,M_2)$ and $(u(\cdot, 1), v(\cdot, 1))\preceq (M_1, M_2)$.
By the comparison principle, we have
$(w,z) \preceq(\hat{u},\hat{v}) $ for all $t>0$ and $x\in[-l,l]$. Thus $(w,z)\preceq (w_1,z_1)$ in $[-l,l]$.
Since $(w,z)\in D_+^0$,
\begin{equation*}
\zeta_0:=\inf\{\zeta\geq 1:\zeta(w(x),z(x))\succeq(w_1(x),z_1(x)) \mbox{ for }x\in[-l,l]\}
\end{equation*}
is well-defined and finite.
Moreover, we have $\zeta_0(w(x),z(x))\succeq (w_1(x),z_1(x))$ for $x\in[-l,l]$, and there exists some $x_0\in [-l,l]$ such that
\begin{equation*}
{\rm (i)\ } \zeta_0w(x_0)= w_1(x_0) \mbox{ or } {\rm (ii)\  }\zeta_0z(x_0)= z_1(x_0).
\end{equation*}
We claim $\zeta_0=1$. If not, $\zeta_0>1$. By assumption
$(G2)$, we see that $\zeta_0G(w)>G(\zeta_0w)$ in $[-l,l]$. If (ii) happens, then
\begin{align*}
0=&\ d_2\int^{l}_{-l}J_2(x_0-y)(\zeta_0z(y)-z_1(y))dy-(d_2+a_{22})(\zeta_0z(x_0)-z_1(x_0))\\
&+\zeta_0G(w(x_0))-G(w_1(x_0))\\
>&\ d_2\int^{l}_{-l}J_2(x_0-y)(\zeta_0z(y)-z_1(y))dy+G(\zeta_0w(x_0))-G(w_1(x_0))\geq0,
\end{align*}
which is impossible. If (i) happens,
then
\begin{align*}
0=&\ d_1\int^{l}_{-l}J_1(x_0-y)(\zeta_0w(y)-w_1(y))dy-(d_1+a_{11})(\zeta_0w(x_0)-w_1(x_0))\\
&+ a_{12}\int_{-l}^l K(x_0-y)(\zeta_0z(y)-z_1(y))dy\\
\geq &\  a_{12}\int_{-l}^l K(x_0-y)(\zeta_0z(y)-z_1(y))dy.
\end{align*}
Since $K(0)>0$ and $\zeta_0z(y)-z_1(y)\geq 0$ for all $y\in [-l,l]$, the above inequality can hold only if $\zeta_0z(y)-z_1(y)=0$ in a neighbourhood of
$x_0$ in $[-l,l]$. In particular, $\zeta_0z(x_0)-z_1(x_0)=0$ and hence (ii) holds and we obtain a contradiction as before.

 Thus we must have $\zeta_0=1$, and so $(w,z)\succeq (w_1,z_1)\succeq (w,z)$. The uniqueness is thus proved.
Therefore, problem \eqref{model_fixbound} admits a unique positive steady state.

By Remark \ref{compariprinciple1}, it is easy to check that
$u(x,1)>0,~v(x,1)>0$ for $x\in[-l,l]$. Then we can choose $\epsilon>0$ small enough such that $(\epsilon\theta_1,\epsilon\theta_2)$ is a lower solution of \eqref{model_steady} and
\begin{equation*}
u(x,1)\geq \epsilon\theta_1(x), v(x,1)\geq\epsilon\theta_2(x) \mbox{ for }x\in[-l,l].
\end{equation*}

Denote by $(\breve{u},\breve{v})$ the unique positive
solution of \eqref{model_fixbound} with initial functions $(\epsilon\theta_1,\epsilon\theta_2)$.
Then we easily see, similar to before, that
$(\breve{u}(x,t),\breve{v}(x,t))$ is nondecreasing in $t>0$ for $x\in[-l,l]$ and
\begin{equation}\label{brevewz}
(w_2(x),z_2(x))=\lim_{t\rightarrow\infty}(\breve{u}(x,t),\breve{v}(x,t))\preceq (M_1, M_2)
\end{equation}
is a positive solution of \eqref{model_steady}. Moreover the solution pair is continuous and the convergence in \eqref{brevewz} is uniform for $x\in [-l, l]$.
Since $(\epsilon\theta_1,\epsilon\theta_2)\preceq (u(\cdot,1),v(\cdot,1))\preceq(M_1,M_2)$, it follows from
the comparison principle that
\begin{equation*}
(\breve{u}(x,t),\breve{v}(x,t))\preceq(u(x,1+t),v(x,1+t))\preceq(\hat{u}(x,t),\hat{v}(x,t))
\end{equation*}
for all $t>0$ and $x\in[-l,l]$.
Since \eqref{model_steady} has a unique positive solution, we have $(w_i,z_i)=(w,z)$ for $i=1,2$, and hence it follows from \eqref{hatwz} and \eqref{brevewz}
that
\begin{equation*}
\lim_{t\rightarrow \infty}(u(x,t),v(x,t))=(w(x),z(x)) \mbox{ uniformly for }x\in[-l,l].
\end{equation*}

$(ii)$ When $\lambda_0(l)\leq0$, we first prove that problem \eqref{model_fixbound} has no positive
steady state. If it does, let $(w,z)$ be a positive solution pair of \eqref{model_steady}.
Then
\begin{equation*}
\begin{cases}\displaystyle
d_1\int_{-l}^{l}J_1(x-y)w(y)dy-d_1w-a_{11}w
+a_{12}\int_{-l}^{l}K(x-y)z(y)dy=0,&x\in[-l,l],\\
\displaystyle
d_2\int_{-l}^{l}J_2(x-y)z(y)dy-d_2z-a_{22}z+G'(0)w
>0,&x\in[-l,l],\\
\end{cases}
\end{equation*}
due to $G'(0)w>G(w)$.
By applying Lemma \ref{lower-eigen-lemma}, we see that $\lambda_0(l)>0$, which contradicts $\lambda_0(l)\leq0$.
Thus the only nonnegative steady state of \eqref{model_fixbound} is $(0,0)$.

Next we prove the second assertion. We still denote by $(\hat{u},\hat{v})$ the unique positive
solution of \eqref{model_fixbound} with initial functions $(M_1,M_2)$ as chosen earlier. Then much as before,
\begin{equation*}
\lim_{t\rightarrow \infty}(\hat{u}(x,t),\hat{v}(x,t)) =(w,z)\mbox{ uniformly for } x\in[-l,l],
\end{equation*}
with $(w,z)$ a nonnegative solution of \eqref{model_steady}.
Since  $(0,0)$
is the only nonnegative solution of \eqref{model_steady} we thus must have $(w,z)=(0,0)$. Since  $(0,0)\preceq(u,v)\preceq(\hat{u},\hat{v})$ for $t>0$ and $x\in[-l,l]$,
it follows that
\begin{equation*}
\lim_{t\rightarrow \infty}({u}(x,t),{v}(x,t))=(0,0) \mbox{ uniformly for } x\in[-l,l].
\end{equation*}
\end{proof}

\begin{lemma}\label{l_infty}
Suppose that $J_1, J_2$ and $K$ satisfy {\bf (J)},  $G$ satisfies $(G1)$-$(G2)$, and $R_0>1$. So for  $l>l^*$, $\lambda_0(l)>0$
and \eqref{model_steady} has a unique positive solution $(w_l, z_l)$. We have
\begin{equation}\label{l_convergence}
\lim_{l\rightarrow\infty} (w_l(x),z_l(x))=(u^*,v^*) \mbox{ locally uniformly in }\mathbb{R}
 \end{equation}
with $(u^*,v^*)$ defined in \eqref{u^*v^*}.
\end{lemma}

\begin{proof}
Due to $R_0>1$, it follows from $(ii)$ of Proposition \ref{eigen_proposition} that, for all $l>l^*$,
$\lambda_0(l)>0$ holds. Then by Lemma \ref{0steadystate}, we see that \eqref{model_steady}
admits a unique positive solution pair $(w_l, z_l)$ for $x\in[-l,l]$ when $l>l^*$.

We first claim that  $(w_l,v_l)$ is nondecreasing in $l>l^*$.
Assume that $l^*<l_1<l_2$. For $x\in[-l_i,l_i]$, we denote by $(u_{l_i}, v_{l_i})$ the corresponding
solution of \eqref{model_fixbound} with initial functions $(u_{0i},v_{0i})\in (C([-l_i,l_i]))^2$ satisfying
\begin{equation*}
0< u_{01}(x)\leq u_{02}(x)\leq u^*, ~0< v_{01}(x)\leq v_{02}(x)\leq v^* \mbox{ for }x\in[-l_1,l_1].
 \end{equation*}
For $x\in[-l_1,l_1]$ and $t>0$,
\begin{align*}
0&=(u_{l_2})_t-d_1\int_{-l_2}^{l_2}J_1(x-y)u_{l_2}(y,t)dy+d_1u_{l_2}+a_{11}u_{l_2}
-a_{12}\int_{-l_2}^{l_2}K(x-y)v_{l_2}(y,t)dy\\
&\leq (u_{l_2})_t-d_1\int_{-l_1}^{l_1}J_1(x-y)u_{l_2}(y,t)dy+d_1u_{l_2}+a_{11}u_{l_2}
-a_{12}\int_{-l_1}^{l_1}K(x-y)v_{l_2}(y,t)dy\\
\end{align*}
and similarly,
\begin{align*}
0\leq (v_{l_2})_t-d_2\int_{-l_1}^{l_1}J_2(x-y)v_{l_2}(y,t)dy+d_2v_{l_2}+a_{22}v_{l_2}
+G(u_{l_2}).\\
\end{align*}
By the comparison principle we obtain
\begin{align*}
(u_{l_1}, v_{l_1})\preceq(u_{l_2}, v_{l_2})\preceq (u^*, v^*)\mbox { for }x\in[-l_1,l_1], ~t>0.
\end{align*}
It follows from $(ii)$ of Proposition \ref{0steadystate} that $(w_{l_1}, z_{l_1})\preceq (w_{l_2},z_{l_2})\preceq (u^*, v^*)$ for $x\in[-l_1,l_1]$,
which in particular proved the claim.

Define
\begin{equation*}
(\tilde{w}(x),\tilde{z}(x)):=\lim_{l\rightarrow \infty}(w_l(x),z_l(x)).
\end{equation*}
Then $(0,0)\llp(\tilde{w}(x),\tilde{z}(x))\preceq (u^*,v^*)$ in $\mathbb{R}$, and
by applying the dominated convergence theorem, it is easily seen that $(\tilde{w}(x),\tilde{z}(x))$
is a positive solution of \eqref{model_steady} with $[-l,l]$ replaced by $(-\infty,\infty)$.

Next we prove that $(\tilde{w}(x),\tilde{z}(x))$ is a pair of constants for all $x\in \mathbb{R}$. It suffices
to show that for any given $x_0\in \mathbb{R}$, $(\tilde{w}(x_0),\tilde{z}(x_0))=(\tilde{w}(0),\tilde{z}(0))$.
Let $l_0:=|x_0|$. For $l>l^*+2l_0$, define
\begin{equation*}
(w_l^0(x),z_l^0(x)):=(w_{l-l_0}(x+x_0),z_{l-l_0}(x+x_0)).
\end{equation*}
Due to
\begin{equation*}
[-l+2l_0,l-2l_0]\subset[-l+l_0-x_0,l-l_0-x_0]\subset[-l,l],
\end{equation*}
we can show, similar to the proof of the monotonicity claim above, that
\begin{equation*}
(w_{l-2l_0}(x),z_{l-2l_0}(x))\preceq(w_l^0(x),z_l^0(x))\preceq (w_l(x),z_l(x)).
\end{equation*}
Letting $l\rightarrow\infty$, we obtain
\begin{equation*}
(\tilde{w}(x),\tilde{z}(x))\preceq (\tilde{w}(x+x_0),\tilde{z}(x+x_0))\preceq (\tilde{w}(x),\tilde{z}(x)).
\end{equation*}
Therefore, $(\tilde{w}(x_0),\tilde{z}(x_0))=(\tilde{w}(0),\tilde{z}(0))$ by letting $x=0$.

It follows from $R_0>1$ that $(u^*,v^*)$ is the only positive constant solution of \eqref{model_steady} in $\mathbb{R}$.
This implies that $(\tilde{w}(x),\tilde{z}(x)))=(u^*,v^*)$ for $x\in\mathbb{R}$.
By Dini's theorem, the monotonic convergence of $(w_l(x), z_l(x))$ to $(u^*, v^*)$ is uniform for $x$ in any bounded interval, and \eqref{l_convergence} is proved.
\end{proof}

\section{Spreading-vanishing dichotomy and criteria}
\begin{lemma}\label{vanishing1}
If $h_{\infty}-g_{\infty}<\infty$, then $\lambda_0(g_{\infty},h_{\infty})\leq 0$ and
\begin{equation}\label{vanishing_uv}
\lim_{t\rightarrow \infty}\|u(\cdot,t)\|_{C([g(t),h(t)])}
=\lim_{t\rightarrow \infty}\|v(\cdot,t)\|_{C([g(t),h(t)])}=0,
\end{equation}
 where
$\lambda_0(g_{\infty},h_{\infty})$ is the principal eigenvalue of \eqref{eigenprobelm1} with $[-l,l]$ replaced by $[g_{\infty},h_{\infty}]$
\end{lemma}

\begin{proof}
We first prove $\lambda_0(g_{\infty},h_{\infty})\leq 0$ if $h_{\infty}-g_{\infty}<\infty$. Otherwise, $\lambda_0(g_{\infty},h_{\infty})>0$.
It follows from Lemma \ref{PEV-l} that there exists a large constant $T_1$ such that $\lambda_0(g(T),h(T))>0$ for $T\geq T_1$.
Since $J_i(0)>0$, there exist a sufficiently small constant $\varepsilon>0$ and some positive constant
 $\delta>0$ such that, for $i=1,2$,
\begin{equation*}
J_i(x)\geq\delta>0  \mbox{ for }x\in[-4\varepsilon,4\varepsilon].
\end{equation*}
Fix $T>T_1$ so that
\[
[g(T), h(T)]\supset [g_\infty+\varepsilon, h_\infty-\varepsilon],
\]
and let $(\tilde{u}(t,x),\tilde{v}(t,x))$ be the solution of problem \eqref{model_fixbound} with $[-l, l]=[g(T),h(T)]$
and  initial
functions $(u(T,x),v(T,x))$.
Define \[\tilde{U}(x,t):=u(x,t+T)-\tilde{u}(x,t),~~\tilde{V}(x,t):=v(x,t+T)-\tilde{v}(x,t)\] for $t>0$ and $x\in[g(T),h(T)]$.
Then for such $t$ and $x$,
\begin{align*}
&d_1\int_{g(t+T)}^{h(t+T)}J_1(x-y)u(y,t+T)dy-d_1\int_{g(T)}^{h(T)}J_1(x-y)\tilde{u}(y,t)dy\\
\geq&\ d_1\int_{g(T)}^{h(T)}J_1(x-y)u(y,t+T)dy-d_1\int_{g(T)}^{h(T)}J_1(x-y)\tilde{u}(y,t)dy
=d_1\int_{g(T)}^{h(T)}J_1(x-y)\tilde{U}(y,t)dy,
\end{align*}
and similarly
\begin{align*}
&d_2\int_{g(t+T)}^{h(t+T)}J_2(x-y)v(y,t+T)dy-d_2\int_{g(T)}^{h(T)}J_2(x-y)\tilde{v}(y,t)dy
\geq d_2\int_{g(T)}^{h(T)}J_2(x-y)\tilde{V}(y,t)dy,\\
&a_{12}\int_{g(t+T)}^{h(t+T)}K(x-y)v(y,t+T)dy-a_{12}\int_{g(T)}^{h(T)}K(x-y)\tilde{v}(y,t)dy
\geq a_{12}\int_{g(T)}^{h(T)}K(x-y)\tilde{V}(y,t)dy.\\
\end{align*}
Since $\tilde{U}(x,0)=\tilde{V}(x,0)=0$,
by applying Remark \ref{compariprinciple1}, we see that
\begin{equation}\label{compare_inequa}
(\tilde{u}(x,t),\tilde{v}(x,t))\preceq (u(x,t+T),v(x,t+T)) \mbox{ for }x\in[g(T),h(T)] \mbox{ and }t>0.
\end{equation}
By $(i)$ of Proposition \ref{0steadystate},
\begin{equation*}
\lim_{t\rightarrow \infty}(\tilde{u}(x,t),\tilde{v}(x,t))=:(w(x),z(x)) \mbox{ uniformly for }x\in[g(T),h(T)],
\end{equation*}
where $(w(x),z(x))$ is the unique positive solution of \eqref{model_steady} with $[-l, l]$ replaced by $[g(T), h(T)]$.
Therefore, we can find some $T_0>T$ such that
\begin{equation*}
(0,0)\llp\frac{1}{2}(w(x),z(x))\preceq(u(x,t),v(x,t)) \mbox{ for }x\in[g(T),h(T)] \mbox{ and }t>T_0.
\end{equation*}
Hence, for $t>T_0$,
\begin{equation}\label{primeh}
\begin{aligned}
h'(t)=&\ \mu \int_{g(t)}^{h(t)}\int^{\infty}_{h(t)}J_1(x-y)u(x,t)dydx
+\mu\rho \int_{g(t)}^{h(t)}\int^{\infty}_{h(t)}J_2(x-y)v(x,t)dydx\\
\geq&\ \mu \int_{h(t)-2\varepsilon}^{h(t)}\int^{h(t)+2\varepsilon}_{h(t)}J_1(x-y)u(x,t)dydx
+\mu\rho \int_{h(t)-2\varepsilon}^{h(t)}\int^{h(t)+2\varepsilon}_{h(t)}J_2(x-y)v(x,t)dydx\\
\geq &\ 2\mu\delta\varepsilon\int_{h(t)-2\varepsilon}^{h(t)}u(x,t)dx+2\mu\rho\delta\varepsilon
\int_{h(t)-2\varepsilon}^{h(t)}v(x,t)dx\\
\geq&\ 2\mu\delta\varepsilon\int_{h(T)-\varepsilon}^{h(T)}u(x,t)dx+2\mu\rho\delta\varepsilon
\int_{h(T)-\varepsilon}^{h(T)}v(x,t)dx\\
\geq&\
2\mu\delta\varepsilon^2m_2(1+\rho)>0,
\end{aligned}\end{equation}
where $m_2:=\min_{x\in[g(T),h(T)]}\min\{w(x),z(x)\}$.
This is a contradiction with $h_{\infty}<\infty$. We have thus proved $\lambda_0(g_\infty, h_\infty)\leq 0$.

Denote by $(\hat{u}_1(x,t),\hat{v}_1(x,t))$ the solution of problem \eqref{model_fixbound} for
$x\in[g_{\infty},h_{\infty}]$ and $t>0$ with initial
functions $(\|u_0\|_\infty,\|v_0\|_\infty)$. By Lemma 2.1, we have
\begin{equation*}
(0,0)\preceq (u(x,t),v(x,t))\preceq (\hat{u}_1(x,t),\hat{v}_1(x,t)) \mbox{ for }x\in[g(t),h(t)],~t>0.
\end{equation*}
By $(ii)$ of Proposition \ref{0steadystate}, we have
\begin{equation*}
\lim_{t\rightarrow \infty}(\hat{u}_1(x,t),\hat{v}_1(x,t))=(0,0)\mbox{ uniformly for }x\in[g_{\infty},h_{\infty}],
\end{equation*}
and hence \eqref{vanishing_uv} holds.
\end{proof}

\begin{lemma}\label{vanishing2}
If $R_0\leq 1$, then vanishing happens.
\end{lemma}

\begin{proof}
In view of Lemma \ref{vanishing1}, it suffices to show $h_{\infty}-g_{\infty}<\infty$.
We calculate
\begin{equation}\begin{aligned}\label{vanishing_eq1}
&\frac{d}{dt}\int_{g(t)}^{h(t)}\Big[u(x,t)+\frac{a_{12}}{a_{22}}v(x,t)\Big]dx\\
=&\int_{g(t)}^{h(t)}\Big[u_t(x,t)+\frac{a_{12}}{a_{22}}v_t(x,t)\Big]dx
+h'(t)\times 0-g'(t)\times 0\\
=&\int_{g(t)}^{h(t)}\bigg[d_1\int_{g(t)}^{h(t)}J_1(x-y)u(y,t)dy-(d_1+a_{11})u(x,t)
+a_{12}\int_{g(t)}^{h(t)}K(x-y)v(y,t)dy\bigg]dx\\
&+\int_{g(t)}^{h(t)}\frac{a_{12}}{a_{22}}\bigg[d_2\int_{g(t)}^{h(t
)}J_2(x-y)v(y,t)dy-(d_2+a_{22})v(x,t)+G(u(x,t))\bigg]dx\\
=&\int_{g(t)}^{h(t)}\bigg[d_1\int_{g(t)}^{h(t)}J_1(x-y)u(y,t)dy-d_1u(x,t)\bigg]dx\\
&+\int_{g(t)}^{h(t)}\frac{a_{12}}
{a_{22}}\bigg[d_2\int_{g(t)}^{h(t)}J_2(x-y)v(y,t)dy-d_2v(x,t)\bigg]dx\\
&+\int_{g(t)}^{h(t)}\bigg[a_{12}\int_{g(t)}^{h(t)}K(x-y)v(y,t)dy-a_{12}v(x,t)\Big]dx\\
&+\int_{g(t)}^{h(t)}
\big[-a_{11}u(x,t)+\frac{a_{12}}{a_{22}}G(u(x,t))\big]dx.
\end{aligned}\end{equation}
By assumption {\bf (J)} and the third and fourth equations in \eqref{model*}, we have
\begin{align*}
&d_1\int_{g(t)}^{h(t)}\bigg[\int_{g(t)}^{h(t)}J_1(x-y)u(y,t)dy-u(x,t)\bigg]dx\\
&+\frac{d_2a_{12}}{a_{22}}\int_{g(t)}^{h(t)}\bigg[\int_{g(t)}^{h(t)}J_2(x-y)v(y,t)dy-v(x,t)\bigg]dx\\
=&d_1\int_{g(t)}^{h(t)}\int_{g(t)}^{h(t)}J_1(x-y)\big[u(y,t)-u(x,t)\big]dydx\\
&+\frac{d_2a_{12}}{a_{22}}\int_{g(t)}^{h(t)}\int_{g(t)}^{h(t)}J_2(x-y)\big[v(y,t)-v(x,t)\big]dydx\\
&-d_1\int_{g(t)}^{h(t)}\int^{\infty}_{h(t)}J_1(x-y)u(x,t)dydx-\frac{d_2a_{12}}{a_{22}}\int_{g(t)}^{h(t)}\int^{\infty}_{h(t)}J_2(x-y)v(x,t)dydx\\
&- d_1\int_{g(t)}^{h(t)}\int_{-\infty}^{g(t)}J_1(x-y)u(x,t)dydx-\frac{d_2a_{12}}{a_{22}}\int_{g(t)}^{h(t)}\int_{-\infty}^{g(t)}J_2(x-y)v(x,t)dydx\\
\leq&-\frac{m_0d_1}{\mu}\big[h'(t)-g'(t)\big],
\end{align*}
where $m_0:=\min\{1,\frac{d_2a_{12}}{\rho d_1a_{22}}\}$, and we have used
\begin{align*}
\int_{g(t)}^{h(t)}\int_{g(t)}^{h(t)}J_i(x-y)\big[w(y)-w(x)\big]dydx=0 \mbox{ for }i=1,2,
\end{align*}
for any continuous function $w(x)$.
Since $v(x,t)\geq 0$ for $x\in[g(t),h(t)]$ and $K(x)\geq0$ for $x\in \mathbb{R}$, we see that
\begin{align*}
&\int_{g(t)}^{h(t)}\Big[\int_{g(t)}^{h(t)}K(x-y)v(y,t)dy-v(x,t)\Big]dx\\
=&\int_{g(t)}^{h(t)}\int_{g(t)}^{h(t)}K(x-y)\big[v(y,t)-v(x,t)\big]dydx\\
&-\int_{g(t)}^{h(t)}\int_{-\infty}^{g(t)}K(x-y)v(x,t)dydx-\int_{g(t)}^{h(t)}\int^{\infty}_{h(t)}K(x-y)v(x,t)dydx\\
\leq &\ 0.
\end{align*}
Due to $G(u)<G'(0)u$ for all $u>0$, and $R_0\leq1$, we have
\begin{align*}
\int_{g(t)}^{h(t)}\Big[-a_{11}u(x,t)+\frac{a_{12}}{a_{22}}G(u(x,t))\Big]dx
<\int_{g(t)}^{h(t)}\Big[-a_{11}u(x,t)+\frac{a_{12}G'(0)}{a_{22}}u(x,t)\Big]dx
\leq 0.
\end{align*}
Therefore, by \eqref{vanishing_eq1},
\begin{align*}
\frac{d}{dt}\int_{g(t)}^{h(t)}\big[u(x,t)+\frac{a_{12}}{a_{22}}v(x,t)\big]dx
< -\frac{m_0d_1}{\mu}\big[h'(t)-g'(t)\big].
\end{align*}
It follows that
\begin{align*}
\frac{m_0d_1}{\mu}\big[h(t)-g(t)\big]< \int_{-h_0}^{h_0}\big[u_0(x)+\frac{a_{12}}{a_{22}}v_0(x)\big]dx
+\frac{2m_0d_1}{\mu}h_0\mbox{ for all }t\geq0.
\end{align*}
Therefore, $h_{\infty}-g_{\infty}<\infty$ and vanishing happens.
\end{proof}

\begin{lemma}\label{vanishing3}
If $R_0>1$ and $h_0<l^*$, 
then vanishing happens if the initial values $(u_0,v_0)$ of problem \eqref{model*}
is suitably small.
\end{lemma}
\begin{proof}
To prove this lemma, we construct a suitable upper solution to problem \eqref{model*}. Since $R_0>1$ and $h_0<l^*$,
by applying $(ii)$ of Proposition \ref{eigen_proposition}, we see that $\lambda_0(h_0)<0$, where $\lambda_0(h_0)$
is the principal eigenvalue of problem \eqref{eigenprobelm1} with $[-l,l]$ replaced by $[-h_0,h_0]$.
Fix $h_1\in(h_0,l^*)$; we have $\lambda_0(h_1)<0$. Denote  by $(\theta_1,\theta_2)$ a positive eigenfunction pair of \eqref{eigenprobelm1} with
$\lambda=\lambda_0(h_1)$.
Define, for $x\in[-h_1,h_1]$, $t>0$,
\begin{align*}
&\bar{h}(t):=h_1- \big(h_1-h_0\big)e^{\sigma_1 t},~\bar{g}(t)=-\bar{h}(t)\\
&\bar{u}(x,t):=Me^{\sigma_1 t}\theta_1(x),~\bar{v}(x,t):=Me^{\sigma_1 t}\theta_2(x)
\end{align*}
where
\begin{align*}
\sigma_1:=\frac{\lambda_0(h_1)}{2}<0,~~M:=\frac{-\sigma_1(h_1-h_0)}{M_0\mu(1+\rho)},
~~M_0:=\max\Big\{\int^{h_1}_{-h_1}\theta_1(x)dx,\int^{h_1}_{-h_1}\theta_2(x)dx\Big\}.
\end{align*}
Before applying Lemma \ref{compariprinciple}, we first check some conditions.
Clearly,
\begin{align*}
\bar{u}(x,t)\geq 0,~\bar{v}(x,t)\geq0 \mbox{ for }x=\bar{g}(t) \mbox{ or } \bar{h}(t),~t>0.
\end{align*}
For the initial values, if we choose $(u_0,v_0)$ suitably small such that
\begin{align*}
\|u_0\|_{\infty}+\|v_0\|_{\infty}\leq M\min_{x\in[-h_0,h_0]}\{\theta_1(x),\theta_2(x)\}.
\end{align*}
then
\begin{align*}
u_0(x)\leq \bar{u}(x,0)=M\theta_1(x),~v_0(x)\leq \bar{v}(x,0)=M\theta_2(x) \mbox{ for }x\in[-h_0,h_0].
\end{align*}
It is easily checked that $\bar{h}(0)=h_0$ and moreover $h_0\leq\bar{h}(t)< h_1$ for $t\geq0$.
Similarly, $\bar{g}(0)=-h_0$ and $-h_1<\bar{g}(t)\leq -h_0$ for $t\geq0$.

By simple calculations, for $x\in[\bar{g}(t),\bar{h}(t)],~t>0$,
\begin{align*}
&\bar{u}_t-d_1\int_{\bar{g}(t)}^{\bar{h}(t)}J_1(x-y)\bar{u}(y,t)dy+(d_1+a_{11})\bar{u}
-a_{12}\int_{\bar{g}(t)}^{\bar{h}(t)}K(x-y)\bar{v}(y,t)dy\\
\geq&\ \sigma_1 \bar{u}-\lambda_0(h_1) \bar{u}\geq 0,
\end{align*}
and
\begin{align*}
&\bar{v}_t-d_2\int_{\bar{g}(t)}^{\bar{h}(t)}J_2(x-y)\bar{v}(y,t)dy+(d_2+a_{22})\bar{v}
-G(\bar{u})\\
> &\ \sigma_1 \bar{v}-d_2\int_{\bar{g}(t)}^{\bar{h}(t)}J_2(x-y)Me^{\sigma_1 t}\theta_2(y)dy+(d_2+a_{22})Me^{\sigma_1 t}\theta_2(x)
-G'(0)Me^{\sigma_1 t}\theta_1(x)\\
\geq&\ \sigma_1 \bar{v}-\lambda_0(h_1) \bar{v}\geq0,
\end{align*}
due to $G'(0)\bar{u}>G(\bar{u})$ for $\bar{u}>0$.
Furthermore, we have
\begin{align*}
&\mu  \int_{\bar{g}(t)}^{\bar{h}(t)}\int^{\infty}_{\bar{h}(t)}J_1(x-y)\bar{u}(x,t)dydx+
\mu\rho \int_{\bar{g}(t)}^{\bar{h}(t)}\int^{\infty}_{\bar{h}(t)}J_2(x-y)\bar{v}(x,t)dydx\\
\leq &\ \mu  \int_{\bar{g}(t)}^{\bar{h}(t) }\bar{u}(x,t)dx+\mu\rho \int_{\bar{g}(t)}^{\bar{h}(t)}\bar{v}(x,t)dx\\
=&\ \mu M e^{\sigma_1 t}\int_{\bar{g}(t)}^{\bar{h}(t)}\theta_1(x)dx +\mu\rho
Me^{\sigma_1 t}\int_{\bar{g}(t)}^{\bar{h}(t)}\theta_2(x)dx\\
\leq &\ \mu M e^{\sigma_1 t}\int_{-h_1}^{h_1}\theta_1(x)dx +\mu\rho
Me^{\sigma_1 t}\int_{-h_1}^{h_1}\theta_2(x)dx\\
\leq&-(h_1-h_0)\sigma_1 e^{\sigma_1 t}=\bar{h}'(t) ~~\mbox{for }t>0,
\end{align*}
since $[\bar{g}(t),\bar{h}(t)]\subset[-h_1,h_1]$.
Analogously, we have, for $t>0$,
\begin{align*}
-\mu  \int_{\bar{g}(t)}^{\bar{h}(t)}\int_{-\infty}^{\bar{g}(t)}J_1(x-y)\bar{u}(x,t)dydx-
\mu\rho\int_{\bar{g}(t)}^{\bar{h}(t)} \int_{-\infty}^{\bar{g}(t)}J_2(x-y)\bar{v}(x,t)dydx
\geq \bar{g}'(t).
\end{align*}

Thus $(\bar{u},\bar{v},\bar{g},\bar{h})$
is an upper solution of problem \eqref{model*}, and by
 Lemma \ref{compariprinciple}, we have $h(t)\leq \bar{h}(t)$ and $g(t)\geq \bar{g}(t)$
for all $t>0$, which gives
 \begin{align}\label{h-g<infty}
h_{\infty}-g_{\infty}\leq \bar{h}(\infty)-\bar{g}(\infty)=2h_1<2l^*.
\end{align}
This shows that vanishing happens by Lemma \ref{vanishing1}.
\end{proof}
\begin{remark}\label{vanishingrm}
By the expression of $M$ in Lemma \ref{vanishing3}, we see that
 $M\rightarrow \infty$ when  $\mu\rightarrow 0$. This implies that, under the assumptions $R_0>1$ and $h_0<l^*$, for any given pair of initial functions
$(u_0,v_0)$ satisfying \eqref{Assumption}, there exists a small $\mu_0>0$ such that vanishing happens if
$\mu\in(0,\mu_0]$.
\end{remark}

\begin{lemma}\label{vanishing4}
$h_{\infty}<\infty$ if only and only if $-g_{\infty}<\infty$.
\end{lemma}
\begin{proof}
By way of contradiction, we assume that $h_{\infty}<\infty$ and $-g_{\infty}=\infty$. Then we see
that $h_{\infty}-g_{\infty}=\infty$ and thus by Lemma \ref{vanishing2}, $R_0>1$ must hold.
Due to $g_{\infty}=\infty$, there exists a large $T$ such that for all $t\geq T$, $h(t)-g(t)> 2l^*$.
It follows from $(ii)$ of Proposition \ref{eigen_proposition} that $\lambda_0(g(t),h(t))>0$ for such
$t\geq T$.
Since $h_{\infty}-g_{\infty}=\infty$, similar to \eqref{primeh}, we can find some constant $c_2>0$ such that for some
large $T^0>T$,
\begin{align*}
h'(t)=\mu \int_{g(t)}^{h(t)}\int^{\infty}_{h(t)}J_1(x-y)u(x,t)dyd
+\mu\rho \int_{g(t)}^{h(t)}\int^{\infty}_{h(t)}J_2(x-y)v(x,t)dydx>
c_2\mbox{ for }t>T^0,
\end{align*}
which contradicts $h(t)<\infty$ for all $t>0$.
Therefore, if $h_{\infty}<\infty$ holds, we have $-g_{\infty}<\infty$. Similarly we can show that $-g_{\infty}<\infty$ and $h_{\infty}=\infty$ lead to a contradiction.
\end{proof}
\begin{lemma}\label{spreading1}
If $R_0>1$ and $h_0\geq l^*$, then $h_{\infty}=-g_{\infty}=\infty$ and
\begin{align*}
\lim_{t\rightarrow \infty}(u(x,t),v(x,t))=(u^*,v^*) \mbox{ locally uniformly in }\mathbb{R},
\end{align*}
where $(u^*,v^*)$ is defined in \eqref{u^*v^*}.
\end{lemma}
\begin{proof}
If $R_0>1$ and $h_0\geq l^*$, by $(ii)$ of Proposition \ref{eigen_proposition}, for any $t>0$, we have
$\lambda_0(g(t),h(t))>0$, where $\lambda_0(g(t),h(t))$ is the unique eigenvalue of \eqref{eigenprobelm1}
with $[-l,l]$ replaced by $[g(t),h(t)]$. Since $(g(t),h(t))\subset (g_{\infty},h_{\infty})$ for $t>0$,
$\lambda_0(g_{\infty},h_{\infty})>0$ holds.
By Lemma \ref{vanishing1}, we have $h_{\infty}-g_{\infty}=\infty$. Lemma
\ref{vanishing4} shows that $h_{\infty}=-g_{\infty}=\infty$.

Next we will prove the second assertion by showing the limit superior and limit inferior
of the solution of \eqref{model*} are both $(u^*,v^*)$. We first consider the limit inferior of
this solution.
Due to $\lambda_0(g(T),h(T))>0$ for any fixed $T>0$, we define by $(\tilde{u},\tilde{v})$
the unique positive solution of \eqref{model_fixbound} in $[g(T),h(T)]\times(0,\infty)$ with initial function pair
$(\tilde{u}(x,0),\tilde{v}(x,0))=(u(x,T),v(x,T)$. By applying Proposition \ref{0steadystate}, we see
\begin{align*}
\lim_{t\rightarrow\infty }(\tilde{u}(x,t),\tilde{v}(x,t))=(w(x),z(x)) \mbox{ uniformly for }x\in[g(T),h(T)],
\end{align*}
where $(w(x),z(x))$ is the unique solution of problem \eqref{model_steady} with
$[-l,l]$ replaced by $[g(T),h(T)]$.
Similar to \eqref{compare_inequa}  we see that for any $t>0$,
\begin{align*}
(\tilde{u}(x,t),\tilde{v}(x,t))\preceq(u(x,t+T),v(x,t+T)) \mbox{ for }x\in[g(T),h(T)],
\end{align*}
which implies that
\begin{align*}
(w(x),z(x))\preceq\liminf_{t\rightarrow \infty} (u(x,t),v(x,t)) \mbox{ for }x\in[g(T),h(T)].
\end{align*}
Since $(g(T),h(T))\rightarrow (-\infty,\infty)$ as $T\rightarrow \infty$, by Lemma \ref{l_infty}, we see that
\begin{align*}
(u^*,v^*)\preceq \liminf_{t\rightarrow \infty}(u(x,t),v(x,t))   \mbox{ locally uniformly in }\mathbb{R}.
\end{align*}

Finally we will consider the limit superior of the solution of \eqref{model*}. Denote by $(\bar{u}(t),\bar {v}(t))$
the unique solution of \eqref{model*ODE2} with initial function pair $(M_1,M_2)=(\|u_0\|_\infty,\|v_0\|_\infty)$. It is easily checked
that $(\bar{u},\bar {v})$ is an upper solution of problem \eqref{model*}, and hence
\begin{align*}
(u(x,t),v(x,t))\preceq (\bar{u}(t),\bar {v}(t))\mbox{ for }t>0,~x\in[g(t),h(t)].
\end{align*}
Since $R_0>1$, we have
\begin{align*}
\lim_{t\rightarrow \infty}(\bar{u}(t),\bar {v}(t))=(u^*,v^*),
\end{align*}
which implies that the limit superior of $(u(x,t),v(x,t))$ is $(u^*,v^*)$ uniformly in $x$. This finishes the proof.
\end{proof}

\begin{lemma}\label{spreading2}
Assume that $R_0>1$ and $h_0<l^*$. Then there exists $\mu^0$ such that spreading occurs
when $\mu>\mu^0$.
\end{lemma}
\begin{proof}
In this proof, to highlight the dependence on $\mu$, we denote by
$(u_{\mu},v_{\mu},g_{\mu},h_{\mu})$ the unique
solution of \eqref{model*}.
Then, suppose for contradiction that for all $\mu>0$, vanishing happens and
thus $h_{\mu,\infty}-g_{\mu,\infty}<\infty$, where
\begin{equation}\label{hg_mu}
g_{\mu,\infty}:=\lim_{t\rightarrow \infty}g_{\mu}(t),~h_{\mu,\infty}:=\lim_{t\rightarrow \infty}h_{\mu}(t).
\end{equation}
It is easily checked that $u_{\mu}$, $v_{\mu}$, $-g_{\mu}(t)$ and $h_{\mu}(t)$ are increasing in $\mu>0$ by
Lemma \ref{compariprinciple}. Therefore,
\begin{equation*}
H_{\infty}:=\lim_{\mu\rightarrow\infty}h_{\mu,\infty}\mbox{ and }G_{\infty}:=\lim_{\mu\rightarrow\infty}g_{\mu,\infty}
\end{equation*}
exist.
By Lemma \ref{vanishing1}, $\lambda_0(G_{\infty},H_{\infty})\leq0$. Due to $R_0>1$, by
$(ii)$ of Proposition \ref{eigen_proposition}, we see that $H_{\infty}-G_{\infty}\leq 2l^*$.
Since $J_i(0)>0$, there exist $\varepsilon>0$ and some positive constant
 $\delta_0>0$ such that, for $i=1,2$,
\begin{equation*}
J_i(x)\geq\delta_0>0  \mbox{ for }x\in[-4\varepsilon,4\varepsilon].
\end{equation*}
Fix $t_1>0$ and $\mu_1>0$ such that
\[
\begin{cases}h_{\mu,\infty}>H_\infty-\varepsilon/2,\; g_{\mu,\infty}<G_\infty+\varepsilon/2 \mbox{ for }\mu\geq \mu_1,\\
h_{\mu_1}(t)\geq h_{\mu_1,\infty}-\varepsilon/2> H_\infty-\varepsilon,\; g_{\mu_1}(t)\leq g_{\mu_1,\infty}+\varepsilon/2< G_\infty+\varepsilon \mbox{ for } t\geq t_1.
\end{cases}
\]
 Then by \eqref{model*}, we see that, for all $\mu>\mu_1$,
\begin{align*}
\mu=&\bigg(\int^{\infty}_{t_1}\int_{g_{\mu}(s)}^{h_{\mu}(s)}\int^{\infty}_{h_{\mu}(s)}J_1(x-y)u_{\mu}(x,s)dydxds\\
&+\rho\int^{\infty}_{t_1}\int_{g_{\mu}(s)}^{h_{\mu}(s)}\int^{\infty}_{h_{\mu}(s)}J_2(x-y)v_{\mu}(x,s)dydxds\bigg)^{-1}
\big[h_{\mu,\infty}-h_{\mu}(t_1)\big]\\
\leq&\bigg(\int^{t_1+1}_{t_1}\int_{g_{\mu_1}(s)}^{h_{\mu_1}(s)}\int^{\infty}_{H_\infty}J_1(x-y)u_{\mu_1}(x,s)dydxds\\
&+\rho\int^{t_1+1}_{t_1}\int_{g_{\mu_1}(s)}^{h_{\mu_1}(s)}
\int^{\infty}_{H_\infty}J_2(x-y)v_{\mu_1}(x,s)dydxds\bigg)^{-1}
2l^*\\
\leq&\bigg(\int^{t_1+1}_{t_1}\int_{H_\infty-3\varepsilon}^{H_\infty-\varepsilon}\int^{H_\infty+\varepsilon}_{H_\infty}\delta_0u_{\mu_1}(x,s)dydxds\\
&+\rho\int^{t_1+1}_{t_1}\int_{H_\infty-3\varepsilon}^{H_\infty-\varepsilon}
\int^{H_\infty+\varepsilon}_{H_\infty}\delta_0v_{\mu_1}(x,s)dydxds\bigg)^{-1}
2l^*\\
=&\ \bigg(\int^{t_1+1}_{t_1}\int_{H_\infty-3\varepsilon}^{H_\infty-\varepsilon}
\varepsilon\delta_0 [u_{\mu_1}(x,s)+\rho v_{\mu_1}(x,s)]dxds\bigg)^{-1}2l^*
< \infty.
\end{align*}
This is a contradiction. Therefore, there exists some $\mu^0>0$ such that
 spreading happens when $\mu>\mu^0$.
\end{proof}

\begin{theorem}
Suppose that $R_0>1$ and $h_0<l^*$. Then there exists $\mu^*>0$ depending on $(u_0,v_0)$ such that
vanishing happens when $\mu\in(0,\mu^*]$ and spreading occurs when $\mu\in(\mu^*,\infty)$.
\end{theorem}
\begin{proof}
Let us denote by $(u_{\mu},v_{\mu},g_{\mu},h_{\mu})$ the unique
solution of \eqref{model*}
to highlight its dependence on $\mu$.
By Lemma \ref{vanishing1}, $h_{\mu,\infty}-g_{\mu, \infty}\leq 2l^*$ holds once vanishing happens,  where $g_{\mu,\infty},h_{\mu,\infty}$ are defined as in
\eqref{hg_mu}.
Define
\begin{equation*}
\Lambda:=\{\mu>0:h_{\mu,\infty}-g_{\mu,\infty}\leq 2l^*\}.
\end{equation*}
By Remark \ref{vanishingrm}, it is obvious that $(0,\mu_0]\subset\Lambda$. By Lemma \ref{spreading2},
$(\mu^0,\infty)\cap\Lambda=\emptyset$.
Therefore, $\sup\Lambda\in[\mu_0,\mu^0]$. We define $\mu^*:=\sup\Lambda$. By Lemma \ref{spreading1},  clearly
$h_{\mu,\infty}-g_{\mu,\infty}=\infty$ and spreading happens for every $\mu>\mu^*$.

Next we  show that $\mu^*\in\Lambda$. Otherwise, $\mu^*\notin\Lambda$. Then we see that
$h_{\mu,\infty}-g_{\mu,\infty}=\infty$ if $\mu=\mu^*$. Hence we can find a large $T>0$ such that
$h_{\mu^*}(T)-g_{\mu^*}(T)>2l^*$. By the continuous dependence of $(u_{\mu},v_{\mu},g_{\mu},h_{\mu})$
on $\mu$, there exists a small enough $\epsilon>0$ such that $h_{\mu}(T)-g_{\mu}(T)>2l^*$ holds
for every $\mu\in[\mu^*-\epsilon,\mu^*+\epsilon]$.
Since $h_{\mu}(t)$ and $-g_{\mu}(t)$ are both increasing in $t$, we see that
\begin{equation*}
\lim_{t\rightarrow \infty}\big[h_{\mu}(t)-g_{\mu}(t)\big]>h_{\mu}(T)-g_{\mu}(T)>2l^*
\end{equation*}
for $\mu\in[\mu^*-\epsilon,\mu^*+\epsilon]$. Then $\sup\Lambda\leq \mu^*-\epsilon$, which contradicts
the definition of $\mu^*$. Thus, $\mu^*\in\Lambda$.

By applying Lemma \ref{compariprinciple}, we see that $(u_{\mu^*},v_{\mu^*},g_{\mu^*},h_{\mu^*})$
is an upper solution of \eqref{model*} for every $\mu\in(0,\mu^*)$ and
\begin{equation*}
h_{\mu}(t)\leq h_{\mu^*}(t),~g_{\mu}(t)\geq g_{\mu^*}(t) \mbox{ for }t>0.
\end{equation*}
Therefore,
\begin{equation*}
\lim_{t\rightarrow \infty}\big[h_{\mu}(t)-g_{\mu}(t)\big]\leq \lim_{t\rightarrow \infty}\big[
h_{\mu^*}(t)-g_{\mu^*}(t)\big]\leq 2l^*,
\end{equation*}
which implies that $\mu\in\Lambda$ for every $\mu\in(0,\mu^*]$. Hence, $\Lambda=(0,\mu^*]$.
\end{proof}

\section{Proof of Theorem 1.1}

We give the  proof of Theorem 1.1 in this section. We always assume that $J_1, J_2$ and $K$ satisfy {\bf (J)}, $G$ satisfies $(G1)$-$(G2)$,
and $(u_0(x),v_0(x))$ satisfies \eqref{Assumption}.

Define
\[K_1:=\max\Big\{u^*,~\|u_0\|_{\infty},~\frac{a_{12}}{a_{11}}\|v_0\|_{\infty}\Big\},
K_2:=\max\Big\{\|v_0\|_{\infty},,~\frac{G(K_1)}{a_{22}}\Big\},\]
where $ u^*$ is given as before if $R_0>1$, and $u^*=0 \mbox{ if } R_0\leq1$.
Clearly,
\begin{equation}\label{K_1K_2}
-a_{11}K_1+a_{12}K_2\leq0~~\text{and} ~~-a_{22}K_2+G(K_1)\leq0.
\end{equation}
\begin{lemma}\label{solution_uv}
For any $T>0$ and $(g,h)\in G^T\times H^T$ with $G^T$, $H^T$ given by \eqref{definition_original}, the following problem
\begin{equation}\label{model3}
\begin{cases}\displaystyle
u_t=d_1\int_{g(t)}^{h(t)}J_1(x-y)u(y,t)dy-d_1u-a_{11}u&\\
\displaystyle
\hspace{2.7cm}+a_{12}\int_{g(t)}^{h(t)}K(x-y)v(y,t)dy,&0<t\leq T,~~x\in(g(t),h(t)),\\
\displaystyle
v_t=d_2\int_{g(t)}^{h(t)}J_2(x-y)v(y,t)dy-d_2v-a_{22}v+G(u),&0<t\leq T,~~x\in(g(t),h(t)),\\
u(h(t),t)=u(g(t),t)=v(h(t),t)=v(g(t),t)=0,&0<t\leq T,\\
u(x,0)=u_0(x),~~v(x,0)=v_0(x),&x\in[-h_0,h_0],
\end{cases}
\end{equation}
admits a unique solution $(u,v)\in X^T=X^T(g,h,u_0,v_0)$ satisfying
\begin{equation}\label{estimate_uv}
0<u(x,t)\leq K_1,~~0<v(x,t)\leq K_2\mbox{ for } (x,t)\in\Omega_{T}.
\end{equation}
\end{lemma}

\begin{proof}
We shall finish this proof in three steps.

\textbf{Step 1}: A parameterised ODE problem.

We first denote
\begin{equation*}
\begin{split}
&\hat{f}_1(u,v)=-(d_1+a_{11})u, \\
&\hat{f}_2(u,v)=-(d_2+a_{22})v+G(u),
\end{split}
\end{equation*}
and define for $i=1,2$,
\[\tilde{f}_i(u,v)=
\begin{cases}
\hat{f}_i(u,v),&u,v\geq0,\\
0,&u,v<0,\\
\hat{f}_i(u,0),&u\geq0,~~v<0,\\
\hat{f}_i(0,v),&u<0,~~v\geq0.\\
\end{cases}
\]

For any $x\in[g(T),h(T)]$, define
\begin{equation*}
\tilde{u}_0(x)=\begin{cases}
u_0(x),&x\in[-h_0,h_0],\\
0,&\mbox{otherwise},
\end{cases}\hspace{0.5cm}\hspace{0.5cm}
\tilde{v}_0(x)=\begin{cases}
v_0(x),&x\in[-h_0,h_0],\\
0,&\mbox{otherwise},
\end{cases}
\end{equation*}
and
\begin{equation*}
t_x:=\begin{cases}
t_{x}^h, &\mbox{ if } x\in(h_0,h(T)]\mbox{ and } x=h(t_{x}^h),\\
0,&\mbox{ if } x\in[-h_0,h_0],\\
t_{x}^g,&\mbox{ if } x\in[g(T),-h_0)\mbox{ and } x=g(t_{x}^g).\\
\end{cases}
\end{equation*}

Clearly, $t_x=T$ if $x=g(T)$ and $x=h(T)$, and $0\leq t_x<T$ if $x\in(g(T),h(T))$.
For any given $(\phi,\psi)\in X^T$, consider the following ODE initial value problem
\begin{equation}\label{ODEmodel1}
\begin{cases}
u_t=\tilde{f}_1(u,v)+g_1(x,t),\ \
v_t=\tilde f_2(u,v)+g_2(x,t),& t> t_x,\\
u(x,t_x)=\tilde{u}_0(x),~v(x,t_x)=\tilde{v}_0(x), &x\in(g(T),h(T)),
\end{cases}
\end{equation}
where
\begin{equation*}
\begin{split}
&g_1(x,t):=d_1\int_{g(t)}^{h(t)}J_1(x-y)\phi(y,t)dy+a_{12}\int_{g(t)}^{h(t)}K(x-y)\psi(y,t)dy,\\
&g_2(x,t):=d_2\int_{g(t)}^{h(t)}J_2(x-y)\psi(y,t)dy.
\end{split}
\end{equation*}
Fix $M_1,M_2>\max\{ \|\phi\|_{C(\overline{\Omega}_{T})}, \|\psi\|_{C(\overline{\Omega}_{T})},\|u_0\|_{\infty},\|v_0\|_{\infty}\}$ such that \eqref{M_1M_2} holds.
Then define
\begin{equation*}
M^*:=2+M_1+M_2.
\end{equation*}
Since we have extended $G(u)$ to 0 when $u<0$ and $G\in C^1([0,\infty))$, for any $u_1,u_2\in(-\infty,M^*]$, we can find some constant $M_0$ such that
\begin{equation*}
|G(u_1)-G(u_2)|\leq M_0|u_1-u_2|.
\end{equation*}
Therefore, for any $u_i,~v_i\in(-\infty,M^*]$ ($i=1,2$),
\begin{equation*}
\begin{split}
&|\tilde f_1(u_1,v_1)-\tilde f_1(u_2,v_2)|\leq \tilde{M}\big(|u_1-u_2|+|v_1-v_2|\big),\\
&|\tilde f_2(u_1,v_1)-\tilde f_2(u_2,v_2)|\leq \tilde{M}\big(|u_1-u_2|+|v_1-v_2|\big),
\end{split}
\end{equation*}
where $\tilde{M}=d_1+d_2+a_{11}+a_{22}+M_0$.
This implies that  $\tilde f_i(u,v)$ $(i=1,2)$ is Lipschitz continuous in $u,v\in[-\infty,M^*]$ with Lipschitz constant $\tilde{M}$.
Moreover, by the definition, $g_i$ $(i=1,2)$ is continuous in $(x,t)$.
Therefore, by applying the Fundamental Theorem of ODEs,
for every fixed $x\in(g(T),h(T))$, \eqref{ODEmodel1} admits a unique solution $(U_{\phi}(x,t),V_{\psi}(x,t))$ in some interval of $t$, say $[t_x,T_x)$.

Next we claim that the
solution $(U_{\phi}(x,t),V_{\psi}(x,t))$ can be uniquely extended to $t\in[t_x,T]$. It suffices to show the following holds:
\begin{equation}\label{inequaility_UV}
0\leq U_{\phi}(x,t),V_{\psi}(x,t)\leq M^* \mbox{ for }t\in(t_x,T_1],\\
\end{equation} whenever $(U_{\phi}(x,\cdot),V_{\psi}(x,\cdot))$ is uniquely defined in $[t_x,T_1]$ with $T_1\in(t_x,T]$.

For $t\in(t_x, T_1]$ and $x\in(g(T_1),h(T_1))$,
set $z_1(x,t)=M_1-U_{\phi}(x,t)$ and $z_2(x,t)=M_2-V_{\psi}(x,t)$. By assumption {\bf (J)} and \eqref{M_1M_2}, we have
\begin{equation*}
\begin{split}
&(z_1)_t+(d_1+a_{11})z_1\\
=& -(U_{\phi})_t+(d_1+a_{11})M_1-(d_1+a_{11})U_{\phi}
\\
=&-d_1\int_{g(t)}^{h(t)}J_1(x-y)\phi(y,t)dy+(d_1+a_{11})M_1-a_{12}\int_{g(t)}^{h(t)}K(x-y)\psi (y,t)dy\\
>&-d_1M_1+(d_1+a_{11})M_1-a_{12}M_2\geq0.
\end{split}
\end{equation*}
Since the initial value $z_1(x,t_x)=M_1-\tilde{u}_0(x)>0$ for $x\in(g(T_1),h(T_1))$, it is easily checked that
$z_1(x,t)\geq0$ and thus $M_1\geq U_{\phi}(x,t)$ for $t\in[t_x,T_1]$ and $x\in(g(T_1),h(T_1))$.
Furthermore, we have
\begin{equation*}
\begin{split}
&(z_2)_t+(d_2+a_{22})z_2
=-(V_{\psi})_t+(d_2+a_{22})M_2-(d_2+a_{22})V_{\psi}\\
=&-d_2\int_{g(t)}^{h(t)}J_2(x-y)\psi(y,t)dy+(d_2+a_{22})M_2-G(U_{\phi})\\
>&-d_2M_2+(d_2+a_{22})M_2-G(M_1)\geq0.
\end{split}
\end{equation*}
By $z_2(x,t_x))=M_2-\tilde{v}_0(x)>0 $ for $ x\in(g(T_1),h(T_1))$,
we similarly obtain $V_{\psi}(x,t)\leq M_2$ for $t\in(t_x,T_1]$ and $x\in(g(T_1),h(T_1))$.

Moreover, since
\begin{equation*}
\begin{split}
&(U_{\phi})_t+(d_1+a_{11})U_{\phi}=
d_1\int_{g(t)}^{h(t)}J_1(x-y)\phi(y,t)dy+a_{12}\int_{g(t)}^{h(t)}K(x-y)\psi(y,t)dy\geq 0,\\
&(V_{\psi})_t+(d_2+a_{22})V_{\psi}=d_2\int_{g(t)}^{h(t)}J_2(x-y)\psi(y,t)dy+G(U_{\phi})\geq 0,
\end{split}
\end{equation*}
and \[(U_{\phi}(x,t_x),V_{\psi}(x,t_x))=(\tilde{u}_0(x),\tilde{v}_0(x))\succeq(0,0)\mbox{ for } x\in(g(T_1),h(T_1)) ,\]
it is easily checked that
$(U_{\phi}(x,t),V_{\psi}(x,t))\succeq (0,0)$ for $t\in(t_x,T_1]$.
Thus \eqref{inequaility_UV} holds, and the claim is proved.

\textbf{Step 2}: A fixed point problem.

From Step 1, we know that $(U_{\phi}(x,0),V_{\psi}(x,0))=(u_0(x),v_0(x))$ for $x\in [-h_0,h_0]$, and $(U_{\phi}(x,t),V_{\psi}(x,t))=(0,0)$ for
$x\in \{g(t),h(t)\}$ and $t\in[0,T]$.
By the continuous dependence of the ODE solution on the initial value and parameters, $(U_{\phi}(x,t),V_{\psi}(x,t))$ is continuous in $\overline{\Omega}_{T}$. For any  $T_1\in(0,T]$, we define the mapping $\mathcal{F}:~X^{T_1}\longrightarrow X^{T_1}$ by
\begin{equation*}
\mathcal{F}(\phi,\psi)=(U_{\phi},V_{\psi}).
\end{equation*}
Noting $\tilde{f}_i(u,v)=\hat{f}_i(u,v)$ $(i=1,2)$ for any $u,v\geq0$, we see that $(\phi,\psi)$ is a fixed point of $\mathcal{F}$ if and only if it solves \eqref{model3} for $t\in(0,T_1]$.
Clearly $X^{T_1}$ is a complete metric space equipped with metric
\begin{equation*}
d\big((\phi_1,\psi_1),(\phi_2,\psi_2)\big)=\|\phi_1-\phi_2\|_{C(\overline{\Omega}_{T_1})}+\|\psi_1-\psi_2\|_{C(\overline{\Omega}_{T_1})},
\end{equation*}
and
\begin{equation*}
X^{T_1}_{K}:=\{(\phi,\psi)|(\phi,\psi)\in X^{T_1},~\|\phi\|_{C(\overline{\Omega}_{T_1})}+\|\psi\|_{C(\overline{\Omega}_{T_1})}\leq K \}
\end{equation*}
with $K:=4(K_1+K_2)$ is a closed subset of $X^{T_1}$.

In this step, we will prove that
$\mathcal{F}$ maps $X^{T_1}_{K}$ into itself and is a
contraction mapping provided that $T_1$ is sufficiently small. It then follows from the
contraction mapping theorem that problem \eqref{model3} admits a unique solution in $X^{T_1}_{K}$.
Finally, we will explain that any solution of \eqref{model3} defined for $t\in [0, T_1]$ belongs to $X^{T_1}_{K}$.

For any $(\phi,\psi)\in X^{T_1}_{K}$,  $(u,v)=(U_{\phi},V_{\psi})$ solves \eqref{ODEmodel1} with $T$ replaced by $T_1$.
By \eqref{inequaility_UV} and $(G1)$,
\begin{equation*}
\sum_{i=1}^{2}\hat{f}_i(U_{\phi},V_{\psi})\leq
 k_1[U_{\phi}+V_{\psi}],
\end{equation*}
where $k_1:=d_1+d_2+a_{11}+a_{22}+G'(0)$.
Then we have
\begin{equation*}
\begin{split}
(U_{\phi})_t+(V_{\psi})_t\leq &d_1\int_{g(t)}^{h(t)}J_1(x-y)\phi(y,t)dy+d_2\int_{g(t)}^{h(t)}J_2(x-y)\psi(y,t)dy\\
&+a_{12}\int_{g(t)}^{h(t)}K(x-y)\psi(y,t)dy
+k_1[U_{\phi}+V_{\psi}]\\
\leq&d_1\|\phi\|_{C(\overline{\Omega}_{T_1})}+(d_2+a_{12})\|\psi\|_{C(\overline{\Omega}_{T_1})}+k_1[U_{\phi}+V_{\psi}].
\end{split}
\end{equation*}
Multiplying $e^{-k_1t}$ to the above inequality and then integrating it from $t_x$ to $t\in(t_x,T_1]$, we obtain
\begin{equation*}
\begin{split}
U_{\phi}(x,t)+V_{\psi}(x,t)&\leq e^{k_1t}(U_{\phi}(x,t_x)+V_{\psi}(x,t_x))+
\big(d_1\|\phi\|_{C(\overline{\Omega}_{T_1})}+(d_2+a_{12})\|\psi\|_{C(\overline{\Omega}_{T_1})}\big)\int^t_{t_x}e^{k_1(t-s)}ds\\
&\leq e^{k_1t}\big(\|u_0\|_{\infty}+\|v_0\|_{\infty}\big)+te^{k_1t}\big(d_1\|\phi\|_{C(\overline{\Omega}_{T_1})}+(d_2+a_{12})
\|\psi\|_{C(\overline{\Omega}_{T_1})}\big)\\
&\leq e^{k_1t}(K_1+K_2)+Kte^{k_1t}(d_1+d_2+a_{12}).
\end{split}
\end{equation*}

Choose some small enough constant $T_1^*\in(0,T)$ such that $e^{k_1T_1^*}<2$ and $T^*_1e^{k_1T^*_1}(d_1+d_2+a_{12})<\frac{1}{4}$.
Then, for any $x\in (g(T_1), h(T_1))$ and $t_x<t\leq T^*_1$,
\begin{equation*}
U_{\phi}(x,t)+V_{\psi}(x,t)\leq 2(K_1+K_2)+\frac{1}{4}K=\frac{3}{4}K<K,
\end{equation*}
which implies that $(U_{\phi}(x,t),V_{\psi}(x,t))\in X^{T_1}_{K}$ for $(x,t)\in \overline{\Omega}_{T_1}$ with $T_1\in(0,T_1^*]$.

For any  $T_1\in(0,T^*]$,
let
$(\phi_i,\psi_i)\in  X^{T_1}_{K}$ ($i=1,2$), and denote $U=U_{\phi_1}-U_{\phi_2}$, $V=V_{\psi_1}-V_{\psi_2}$. Then
 \begin{equation*}
 \begin{cases}\displaystyle
U_t+(d_1+a_{11})U=d_1\int_{g(t)}^{h(t)}J_1(x-y)(\phi_1-\phi_2)(y,t)dy\\
\displaystyle
\hspace{4cm}+a_{12}\int_{g(t)}^{h(t)}K(x-y)(\psi_1-\psi_2)(y,t)dy,&t_x<t\leq T_1,\\
U(x,t_x)=0,&x\in(g(T_1),h(T_1)).
 \end{cases}
 \end{equation*}
It follows  that for $x\in (g(T_1), h(T_1))$ and $t_x<t\leq T_1$,
\begin{equation*}
\begin{split}
&U(x,t)=e^{-(d_1+a_{11})t}\int_{t_x}^t e^{(d_1+a_{11})s}\Big(d_1\int_{g(s)}^{h(s)}J_1(x-y)(\phi_1-\phi_2)(y,s)dy\\
&\hspace{6cm}+a_{12}\int_{g(s)}^{h(s)}K(x-y)(\psi_1-\psi_2)(y,s)dy\Big)ds.
 \end{split}
 \end{equation*}
Therefore,
\begin{equation*}
\begin{split}
|U(x,t)|&\leq\big( d_1\|\phi_1-\phi_2\|_{C(\overline{\Omega}_{T_1})}+a_{12}\|\psi_1-\psi_2\|_{C(\overline{\Omega}_{T_1})}\big)(t-t_x)\\
&\leq  T_1\big( d_1\|\phi_1-\phi_2\|_{C(\overline{\Omega}_{T_1})}+a_{12}\|\psi_1-\psi_2\|_{C(\overline{\Omega}_{T_1})}\big),
\end{split}
\end{equation*}
and thus
\begin{equation*}
\|U\|_{C(\overline{\Omega}_{T_1})}\leq T_1\big( d_1\|\phi_1-\phi_2\|_{C(\overline{\Omega}_{T_1})}+a_{12}\|\psi_1-\psi_2\|_{C(\overline{\Omega}_{T_1})}\big).
\end{equation*}
Similarly,
\begin{equation*}
\|V\|_{C(\overline{\Omega}_{T_1})}\leq T_1\big( d_2\|\psi_1-\psi_2\|_{C(\overline{\Omega}_{T_1})}+G^*\|U\|_{C(\overline{\Omega}_{T_1})}\big),
\end{equation*}
with $G^*:=\max_{u\in[0,K]} |G'(u)|$.
Therefore, if $T^*_2\in(0,T)$ is such that
\begin{equation*}
T^*_2(d_1+d_2+a_{12})<\frac{1}{4}\mbox{ and }T^*_2G^*<\frac{1}{2},
\end{equation*}
 then for any $T_1\in (0,T^*]$ with $T^*:=\min\{T^*_1,T^*_2\}$,
\begin{equation*}
\|U\|_{C(\overline{\Omega}_{T_1})}+\|V\|_{C(\overline{\Omega}_{T_1})}\leq  \frac{1}{4}\big(\|\phi_1-\phi_2\|_{C(\overline{\Omega}_{T_1})}+\|\psi_1-\psi_2\|_{C(\overline{\Omega}_{T_1})} \big)+\frac 12 \|U\|_{C(\overline{\Omega}_{T_1})},
\end{equation*}
which implies
\[
\|U\|_{C(\overline{\Omega}_{T_1})}+\|V\|_{C(\overline{\Omega}_{T_1})}\leq  \frac{1}{2}\big(\|\phi_1-\phi_2\|_{C(\overline{\Omega}_{T_1})}+\|\psi_1-\psi_2\|_{C(\overline{\Omega}_{T_1})} \big).
\]

Hence,  $\mathcal{F}$ is a contraction mapping in $ X^{T_1}_{K}$ for $T_1\in(0,T^*]$. By applying the contraction mapping theorem, $\mathcal{F}$ admits a unique fixed point $(u,v)$ in $ X^{T_1}_{K}$, which implies that $(u,v)$ is a solution of  \eqref{model3} in $  X^{T_1}_{K}$.

To conclude that  \eqref{model3} has a unique solution for $t\in (0, T_1]$ with
$T_1\in(0,T^*]$, it suffices to show that any solution $(u,v)$ of \eqref{model3} defined for $t\in (0, T_1]$  belongs to $ X^{T_1}_{K}$. In fact, we claim the following more accurate estimates
\begin{equation}\label{estimate_uvlocal}
0\leq u(x,t)\leq K_1,~~0\leq v(x,t)\leq K_2 \mbox{ for } t\in[0,T_1] \mbox{ and }x\in[g(t),h(t)].\\
\end{equation}

Indeed, it follows from \eqref{K_1K_2} that
 \begin{equation*}
 \begin{split}
&d_1\int_{g(t)}^{h(t)}J_1(x-y)K_1dy-d_1K_1-a_{11}K_1+a_{12}\int_{g(t)}^{h(t)}K(x-y)K_2dy\\
\leq&\ d_1K_1-d_1K_1-a_{11}K_1+a_{12}K_2\leq0,\\
&d_2\int_{g(t)}^{h(t)}J_2(x-y)K_2dy-d_2K_2-a_{22}K_2+G(K_1)\leq0.
 \end{split}
 \end{equation*}
Since $(K_1,K_2)\succeq (\|u_0\|_{\infty},\|v_0\|_{\infty})$ and $(K_1,K_2)\ggs (u(x,t),v(x,t)))$ for $t\in(0,T_1]$ and $x\in\{g(t),h(t)\}$,
we may apply Lemma \ref{maximalprinciple} to $K_1-u$ and $K_2-v$ to obtain
 $u(x,t)\leq K_1$ and  $v(x,t)\leq K_2$ for $t\in[0,T_1]$
and $x\in[g(t),h(t)]$.
Since it is clear that $(u(x,t),v(x,t))\succeq(0,0)$, \eqref{estimate_uvlocal} is proved.

Thus, we have proved that for any $T_1\in(0,T^*]$,
\eqref{model3} admits a unique solution defined for $t\in[0,T_1]$.

\textbf{Step 3}: Extension of the solution to $t\in[0,T]$.

By the choice of $T^*$, we can replace the initial time $t=0$ with $t=T_1$ $(T_1\in(0,T^*])$ and then repeat Step 2.
Therefore, the solution of \eqref{model3} can be uniquely extended to $t\in[0,T_2]$ with $T_2\in(0,\min\{2T^*,T\}]$
and also satisfy  the estimates in \eqref{estimate_uvlocal} for $(x,t)\in\overline{\Omega}_{T_2}$. It is now clear that after repeating this process finitely many times,
the solution
of \eqref{model3} is uniquely extended to $t\in[0,T]$ and the estimates in \eqref{estimate_uvlocal} hold for $(x,t)\in\overline{\Omega}_{T}$. Since $u_0(x),v_0(x)>0$
for $x\in(-h_0,h_0)$ and $G(u)$ can be written as $G'(\xi^0(x,t))u$ with $\xi^0(x,t)\in[0,u]$ and $G'(\xi^0(x,t))\geq0$, by applying Lemma \ref{maximalprinciple}, we have $u(x,t),v(x,t)>0$ for $(x,t)\in\Omega_{T}$. Hence
\eqref{estimate_uv} holds.
\end{proof}

\begin{proof}[{\bf Proof of Theorem \ref{uniexist}}]
From Lemma \ref{solution_uv}, we know that for any $T>0$ and $(g,h)\in G^T\times H^T$, problem \eqref{model3} admits a unique solution
$(u,v)=(u^{(g,h)}, v^{(g,h)})\in X^T$ satisfying \eqref{estimate_uv}. Using such a pair $(u,v)$,
we can define a mapping $\tilde{\mathcal{F}}(g,h)=(\tilde{g},\tilde{h})$  by
\begin{equation}\label{definition_tildegh}
\begin{cases}
&\tilde{g}(t):=\displaystyle-h_0-\mu\int_0^t\int_{g(s)}^{h(s)}\int_{-\infty}^{g(s)}\Big[J_1(x-y)u(x,s)+\rho J_2(x-y)v(x,s)\Big]dydxds,\\
&\tilde{h}(t):=\displaystyle h_0+\mu\int_0^t\int_{g(s)}^{h(s)}\int^{+\infty}_{h(s)}\Big[J_1(x-y)u(x,s)+\rho J_2(x-y)v(x,s)\Big]dydxds.\\
\end{cases}
\end{equation}

Our plan to complete the proof is  as follows. First,
we will define a closed subset $\Sigma_T$ of $G^T\times H^T$ by making an extra assumption on $(g,h)$. Then,
we prove for all sufficiently small $T>0$, $\tilde{\mathcal{F}}$ maps $\Sigma_T$ to itself and is a contraction mapping. By the contraction mapping theorem, $\tilde{\mathcal{F}}$ has
a unique fixed point in $\Sigma_T$, which implies that \eqref{model*} for $0<t\leq T$ admits a unique solution in the set $\Sigma_T$. To conclude that this is the unique  solution of \eqref{model*} defined for $0<t\leq T$, we will show that for any solution $(u,v,g,h)$ of \eqref{model*} defined for $0<t\leq T$, we have
 $(g,h)\in\Sigma^T$. Finally, we will uniquely extend this solution from $[0,T]$ to $[0,\infty)$.

We carry out the above plan in 4 steps.

\textbf{Step 1}: We define $\Sigma_{T}$ and show that $\tilde{\mathcal{F}}$ maps $\Sigma_T$ to itself if $T>0$ is sufficiently small.

 By \eqref{definition_tildegh},
$\tilde{g},\tilde{h} \in C^1([0,T])$, and for $0<t\leq T$,
\begin{equation}\label{derivative_tildegh}
\begin{cases}
\tilde{g}'(t)=\displaystyle-\mu \int_{g(t)}^{h(t)}\int_{-\infty}^{g(t)}\Big[J_1(x-y)u(x,t)dydx+\rho _2(x-y)v(x,t)\Big]dydx,\\
\tilde{h}'(t)=\displaystyle\mu \int_{g(t)}^{h(t)}\int^{\infty}_{h(t)}\Big[J_1(x-y)u(x,t)dydx+\rho J_2(x-y)v(x,t)\Big]dydx.\\
\end{cases}
\end{equation}
Clearly, $\tilde{g}'(t)< 0$ and $\tilde{h}'(t)>0$ in $(0,T]$.
Hence $(\tilde{g},\tilde{h})$ belongs to $G^T\times H^T$.
Moreover,
\begin{equation*}\label{Auxiliary1}
\begin{cases}
u_t\geq-(d_1+a_{11})u,&0<t\leq T,~~x\in(g(t),h(t)),\\
v_t\geq-(d_2+a_{22})v,&0<t\leq T,~~x\in(g(t),h(t)),\\
u(h(t),t)=u(g(t),t)=v(h(t),t)=v(g(t),t)=0,&0<t\leq T,\\
u(x,0)=u_0(x),~~v(x,0)=v_0(x),&x\in[-h_0,h_0].
\end{cases}
\end{equation*}
Therefore,
\begin{equation}\label{estimate_A1}
\begin{split}
&u(x,t)\geq e^{-(d_1+a_{11})t}u_0(x)\geq e^{-(d_1+a_{11})T}u_0(x)\mbox{ for }x\in[-h_0,h_0],t\in(0,T],\\
&v(x,t)\geq e^{-(d_2+a_{22})t}v_0(x)\geq e^{-(d_2+a_{22})T}v_0(x)\mbox{ for }x\in[-h_0,h_0],t\in(0,T].
\end{split}
\end{equation}

By \eqref{derivative_tildegh}, we see that, for $t\in[0,T]$,
\begin{equation*}
\begin{split}
\big[\tilde{h}(t)-\tilde{g}(t)\big]'=&\ \mu\Big[\int_{g(t)}^{h(t)}\int^{+\infty}_{h(t)}J_1(x-y)u(x,t)dydx
+\int_{g(t)}^{h(t)}\int_{-\infty}^{g(t)}J_1(x-y)u(x,t)dydx\Big]\\
&+\mu\rho\Big[\int_{g(t)}^{h(t)}\int^{+\infty}_{h(t)}J_2(x-y)v(x,t)dydx
+\int_{g(t)}^{h(t)}\int_{-\infty}^{g(t)}J_2(x-y)v(x,t)dydx\Big]\\
\leq&\ \mu\Big[\int_{g(t)}^{h(t)}\int^{+\infty}_{-\infty}J_1(x-y)u(x,t)dydx+\rho\int_{g(t)}^{h(t)}
\int^{+\infty}_{-\infty}J_2(x-y)v(x,t)dydx\Big]\\
\leq&\ \mu \big(K_1+\rho K_2\big)[h(t)-g(t)].
\end{split}
\end{equation*}

It follows from $\textbf{(J)}$ that there exist constants $\varepsilon_0\in(0,h_0/4)$ and $\theta_0>0$ such that, for $i=1,2$,
\begin{equation}\label{estimate_J}
J_i(x-y)\geq\theta_0 \mbox{ if }|x-y|\leq \varepsilon_0.
\end{equation}
Now if we assume $(g,h)$ additionally satisfies
\begin{equation}\label{assumo-extra}
h(T)-g(T)\leq 2h_0+\frac{\varepsilon_0}{4},
\end{equation}
then for $t\in[0,T]$,
\begin{equation*}
\tilde{h}(t)-\tilde{g}(t)\leq 2h_0+\mu \big(K_1+\rho K_2\big) T( 2h_0+\frac{\varepsilon_0}{4})\leq  2h_0+\frac{\varepsilon_0}{4}
\end{equation*}
if $T=T(\mu,\rho,K_1,K_2,h_0,\varepsilon_0)$ is sufficiently small. Denote this $T(\mu,\rho,K_1,K_2,h_0,\varepsilon_0)$ as $T_0$. For
any $T\in(0,T_0]$, it follows from \eqref{assumo-extra} that
\begin{equation*}
h(t)\in[h_0,h_0+\frac{\varepsilon_0}{4}],~~g(t)\in[-h_0-\frac{\varepsilon_0}{4},-h_0].
\end{equation*}

By \eqref{estimate_A1} and \eqref{estimate_J}, for any $t\in(0,T]$ with $T\in(0,T_0]$,
\begin{equation*}
\begin{split}
&\int_{g(t)}^{h(t)}\int^{+\infty}_{h(t)}J_1(x-y)u(x,t)dydx+\rho\int_{g(t)}^{h(t)}\int^{+\infty}_{h(t)}J_2(x-y)v(x,t)dydx\\
\geq&\int_{h(t)-\frac{\varepsilon_0}{2}}^{h(t)}\int^{h(t)+\frac{\varepsilon_0}{2}}_{h(t)}J_1(x-y)u(x,t)dydx+\rho
\int_{h(t)-\frac{\varepsilon_0}{2}}^{h(t)}\int^{h(t)+\frac{\varepsilon_0}{2}}_{h(t)}J_2(x-y)v(x,t)dydx\\
\geq&e^{-(d_1+a_{11})T}\int_{h_0-\frac{\varepsilon_0}{4}}^{h_0}\int^{h_0+\frac{\varepsilon_0}{2}}_{h_0+\frac{\varepsilon_0}{4}}J_1(x-y)dyu_0(x)dx
+\rho e^{-(d_2+a_{22})T}\int_{h_0-\frac{\varepsilon_0}{4}}^{h_0}\int^{h_0+\frac{\varepsilon_0}{2}}_{h_0+\frac{\varepsilon_0}{4}}J_2(x-y)dyv_0(x)dx\\
\geq&\frac{\varepsilon_0}{4}\theta_0\Big(e^{-(d_1+a_{11})T_0}\int_{h_0-\frac{\varepsilon_0}{4}}^{h_0}u_0(x)dx
+\rho e^{-(d_2+a_{22})T_0}\int_{h_0-\frac{\varepsilon_0}{4}}^{h_0}v_0(x)dx\Big)=:\delta_1>0.
\end{split}
\end{equation*}
Note that $\delta_1$ depends on $(\theta_0,u_0,v_0,\rho)$ but is independent of $T\in(0,T_0]$.
It follows that, for any $t\in(0,T]$ with $T\in(0,T_0]$,
\begin{equation}\label{estimate_h}
\tilde{h}'(t)\geq \mu\delta_1.
\end{equation}

Similarly, we can obtain that, for any $t\in(0,T]$ with $T\in(0,T_0]$,
\begin{equation}\label{estimate_g}
\tilde{g}'(t)\leq -\mu\delta_2,
\end{equation}
where
\begin{equation*}
\delta_2:=\frac{\varepsilon_0}{4}\theta_0\Big(e^{-(d_1+a_{11})T_0}\int^{-h_0+\frac{\varepsilon_0}{4}}_{-h_0}u_0(x)dx
+\rho e^{-(d_2+a_{22})T_0}\int^{-h_0+\frac{\varepsilon_0}{4}}_{-h_0}v_0(x)dx\Big)>0.
\end{equation*}

Now we are ready to define, for any $T\in(0,T_0]$,
\begin{equation*}
\begin{split}
\Sigma_T:=&\Big\{(g,h)\in  G^T\times H^T:~\sup_{0\leq t_1<t_2\leq T}\frac{g(t_2)-g(t_1)}{t_2-t_1}\leq -\mu\delta_2,\\
&\ \ \ \ \ \ \inf_{0\leq t_1<t_2\leq T}\frac{h(t_2)-h(t_1)}{t_2-t_1}\geq \mu\delta_1,\ h(t)-g(t)\leq 2h_0+\frac{\varepsilon_0}{4} \mbox{ for }t\in[0,T]\Big\}.
\end{split}
\end{equation*}
Based on the above discussion, it is easily seen that $\tilde{\mathcal{F}}$ maps $\Sigma_T$ to itself for every $T\in (0, T_0]$.

\textbf{Step 2}:
We show that $\tilde{\mathcal{F}}$ is a contraction mapping on $\Sigma_T$ for all small $T>0$.

The proof of this step is long and tedious, but it is only a simple modification of the corresponding step in the proof of Theorem 2.1 in \cite{DWZ}, so we leave the details to the interested reader.

\textbf{Step 3}: Local existence and uniqueness.

By Step 2 and the contraction mapping theorem, we know that for all small $T>0$, say $T\in (0, T^*]$, $\tilde{\mathcal{F}}$ has a fixed point $(g,h)\in \Sigma_T$, and
$(u^{(g,h)},v^{(g,h)},g,h)$ is a solution of \eqref{model*} with $(g,h)\in\Sigma_{T}$. If we can prove that
$(g,h)\in\Sigma_{T}$ holds for any solution $(u,v,g,h) $ of \eqref{model*} defined for $t\in(0,T]$, then $(g,h)$ must coincide with the unique
fixed point of $\tilde{\mathcal{F}}$ in $\Sigma_{T}$, and hence \eqref{model*} admits a unique solution defined for $t\in(0,T]$.

Now let $(u,v,g,h)$ be an arbitrary solution of \eqref{model*} defined for $t\in(0,T]$. Then by Lemma 5.1, necessarily $(u,v)=(u^{(g,h)}, v^{(g,h)})$. Moreover, from
\begin{equation*}
\begin{cases}
g'(t)=\displaystyle-\mu \int_{g(t)}^{h(t)}\int_{-\infty}^{g(t)}J_1(x-y)u(x,t)dydx-\mu\rho \int_{g(t)}^{h(t)}\int_{-\infty}^{g(t)}J_2(x-y)v(x,t)dydx,\\
h'(t)=\displaystyle\mu \int_{g(t)}^{h(t)}\int^{+\infty}_{h(t)}J_1(x-y)u(x,t)dydx+\mu\rho \int_{g(t)}^{h(t)}\int^{+\infty}_{h(t)}J_2(x-y)v(x,t)dydx
\end{cases}
\end{equation*}
we obtain
\begin{equation*}
\begin{split}
h'(t)-g'(t)=&\ \mu \int_{g(t)}^{h(t)}\Big(\int_{-\infty}^{g(t)}+\int^{+\infty}_{h(t)}\Big)J_1(x-y)u(x,t)dydx\\
&+\mu\rho \int_{g(t)}^{h(t)}\Big(\int_{-\infty}^{g(t)}+\int^{+\infty}_{h(t)}\Big)J_2(x-y)v(x,t)dydx,\\
\leq&\ \mu(K_1+\rho K_2)[h(t)-g(t)],
\end{split}
\end{equation*}
which gives
\begin{equation*}\label{hgbound}
h(t)-g(t)\leq 2h_0e^{\mu(K_1+\rho K_2)t},~~t\in(0,T].
\end{equation*}
If we shrink $T^*$ such that $2h_0e^{\mu(K_1+\rho K_2)T^*}\leq 2h_0+\frac{\varepsilon_0}{4}$, then for any $T\in (0, T^*]$,
\begin{equation*}
h(t)-g(t)\leq2h_0+\frac{\varepsilon_0}{4}, ~~t\in(0,T].
\end{equation*}
By \eqref{estimate_h} and \eqref{estimate_g}, we have
\begin{equation*}
h'(t)\geq \mu\delta_1,~~g'(t)\leq \mu\delta_2 \mbox{ for } t\in (0, T].
\end{equation*}
Thus $(g,h)\in\Sigma_{T}$ for all small $T>0$, and we have proved the uniqueness of the solution of \eqref{model3} in $t\in(0,T]$ for small $T>0$.

\textbf{Step 4}: Global existence and uniqueness.

This step can be proved by the same argument used in the proof of Theorem 2.1 in \cite{DWZ}, and we omit the details.
\end{proof}


\end{document}